\documentclass[]{amsart}
\newcommand*\mytitle{Universal Analog Computation}
\newcommand*\mysubtitle{Fra\"{i}ss\'{e} Limits of Dynamical Systems}
\newcommand*\myauthor{Levin Hornischer}
\newcommand*\mykeywords{Analog computation, dynamical systems, category theory, Fra\"{i}ss\'{e} limit,  coalgebra, domain theory.}

\usepackage[utf8]{inputenc}
\usepackage[T1]{fontenc}
\usepackage[english]{babel}
\usepackage{csquotes}

\usepackage{lmodern} 
\usepackage{microtype} 
 
\usepackage[scaled]{helvet} 
\usepackage{eulervm} 
\normalfont
\usepackage{enumitem}
\setlist{noitemsep}

\usepackage{amsmath}
\usepackage{amssymb}
\usepackage{stmaryrd}
\usepackage{amsthm} 
\usepackage{amsfonts}
\usepackage{tikz-cd}

\usepackage{color} 
\usepackage{graphicx} 
\usepackage{tikz}
\usepackage{tabularx}
\usepackage{caption}
\usepackage{subcaption}
\definecolor{darkgreen}{RGB}{33, 97, 11} %html: 0B3B0B
\definecolor{darkblue}{RGB}{11, 56, 97} %html: 0B243B
\definecolor{darkred}{RGB}{97, 11, 11} %html: 610B0B
\definecolor{lightgray}{gray}{0.7}

\usepackage[
	unicode=true,
	pdfborder={0 0 0},
	pdftitle	={\mytitle},
	pdfauthor={\myauthor},
	pdfsubject={},
	pdfkeywords={\mykeywords},
	colorlinks=true,
	linkcolor=darkblue,
	citecolor=darkgreen,
	filecolor=darkred,
	urlcolor=darkred,
]{hyperref}

\usepackage[
	style=numeric,
	sortcites,	
	giveninits=true,
	isbn=false,
	maxbibnames=50,
	maxcitenames=3,
	block=space,
	backref=false,
	backrefstyle=three+,
	date=short,
	backend=biber,
]{biblatex}
\bibliography{lit.bib}

\theoremstyle{plain}
\newtheorem{theorem}{Theorem}[section] 
\newtheorem*{theorem*}{Theorem}
\newtheorem{lemma}[theorem]{Lemma}
\newtheorem{proposition}[theorem]{Proposition}
\newtheorem{corollary}[theorem]{Corollary}
\theoremstyle{definition}
\newtheorem{definition}[theorem]{Definition}
\newtheorem{example}[theorem]{Example}
\newtheorem{remark}[theorem]{Remark}

\newcommand{\Fraisse}{Fra\"{i}ss\'{e}}
\newcommand{\Jonsson}{J\'{o}nsson}

\newcommand{\Gr}{\mathrm{Gr}}
\newcommand{\Cl}{\mathrm{Cl}}
\newcommand{\Clp}{\mathrm{Clp}}
\newcommand{\CoAt}{\mathrm{CoAt}}

\newcommand{\id}{\mathrm{id}}

\newcommand{\F}{\mathsf{F}}
\newcommand{\G}{\mathsf{G}}
\newcommand{\Refl}{\mathsf{R}}

\newcommand{\Alg}{\mathsf{Alg}}
\newcommand{\dynAlg}{\mathsf{dynAlg}}
\newcommand{\bAlgc}{\mathsf{b}\mathsf{Alg}^\mathsf{c}}
\newcommand{\czPolm}{\mathsf{cz}\mathsf{Pol}^\mathsf{m}}

\newcommand{\Inc}{\mathsf{Inc}}
\newcommand{\Path}{\mathsf{Path}}

\newcommand{\allSys}{\mathsf{all}\mathsf{Sys}}
\newcommand{\allSysef}{{\mathsf{all}\mathsf{Sys}^\mathsf{ef}}}
\newcommand{\Sysef}{{\mathsf{Sys}^\mathsf{ef}}}
\newcommand{\Syseffin}{\mathsf{Sys}^\mathsf{ef}_\mathsf{fin}}
\newcommand{\subSysef}{\mathsf{S}^\mathsf{ef}}
\newcommand{\detSysef}{\mathsf{det}\mathsf{Sys}^\mathsf{ef}}
\newcommand{\tSysef}{\mathsf{total}\mathsf{Sys}^\mathsf{ef}}
\newcommand{\DSef}{\mathsf{DS}^\mathsf{ef}}

\newcommand{\SHI}{\mathsf{SHI}}
\newcommand{\SFT}{\mathsf{SFT}}
\newcommand{\SOF}{\mathsf{SOF}}

\begin{document}

\title[\mytitle{}]{\mytitle{}:\\\mysubtitle{}}
\author{Levin Hornischer}
\date{}
\address{Munich Center for Mathematical Philosophy, LMU Munich} 
\email{levin.hornischer@lmu.de}
\keywords{\mykeywords{}}
\thanks{For inspiring discussions and comments, I am grateful to Johan van Benthem, Herbert Jaeger, Michiel van Lambalgen, Hannes Leitgeb, Marcus Pivato, Jan-Willem Romeijn, and the audiences at the Munich Center for Mathematical Philosophy (LMU Munich) and the University of Groningen. Many thanks, in particular, to Marcus Pivato for very helpful comments on an earlier version of the manuscript. Parts of the work were funded by the Bavarian Ministry for Digital Affairs.
}
\subjclass{Primary 68Q09, 37B10; secondary 18A35, 03C30}

\begin{abstract}
Analog computation is an alternative to digital computation, that has recently re-gained prominence, since it includes neural networks and neuromorphic computing. Further important examples are cellular automata and differential analyzers. 
While analog computers offer many advantages, they lack a notion of universality akin to universal digital computers. 
Since analog computers are best formalized as dynamical systems, we review scattered results on universal dynamical systems. We identify four senses of universality and connect to the two main general theories of computation: coalgebra and domain theory.
For nondeterministic systems, we construct a universal system as a Fra\"{i}ss\'{e} limit. It not only is universal in many of the identified senses, it also is unique in additionally being homogeneous. 
For deterministic systems, a universal system cannot exist, but we provide a simple method for constructing subclasses of deterministic systems with a universal and homogeneous system. This way, we introduce sofic proshifts: those systems that are limits of sofic shifts. In fact, their universal and homogeneous system even is a limit of shifts of finite type and has the shadowing property.
\end{abstract}

\maketitle

\section{Introduction}
\label{sec: introduction}

Analog computation is an alternative to digital computation, with a long and rich history~\cite{Ulmann2013,Bournez2021}. 
Examples include \emph{neural networks}, the drivers of the present surge of artificial intelligence. They do not manipulate discrete symbols according to a bespoke program as in digital computation. Rather, they operate on real numbers according to their internal parameters, called weights, which typically are learned from data in a continuous updating process.\footnote{As an abstract model of computation, neural networks are analog. But when they are (physically) implemented, they usually are still simulated on digital hardware. However, developing analog or `neuromorphic' hardware for neural networks is an active field of research~\cite{Kudithipudi2025}.} More generally, \emph{neuromorphic computing} develops computing systems which are inspired by the energy-efficient processing of biological brains. Other famous examples of analog computation are \emph{cellular automata}, which distributively compute global behavior of cells via local update rules; or the \emph{differential analyzer}, which computes by solving differential equations.

A major advantage that secured digital computation its dominant role is that it allows for universal computation: famously, the universal Turing machine can simulate all other digital machines. In analog computation, on the other hand, one typically has to construct one machine for each task.
After all, the fairly uniform architecture of laptops stands in contrast to the many neural network architectures (as mandated, in a sense, by the No-Free-Lunch theorems).
Notwithstanding, in this paper, we investigate when a universal analog computer can exists---and when not.

The plan is as follows. In section~\ref{sec: literature review}, we first recall that analog computers are best conceptualized as dynamical systems. Then we review the scattered literature on universal analog computation and universal dynamical systems, identifying four senses of universality. We also locate our results within existing work.

In section~\ref{sec: categories of dynamical systems}, we motivate and define the notion of time-discrete, space-continuous, possibly nondeterministic dynamical system that we use. Specifically, a \emph{dynamical system} here is a pair $(X,T)$ where $X$ is a zero-dimensional and compact metrizable space and $T : X \rightrightarrows X$ is a closed-valued and upper-hemicontinuos multifunction. We also call $(X,T)$ a \emph{nondeterministic system}, and if $T$ is a function, we speak of a \emph{deterministic system}. We introduce embedding-factor pairs as the morphisms between dynamical systems. Equivalently, a morphism is a \emph{factor} $f : (X,T) \to (Y,S)$, i.e., a surjective continuous function $f : X \to Y$ such that $y'$ is an $S$-image of $y$ iff there are $x$ in $X$ and a $T$-image $x'$ of $x$ such that $f(x) = y$ and $f(x') = y'$ (cf.\ bisimulations). For deterministic systems, this amounts to the usual notion of a factor, and this more general definition is also used for profinite graphs~\cite{Irwin2006}.
With this framework in place, we formally state our main results in section~\ref{sec: statement of the results}.

In section~\ref{sec: coalgebra and domain theory}, we relate our dynamical systems to coalgebras and domain theory, thus contextualizing them within the two main general theories of computation. In particular, we show that our category of dynamical systems is equivalent to a category with a `domain-theoretic spirit': the objects of this category are pairs $(A,\alpha)$ where $A$ is an algebraic lattice (with an additional property) and $\alpha : A \to A$ is a Scott-continuous function (with an additional property), and the morphisms $(\epsilon, \pi) : (A, \alpha) \to (B,\beta)$ are embedding-projections pairs between $A$ and $B$ (with an additional property) such that $\pi \circ \beta \circ \epsilon \leq \alpha$ and $\alpha \circ \pi \leq \pi \circ \beta$.\footnote{Note that these conditions are weaker than what is required for a coalgebra morphism between $(A, \alpha)$ and $(B, \beta)$, where one would require $\alpha \circ \pi = \pi \circ \beta$, which implies $\pi \circ \beta \circ \epsilon =  \alpha \circ \pi \circ \epsilon = \alpha$. In particular, to build universal systems, we cannot apply existence results of a final/terminal coalgebra~\cite{Rutten2000, Adamek2018}.}

Sections~\ref{sec: algebroidal categories of systems}--\ref{sec: universal and homogeneous nondeterministic system} establish our universality result for nondeterministic systems. We use the category-theoretic version of \Fraisse{} limits due to~\textcite{Droste1993}, which works in so-called algebroidal categories. 
In section~\ref{sec: algebroidal categories of systems}, prove a theorem with sufficient conditions for algebroidality and conclude that the category of all dynamical systems is algebroidal. 
In section~\ref{sec: universal and homogeneous nondeterministic system}, we then use the \Fraisse{} limit to prove the following universality result (generalizing one for undirected profinite graphs~\cite{Irwin2006, Camerlo2010, Geschke2022}):
\begin{theorem*}[Theorem~\ref{thm: universal nondeterministic system} below]
There is a nondeterministic system $(U,T)$ that is
\begin{enumerate}
\item
\emph{Universal}: For any system $(Y,S)$, there is a factor $f : (U,T) \to (Y,S)$ (provided that $S$ is nontrivial, i.e., there is $y \in Y$ with $S(y) \neq\emptyset$).
\item
\emph{Homogeneous}: If $(Y, S)$ is a system over a finite set $Y$ and $f , f' : (U,T) \to (Y,S)$ are two factors, then there is an isomorphism $\varphi : (U, T ) \to (U, T )$ such that $f = f' \circ  \varphi$.
\end{enumerate}
Moreover, $(U, T)$ is unique up to isomorphism with these two properties.
\end{theorem*}

In the remainder, we then turn to deterministic systems.
Section~\ref{sec: universality for deterministic systems} analyzes why universality fails for deterministic systems. The reason comes from the so-called category-theoretically finite objects. In an algebroidal category, there should only be countably many of those (up to isomorphism). Indeed, for nondeterministic systems, the category-theoretically finite objects simply are the systems with a finite state space. But for deterministic systems, we show that the category-theoretically finite objects are the shift spaces, and there are uncountably many of those systems.

However, the category of deterministic systems satisfies all other conditions for being algebroidal (i.e., it is semi-algebroidal). Thus, in section~\ref{sec: algebroidal categories of deterministic systems}, we construct algebroidal subcategories. Any countable collection $S$ of shift spaces induces, by taking all limits along countable inverse sequences, an algebroidal category $\DSef(S)$ of deterministic systems.
A natural choice for $S$ is either the class $\SFT$ of shifts of finite type or the class $\SOF$ of sofic shifts. They are among the most studied classes of deterministic systems, and sofic shifts have close ties to automata theory. We call the systems in $\DSef(\SFT)$ (resp., $\DSef(\SOF)$) the \emph{$\omega$-proshifts of finite type} (resp., the \emph{sofic $\omega$-proshifts}). 
In section~\ref{sec: proshifts}, we then show our second main universality result:

\begin{theorem*}[Theorem~\ref{thm: universal and homogeneous proshifts} below]
There is a unique universal and homogeneous $\omega$-proshift of finite type, and it also is the unique universal and homogeneous sofic $\omega$-proshift.
\end{theorem*}

\noindent
For brevity, we call this system the \emph{universal proshift}. While it is not transitive, it does have the desirable shadowing property (which roughly says that an approximate orbit simulated step by step with rounding errors is close to a true orbit). 
We conclude in section~\ref{sec: conclusion} with questions for future research.
An appendix contains outsourced proofs.

\section{Background: Four senses of universal analog computation}
\label{sec: literature review}

With remarkable conceptual clarity, \textcite{Turing1936} formalized digital computation by what is since known as Turing machines. Turing also constructed universal machines.\footnote{Via G\"{o}del encoding, a specific Turing machine can be constructed which simulates, in a clear sense, any other Turing machine. Though, defining what it means, for an arbitrary Turing machine, to be universal is more subtle~\cite{Davis1957}.} However, for analog computation, things are more complicated. There are many nonequivalent models of analog computation~\cite{Bournez2021}. So what would even be a general enough definition of an analog computer? The answer is: they can all be regarded as dynamical systems~\cite{Siegelmann1998, Bournez2016, Bournez2021, Giunti1997, Stepney2012}.
(The converse---i.e., which dynamical systems are analog computers---is a deep philosophical issue beyond scope of this paper, but see, e.g., \cite{Piccinini2021, Gandy1980, Pitowsky1990, Smolensky1988, Sieg2002}.)
At any time, the analog computer is in some state, and when it computes, it updates its state: this is the dynamics. 
This includes---and hence generalizes---digital computation: the state of a digital computer is, roughly, its memory and CPU state, and the dynamics is given by the program that it runs.

So different models of analog computation correspond to different classes of dynamical systems. But what would it mean to say that one analog computer---i.e., one dynamical system $U$---is universal (in its class)? We identify four answers in the scattered literature:
\begin{enumerate}
\item
\label{itm: notion of universality Turing}
\emph{Turing universal}: $U$ can simulate every Turing machine.

\item
\label{itm: notion of universality approximation}
\emph{Approximation universal}: $U$ can approximate any other dynamical system (in the class) up to arbitrary precision.

\item
\label{itm: notion of universality embedding}
\emph{Embedding universal}: any other dynamical system (in the class) can be embedded into $U$.

\item
\label{itm: notion of universality factor}
\emph{Factor universal}: any other dynamical system (in the class) is a \emph{factor} of $U$, i.e., an image under a dynamics-preserving function with domain $U$.
\end{enumerate}

Concerning~\ref{itm: notion of universality Turing}, von~Neumann showed that there are Turing universal cellular automata (see~\cite{Ollinger2012} for discussion) and this also has been done for other models~\cite{Orponen1996}.
Moreover, recurrent neural networks with rational weights can simulate arbitrary Turing machines~\cite{Siegelmann1994}. For continuous-time Hopfield networks, see~\cite{Sima2003a}, and for an overview, see~\cite{Sima2003}.
Furthermore, \textcite{Tao2020} provides a `potential well' dynamical system such that ``the halting of any Turing machine with a given input is equivalent to a certain bounded trajectory in this system entering a certain open set''~(p.~1).
Yet further dynamical systems that are Turing universal are provided by~\cite{Koiran1994, Graca2005, Cardona2021}; for an overview of `fluid computers', see~\cite{Cardona2024}.

Concerning~\ref{itm: notion of universality approximation}, the differential analyzer---a revolutionary analog computer first built in 1931---was mathematically analyzed to precisely compute the solutions of algebraic differential equations~\cite{Shannon1941,PourEl1974}. Later, \cite{Rubel1981} provides ``an analogue, for analog computers, of the Universal Turing Machine''~(p.~345), in the form of a concrete algebraic differential equation that is approximation universal in the sense that, for every continuous function $f$ on the real numbers and precision $\epsilon$, there is a smooth solution $y$ of that differential equation such that $y$ is $\epsilon$-close to $f$. Excitingly, \cite{Bournez2020} recently improved this to polynomial ordinary differential equations (which always have a unique and analytic solution).
For neural networks, the universal approximation theorems~\cite{Cybenko1989, Hornik1991} similarly show that the functions realized by neural networks are dense in the set of all continuous functions (see~\cite{Bournez2025} for a recent result connecting to computability theory).

Concerning~\ref{itm: notion of universality embedding}, a general definition of a dynamical system is as a (Borel) action $\alpha : G \times X \to X$ of a (Polish) group $G$ on a (Borel) space $X$. Here $X$ is the state space of the system, $G$ is time, and $\alpha$ the dynamics, so $\alpha(g,x) = g \cdot x$ is the state the system is in after time $g$ when starting in state $x$. In this setting, there is an embedding universal system, i.e., a $G$-action $\beta$ on a space $U$ such that any other $G$-action $\alpha$ on a space $X$ (Borel) embeds into $\beta$~\cite{Kechris2000}.

Concerning~\ref{itm: notion of universality factor}, viewing dynamical systems as \emph{$G$-flows}, i.e., as continuous actions of a topological group $G$ on a compact topological space $X$, it is well-known that, among the minimal $G$-flows (those where every orbit $\{ g \cdot x : g \in G \}$ is dense in $X$) there is a factor universal one, which is even unique up to isomorphism (mentioned, e.g., in~\cite[1]{Kechris2005}).
Viewing dynamical systems as linear operators $T : X \to X$ on a Banach space $X$, there is a bounded such operator $U$ such that every other such operator with norm~1 is a linear factor of it~\cite{Darji2017}, which recently was generalized to many other categories by~\cite{Balcerzak2023} (also see~\cite{Kubis2023}).
Coming back to cellular automata, \cite{Hochmanms}~asks if there are factor-universal cellular automata (noting the existence of universal effective $\mathbb{Z}$-dynamical systems~\cite{Hochman2009}).

Concerning further related work, for a structurally similar discussion of universality of spin models, see~\cite{DelasCuevas2016}, and for a category-theoretic framework for universality that includes spin models and Turing machines, see~\cite{Gonda2024}.
Moreover, our main tool, namely the \Fraisse{}-limit, has, beyond extensive use in model theory~\cite[ch.~7]{Hodges1993}, also applications in the following fields: 
domain theory~\cite{Droste1993, Droste2007},
set theory~\cite{Kubis2014},
topos theory~\cite{Caramello2014}, 
graph theory~\cite{Irwin2006, Camerlo2010, Geschke2022}, and
dynamical systems theory~\cite{Bernardes2012, Kwiatkowska2012}---and see~\cite{Macpherson2011} for a more detailed overview. 
A stunning connection between these fields is provided by the KPT~correspondence~\cite{Kechris2005}: If $A$ is the \Fraisse{}-limit of a \Fraisse{}-order class $\mathcal{K}$, then $\mathcal{K}$ has a combinatorial property known as the Ramsey property iff the automorphism group of $A$ is extremely amenable, i.e., every continuous action of this group on a compact Hausdorff space has a fixed point. This can also be used to construct universal flows for automorphism groups of \Fraisse{} limits. For a continuation of this line of research, see, e.g.,~\cite{Lupini2018, Masulovic2021, Bartos2024, NguyenVanThe2014, Bartosova2013, Kechris2014, Akin2016}.

Moreover, related to the question of universality is the question of genericity: Given a space $X$, we can ask if there is a \emph{generic} system on $X$, i.e., a dynamics $T : X \to X$ such that the conjugacy class of $(X,T)$, i.e., the set of all systems $(X,S)$ that are isomorphic to $(X,T)$, forms a comeager subset of the space of all dynamics on $X$, appropriately topologized. A result of~\cite{Kechris2006} states that this is true when $X$ is the Cantor space and `dynamics' means homeomorphism on $X$ (\cite{Glasner2001} previously established density). This is further researched by~\cite{Hochman2008, Hochman2012, Kwiatkowska2012, Doucha2024, Doucha2025}.

Building on this rich body of work, our application of \Fraisse{}-limits is, to the best of our knowledge, new in constructing universal dynamical systems---motivated by the search for universal analog computation---in a way that unifies the case of nondeterministic systems and deterministic systems. Methodologically, we connect to the tools of theoretical computer science by using a category-theoretic version of the \Fraisse{}-limits and exploring links to domain theory and coalgebra.

The closest existing results in the nondeterministic case is the work cited above on the projectively universal homogeneous (undirected) graph~\cite{Irwin2006, Camerlo2010, Geschke2022} (as an instance of the category-theoretic \Fraisse{} limit~\cite{Droste1993, Kubis2014}). Indeed, our nondeterministic systems can be seen as directed graphs. Although ``some ideas [for undirected graphs] can be easily adapted to the more general case of directed graphs''~\cite[3]{Geschke2022}, we still include proofs, since the directed setting is conceptually central to dynamical systems and we reuse some results for the  deterministic systems.

The related work in the deterministic case is the literature cited above on KPT ~correspondence and genericity. But the focus is less on a universal system for actions of $\mathbb{N}$, but more on actions of automorphism groups of \Fraisse{} limits and understanding their universal minimal flow and their genericity. Moreover, we add category-theoretic tools to this investigation. For example, we corroborate the important role of shifts in dynamical systems theory by showing that they are the category-theoretically finite objects, and we define the category of, e.g., sofic $\omega$-proshifts as an instance of an important category-theoretic construction (essentially that of a pro-category).
Thus, we add to a growing body of literature of using category theory for dynamical systems~\cite{Lawvere1986, Niefield1996, Behrisch2017, Schultz2020}.

\section{Categories of dynamical systems}
\label{sec: categories of dynamical systems}

We define the categories of dynamical systems that we will work with: subsection~\ref{ssec: objects dynamical systems} defines the objects and subsection~\ref{ssec: morphisms embedding factor pairs} the morphisms.

\subsection{Objects: Dynamical systems}
\label{ssec: objects dynamical systems}

The dynamical systems that we consider (to represent analog computers) are time-discrete, state-continuous, and possibly nondeterministic dynamical systems. We first state the formal definition, and then we motivate this choice and provide examples.

Nondeterminism means that a state might not have a unique next state, but that there is a \emph{set} of potential next states. Hence the dynamics is a multifunction (i.e., set-valued function). So we first recall the necessary definitions from set-valued analysis~\cite[ch.~17]{Aliprantis2006}. 

A \emph{multifunction} $F : X \rightrightarrows Y$ is a function that maps each $x \in X$ to a subset $F(x) \subseteq Y$. We define
\begin{align*}
x \xrightarrow{F} y
:\Leftrightarrow
y \in F(x).
\end{align*}
We call $F$ \emph{total} if each $F(x)$ is nonempty. If $A \subseteq X$, the \emph{image} of $A$ under $F$ is 
$F[A] := \bigcup_{x \in A} F(x)$.
Further, we call $F$ \emph{partition-injective} if $\{ F(x) : x \in X \}$ is a partition of $Y$ (i.e., for all $x \neq x'$ in $X$, we have $F(x) \cap F(x') = \emptyset$ and for all $y \in Y$, there is $x \in X$ with $y \in F(x)$). We identify a function $f : X \to Y$ with the multifunction $X \rightrightarrows Y$ that maps $x$ to $\{ f(x) \}$, and if a multifunction $F: X \rightrightarrows Y$ has only singletons as values, we identify it with the function $X \to Y$ mapping $x$ to the single element of $F(x)$.

If $X$ and $Y$ are topological spaces, $F$ is \emph{closed-valued} if each $F(x)$ is a closed subset of $Y$. 
The appropriate notion of continuity for us is that $F$ is \emph{upper-hemicontinuous} if, for all $x \in X$ and open $V \subseteq Y$ with $F(x) \subseteq V$, there is an open $U \subseteq X$ with $x \in U$ and, for all $x' \in U$, we have $F(x') \subseteq V$. If $F$ is a function, it is continuous in the usual sense iff it is upper-hemicontinuous.
The \emph{closed graph theorem}~\cite[thm.~17.11]{Aliprantis2006} says: If $F : X \rightrightarrows Y$ is a multifunction between two compact Hausdorff spaces, then $F$ is upper-hemicontinuous and closed-valued iff the \emph{graph} of $F$, i.e., the set $\Gr(F) := \{ (x,y) \in X \times Y : y \in F(x) \}$, is closed. A corollary is that then the image $F[A]$ of a closed set $A \subseteq X$ is closed.

Now we can define the notion of dynamical system that we will use.

\begin{definition}
\label{def: dynamical system}
By a \emph{dynamical system} we mean a pair $(X,T)$ where 
\begin{itemize}
\item
$X$ is a zero-dimensional, second-countable, compact Hausdorff space.\footnote{Recall that a topological space $X$ is zero-dimensional if it has a basis of clopen (i.e., closed and open) sets. And it is second-countable if it has a countable basis.
Note that 
$X$ is a zero-dimensional, second-countable, compact Hausdorff space
iff 
$X$ is a compact and zero-dimensional Polish space
iff 
$X$ is a compact and zero-dimensional metrizable space
iff
$X$ is a second-countable Stone space.
}

\item
$T : X \rightrightarrows X$ is a closed-valued and upper-hemicontinuous multifunction.\footnote{Equivalently, by the closed graph theorem, the graph of $F$ is closed.}
\end{itemize}
We call a dynamical system $(X,T)$: 
\begin{itemize}
\item
\emph{nonempty} if $X$ is nonempty.
\item
\emph{nontrivial} if the graph of $T$ is nonempty (i.e., there are $x,y \in X$ with $x \xrightarrow{T} y$),
\item
\emph{total} if $T$ is total (i.e., for every $x \in X$, there is $y \in X$ with $x \xrightarrow{T} y$),
\item
\emph{deterministic} if $T$ is a function (i.e., for every $x \in X$, there is exactly one $y \in X$ with $x \xrightarrow{T} y$).
\end{itemize}
Often, we just say `system' instead of `dynamical system'. We sometimes speak of a `nondeterministic system' to stress that the system under consideration need not be deterministic.
\end{definition}

(Note that a nontrivial dynamical system is nonempty; the converse holds for total systems, and hence also for deterministic systems, but it may fail for non-total systems since states may have no successor.)

To motivate this choice, topological dynamics is a well-established field studying dynamical systems as structures $(X,T)$ where $X$ is a Hausdorff space (typically compact and often also metrizable) and $T : X \to X$ is a continuous function~\cite{Vries2014}. (This was pioneered by Poincar\'{e} studying differential equations also in the absence of analytic solutions.) 
Studying zero-dimensional compact topological systems is an important subfield~\cite{Downarowicz2016}. This is not just because it is the setting of symbolic dynamics~\cite{Lind1995} but also because classical results like the Jewett--Krieger theorem show that the systems in ergodic theory (i.e., dynamics on probability spaces that preserve the measure) have zero-dimensional topological systems as counterparts.

Moreover, as remark~\ref{rmk: zerodimensionalize and compactify} in appendix~\ref{app: proofs from section categories of dynamical systems} makes precise, we can also canonically turn the more general setting of a dynamics on a Polish space---like $\mathbb{R}^n$---into our setting. Specifically, if $X$ is a Polish space and $T:X \to X$ is a continuous function, then there is a---in a sense best possible---way to extend this into a deterministic system $(Y,S)$ in our sense. In particular, there is an injective function $\eta : X \to Y$ that is continuous with respect to a finer Polish topology on $X$ and which is equivariant, i.e., $\eta \circ T = S \circ \eta$.

Here we generalize topological systems by also allowing nondeterministic dynamics. For simplicity, we still work with discrete time, i.e., $T$ describes the \emph{next} states. Future work should generalize our results to continuous time (as used in continuous time analog computation) and to group or monoid actions more generally.

To illustrate our definition, we consider, in the remainder of this section, some examples that highlight our computational perspective on dynamical systems.

\begin{example}[Shifts]
\label{exm: symbolic dynamics}
Symbolic dynamics is an important subfield of dynamical systems theory. It developed from studying physical dynamical systems by discretizing both space and time, but it also found applications in storing and transmitting digital data~\cite{Lind1995}. The objects of study in symbolic dynamics are shifts, which are defined as follows. Let $A$ be a finite set, called \emph{alphabet}. The (one-sided) \emph{full shift} over $A$ is the system $(A^\omega , \sigma)$ where $A^\omega$ carries the product topology of the discrete topology on $A$ and $\sigma$ is the shift map, i.e., $\sigma (x)_n = x_{n+1}$. A \emph{shift} (aka \emph{shift space}) is a closed, shift-invariant subset $X$ of a full shift on some alphabet. Hence, if $X$ is nonempty, $(X,\sigma)$ is in $\detSysef$. 

The two most studied classes of shifts are the following. The \emph{shifts of finite type} are those shifts $X$ described by a finite list of forbidden words, i.e., $X$ is the set of those sequences $x$ over the alphabet $A$ of $X$ that do not contain a word $w$ from a finite set $F$ of words over $A$. Equivalently, shifts of finite type are those shifts whose sequences arise as the sequences of edges of infinite paths in a finite graph.
The \emph{sofic shifts} are those shifts that are factors of shifts of finite type. Equivalently, sofic shifts are those shifts whose sequences arise as the sequences of labels of infinite paths in a finite graph whose edges are labeled. This makes sofic shifts analogous to regular languages in automata theory.
\end{example}

\begin{example}[Cellular automata]
\label{exm: cellular automaton}
A cellular automaton is a deterministic system $(\Sigma^{\mathbb{Z}^n} , G )$ in the above sense. Here $\Sigma$ is the nonempty set of the finitely many states that a single cell can be in, so an element $c \in \Sigma^{\mathbb{Z}^n}$ describes the global configuration of the cellular automaton, i.e., the state of each cell in the $n$-dimensional grid $\mathbb{Z}^n$ of cells. The topology on $\Sigma^{\mathbb{Z}^n}$ is the product topology (where $\Sigma$ has the discrete topology). The function $G$ results from updating a global configuration according to the local updating rule. The local character renders $G$ continuous.
 
For example, in dimension $n = 1$ with binary states $\Sigma = \{ 0,1\}$, the well-known rule ``110'' says that a cell changes from $0$ to $1$ precisely if the cell to its right is $1$ (``life spreads left'), and a cell changes from $1$ to $0$ precisely if both the cell to the left and to the right is $1$ (``overpopulation''). \textcite{Cook2004} shows that this simple rule yields a Turing universal cellular automaton.
\end{example}

\begin{example}[Differential equations]
\label{exm: differential equation}
Differential analyzers compute by solving (ordinary) differential equations---coming, e.g., from engineering~\cite{Bournez2021, Ulmann2013}. As usual~\cite[sec.~6.2]{Teschl2012}, such solutions are, if globally defined, dynamical systems $\Phi : \mathbb{R} \times X \to X$ with $X \subseteq \mathbb{R}^n$ open; or in discrete time steps, $T := \Phi (1, \cdot ) : X \to X$. As described by remark~\ref{rmk: zerodimensionalize and compactify}, we can view this dynamics on the Polish space $X$ in the larger, compact and zero-dimensional space $Y$, rendering it a deterministic system in our formal sense.
\end{example}

\begin{example}[Training dynamics]
\label{exm: training dynamics}
One may view the training process of a neural network as computing weights for the neural networks with which it solves its task well. The training dynamics $T$ takes the current weights $w$ of the network and a batch of the training data, and then produces (using backpropagation) updated weights $w'$. Formally, the dynamics is given by stochastic gradient descent, so the choice of batch is not deterministic, hence there are several possible next states $w'$. Thus, the training dynamics $T : W \rightrightarrows W$ is a multifunction on the weight space $W$. Understanding this dynamics is an important problem in the theory of machine learning~\cite{Saxe2014, Jacot2018}.
\end{example}

\begin{example}[Reservoir computing]
\label{exm: reservoir computing}
Reservoir computing is one prominent approach to computing on non-digital hardware. The basic idea is that inputs are mapped into a high-dimensional state space of a nonlinear dynamical system---called the \emph{reservoir}---and only a simple read-out function is trained to map from the high-dimensional state space to outputs~\cite{Tanaka2019}. This was initiated independently by \cite{Jaeger2001} and \cite{Maass2002}, with a rich history since. For the purpose of this example, let us consider the first approach.\footnote{In the case where there is no feedback from the output back to the reservoir.} The reservoir is a recurrent neural network with weight matrix $W$ whose state, at time $t \in \{0,1,2,\ldots\}$, is the vector $x(t)$ describing the activation value of each neuron. Additionally, there are input neurons whose state, at time $t$, is the vector $u(t)$ describing the value of each input neuron. (So the input can change over time, i.e., is sequential.) A matrix $W^\text{in}$ maps input neuron states into states of the reservoir. The state of the reservoir at the next time step is computed by
\begin{align}
\label{eqn: reservoir computing}
x(t + 1) := f \big( W^\text{in} u(t+1) + W x(t) \big),
\end{align}
where $f$ is the function that applies to each neuron its activation function (e.g., a sigmoid). Moreover, there are output neurons whose activations, at time $t$, are given by the vector $y(t) = W^\text{out} x(t)$. Only the weight matrix $W^\text{out}$ is updated during training.
If the input sequence $u$ is not fixed, we can view this as a nondeterministic system $T : X \rightrightarrows X$ on the state space $X$ of the reservoir system, with $T(x)$ being the set of next states realized by some input vector as in~\ref{eqn: reservoir computing}.
\end{example}

\subsection{Morphisms: Embedding-factor pairs}
\label{ssec: morphisms embedding factor pairs}

What should be the morphisms $\varphi : (X,T) \to (Y,S)$ between dynamical systems? For deterministic systems, these are the continuous functions $\varphi : X \to Y$ that are \emph{equivariant}, i.e., the following diagram commutes:
\begin{equation*}
\begin{tikzcd}
X 
\arrow[r,"\varphi"]
\arrow[d,swap,"T"]
&
Y
\arrow[d,"S"]
\\
X 
\arrow[r,"\varphi"]
&
Y
\end{tikzcd}
\end{equation*}
How to generalize this to nondeterministic systems? Since we allow multifunctions for the dynamics, it is natural to also allow them for the morphisms. A natural way to generalize equivariance is to require that any step in the $T$-dynamics can be matched, via $\varphi$, to a step in the $S$-dynamics, and vice versa. For embeddings and factors we require an even closer match.

\begin{figure}[t]
\begin{subfigure}{.24\linewidth}
\centering
\begin{tikzcd}
x'
\arrow[r,"\varphi", dashed] 
&
y'
\\
x
\arrow[r,"\varphi", dashed] 
\arrow[u,"T"] 
&
y
\arrow[u,"S",swap,dashed] 
\end{tikzcd}
\caption{Forth}
\label{fig: equivariance conditions hom forth}
\end{subfigure}
\begin{subfigure}{.24\linewidth}
\centering
\begin{tikzcd}
x'
\arrow[r,"\varphi", dashed] 
&
y'
\\
x
\arrow[r,"\varphi", dashed] 
\arrow[u,"T", dashed] 
&
y
\arrow[u,"S",swap] 
\end{tikzcd}
\caption{Back}
\label{fig: equivariance conditions hom back}
\end{subfigure}
\begin{subfigure}{.24\linewidth}
\centering
\begin{tikzcd}
x'
\arrow[r,"\varphi", dashed] 
&
y'
\\
x
\arrow[r,"\varphi"] 
\arrow[u,"T"] 
&
y
\arrow[u,"S",swap, dashed] 
\end{tikzcd}
\caption{Factor}
\label{fig: equivariance conditions factor}
\end{subfigure}
\begin{subfigure}{.24\linewidth}
\centering
\begin{tikzcd}
x'
\arrow[r,"\varphi", dashed] 
&
y'
\\
x
\arrow[r,"\varphi"] 
\arrow[u,"T",dashed] 
&
y
\arrow[u,"S",swap] 
\end{tikzcd}
\caption{Embedding}
\label{fig: equivariance conditions embedding}
\end{subfigure}
\caption{Nondeterministic equivariance: ways of requiring that the $T$-dynamics is matched, via $\varphi$, to the $S$-dynamics, and vice versa. Solid arrows are assumed relations, dashed arrows are relations required to exist.}
\label{fig: equivariance conditions}
\end{figure}

\begin{definition}
\label{def: system morphism}
Let $(X,T)$ and $(Y,S)$ be two dynamical systems.
A \emph{system morphism} (or just \emph{morphism}) $\varphi : (X,T) \to (Y,S)$ is a closed-valued and upper-hemicontinuous multifunction $\varphi : X \rightrightarrows Y$ such that
\begin{enumerate}
\item
\label{itm: nondet equivariance hom forth}
for all $x,x' \in X$, if $x \xrightarrow{T} x'$, there is $y,y' \in Y$ such that 
$x \xrightarrow{\varphi} y$,
$x' \xrightarrow{\varphi} y'$, and
$y \xrightarrow{S} y'$
(figure~\ref{fig: equivariance conditions hom forth}).

\item
\label{itm: nondet equivariance hom back}
for all $y,y' \in Y$, if $y \xrightarrow{S} y'$, there is $x,x' \in X$ such that 
$x \xrightarrow{\varphi} y$,
$x' \xrightarrow{\varphi} y'$, and
$x \xrightarrow{T} x'$
(figure~\ref{fig: equivariance conditions hom back}).
\end{enumerate}
A morphism $\varphi$ is a \emph{factor} if it is a surjective function with 
\begin{enumerate}[resume]
\item
\label{itm: nondet equivariance factor}
for all $x,x' \in X$ and $y \in Y$, if $x \xrightarrow{\varphi} y$ and $x \xrightarrow{T} x'$, there is $y' \in Y$ such that 
$y \xrightarrow{S} y'$ and $x' \xrightarrow{\varphi} y'$ (figure~\ref{fig: equivariance conditions factor}).
\end{enumerate}
A morphism $\varphi$ is an \emph{embedding} if it is total and partition-injective with
\begin{enumerate}[resume]
\item
\label{itm: nondet equivariance embedding}
for all $x \in X$ and $y, y' \in Y$, if $x \xrightarrow{\varphi} y$ and $y \xrightarrow{S} y'$, there is $x' \in X$ such that $x \xrightarrow{T} x'$ and
$x' \xrightarrow{\varphi} y'$ (figure~\ref{fig: equivariance conditions embedding}).
\end{enumerate}
\end{definition}

In the definition of a factor, clause~\ref{itm: nondet equivariance hom forth} is superfluous; and, for an embedding, clause~\ref{itm: nondet equivariance hom back} is superfluous. 
Also, it is easily shown that $\varphi : (X,T) \to (Y,S)$ is a factor iff $\varphi : X \to Y$ is a continuous surjection such that, for all $y,y' \in Y$, we have
\begin{align*}
y \xrightarrow{S} y' 
\text{ iff } 
\exists x,x' \in X : \varphi(x) = y , \varphi(x') = y' , \text{ and } x \xrightarrow{T} x',
\end{align*}
which is the definition of a factor in~\cite{Irwin2006}. 
Further, this notion of factor indeed generalizes that from deterministic systems: 
If $(X,T)$ and $(Y,S)$ are deterministic systems, then $\varphi : (X,T) \to (Y,S)$ is a factor iff $\varphi : X \to Y$ is a continuous and equivariant surjection.
Moreover, our morphisms are closely related to the well-known notion of a \emph{bisimulation}, which is a relation with both property~\ref{itm: nondet equivariance factor} and~\ref{itm: nondet equivariance embedding}.

Composition of morphisms is the usual composition of relations.
(The proof is in appendix~\ref{app: proofs from section categories of dynamical systems}.) 

\begin{proposition}
\label{prop: composition of system morphisms}
If $\varphi : (X,T) \to (Y,S)$ and $\psi : (Y,S) \to (Z,R)$ are morphisms (resp.\ factors or embeddings), then their composition
\begin{align*}
\psi \circ \varphi :(X,T) &\to (Z,R) 
&
x 
&\mapsto 
\psi[\varphi(x)]
\end{align*}
is again a morphism (resp.\ factor or embedding). The composition $\circ$ is associative and its unit is the identity function.
\end{proposition}

Our aim is to find a system $(U,T)$ that is, at least, embedding and factor universal. So for any other system $(Y,S)$, we want both an embedding $e : (Y,S) \to (U,T)$ and a factor $f : (U,T) \to (Y,S)$. Moreover, they should correspond to each other: 
\begin{enumerate}
\item
if one first embeds up and then factors down again, one arrives at the state one started with, and 
\item
if one first factors down and then embeds up, one might obtain several states (since the universal system is richer) but the state one started with should be among them. 
\end{enumerate}
We will then speak of an embedding-factor pair, in analogy to the embedding-projection pairs in domain theory~\cite{Abramsky1994}.

\begin{definition}
\label{def: embedding-factor pair}
An \emph{embedding-factor pair} (\emph{ef-pair} for short) from system $(X,T)$ to system $(Y,S)$ is a pair $(e,f)$ of an embedding $e :  (X,T) \to (Y,S)$ and a factor $f :  (Y,S) \to (X,T)$ such that 
\begin{enumerate}
\item
\label{def: embedding-factor pair 1 embed then factor is identity}
for all $x \in X$, $f \circ e (x) =  \{ x \}$, and 
\item
\label{def: embedding-factor pair 2 factor then embed is subset}
for all $y \in Y$, $y \in e \circ f (y)$.
\end{enumerate}
If $(e,f) :  (X,T) \to (Y,S)$ and $(e',f') :  (Y,S) \to (Z,R)$ are ef-pairs, their composition is
$(e',f') \circ (e,f) := (e' \circ e , f \circ f')  :  (X,T) \to (Z,R)$, which again is an ef-pair. 
\end{definition}

\begin{remark}[Simulation and abstraction]
\label{rmk: simulation and abstraction}
The intuition behind an ef-pair $(e,f) : (X,T) \to (Y,S)$ is that the factor $f : Y \to X$ \emph{abstracts} the `detailed' states $y$ in $Y$ into `coarser' states $x$ in $X$, and the embedding $e : X \to Y$ \emph{simulates} states $x$ in $X$ by states $e(x)$ in $Y$ (cf.\ \cite{Winskel1995}).  
Several states $y$ (namely those $f^{-1} (x)$) can be identified under the abstraction into a single state $x$. And several states $y$ (namely those in $e(x)$) can simulate the state $x$: if there is a dynamic step $x \xrightarrow{T} x'$ in $X$, there are states $y$ and $y'$ simulating $x$ and $x'$, respectively, and $y \xrightarrow{S} y'$.
\end{remark}

In an ef-pair, each of the two elements determines the other. This is analogous to embedding-projection pairs in domain theory forming an order adjunction. Thus, by considering ef-pairs, we, in sense, `just' consider factors (or `just' embeddings). However, like for embedding-projection pairs in domain theory, we will see that it is conceptually helpful to make both directions explicit (see remark~\ref{rmk: limit vs colimit} below). (The proof is in appendix~\ref{app: proofs from section categories of dynamical systems}.)  

\begin{proposition}
\label{prop: ef pairs determine each other}
Let $(X,T)$ and $(Y,S)$ be two systems. 
\begin{enumerate}
\item
\label{prop: ef pairs determine each other 1}
If $f :  (Y,S) \to (X,T)$ is a factor, there is a unique embedding $\underline{f} : (X,T) \to (Y,S)$ such that $(\underline{f},f)$ is an ef-pair. It is given by $\underline{f}(x) := f^{-1} (x)$.

\item
\label{prop: ef pairs determine each other 2}
If $e :  (X,T) \to (Y,S)$ is an embedding, there is a unique factor $\overline{e} : (Y,S) \to (X,T)$ such that $(e,\overline{e})$ is an ef-pair. It is given by $\overline{e} (y) =$~the $x \in X$ with $y \in e(x)$.
\end{enumerate}
\end{proposition}

Now we can define the categories of dynamical systems that we need here. We will focus on nontrivial systems. This is not just because trivial systems are uninteresting from a dynamical point of view. It is also because, by definition, no nontrivial system can factor onto a trivial one. So if we want to find a system that is both embedding and factor universal, we cannot, for that simple reason, include trivial systems.  

\begin{definition}
\label{def: category of systems}
Let $\Sysef$ be the category whose objects are nontrivial dynamical systems $(X,T)$ and whose morphisms are ef-pairs $(e,f) : (X,T) \to (Y,S)$. 
Should we need to talk about all systems, we write $\allSysef$ for the supercategory consisting of all dynamical systems with ef-pairs. 
We are also interested in the full subcategories $\tSysef$ and $\detSysef$ of $\Sysef$ consisting of total and deterministic systems, respectively (which hence are nonempty).
\end{definition}

We could also generalize away from ef-pairs and define, e.g., the category $\allSys$ of dynamical systems with system morphisms, but we do not need that here.
Reassuringly, system isomorphism $(e,f) : (X,T) \to (Y,S)$ in the category-theoretic sense coincides with the usual sense, i.e., $e : X \to Y$ is a homeomorphism with inverse $f$ such that $e \circ T = S \circ e$ (proposition~\ref{prop: isomorphism in Sysef} in appendix~\ref{app: proofs from section categories of dynamical systems}).
The category $\detSysef$ is a common setting for topological dynamics in dimension zero: the objects are pairs $(X,T)$ with $X$ a nonempty, compact, and zero-dimensional Polish space and $T : X \to X$ a continuous function, and the morphisms are factors (ignoring the corresponding embedding), i.e., continuous and equivariant surjections.

\begin{remark}[Limit--colimit coincidence]
\label{rmk: limit vs colimit}
Because of proposition~\ref{prop: ef pairs determine each other}, we could, instead of ef-pairs, equivalently work with only embeddings or only factors. The latter is more natural from a dynamical systems point of view, but the former is more common for \Fraisse{} limits. With ef-pairs, we can capture both perspectives. 
In particular, in $\allSysef$ and any of its full subcategories, a \emph{colimit} via ef-pairs (or just the embeddings in the pairs) is the same thing as a \emph{limit} of the factors in the pairs. This is, again, analogous to the \emph{limit--colimit coincidence} in domain theory~\parencite[sec.~3.3.2]{Abramsky1994}. Thus, depending on one's preferred perspective, one can switch back and forth between the embedding/colimit or the factor/limit terminology.
\end{remark}

\section{Main results and proof technique}
\label{sec: statement of the results}

With this framework of dynamical systems in place, we can now state our main results (section~\ref{ssec: statement of the main results}). Then we describe the general proof technique, namely, \Fraisse{} limits (section~\ref{ssec: proof technique Fraisse limits}). 

\subsection{Statement of the main results}
\label{ssec: statement of the main results}

To connect to existing theories of computation, section~\ref{sec: coalgebra and domain theory} establishes a categorical equivalence between our category $\Sysef$ of systems and what we call the category $\dynAlg$ of dynamic algebraic lattices. 
But our two main results on the existence of universal systems are the following.

To state these results, recall that we are looking for a universal system $(U,T)$, i.e., for any system $(Y,S)$, we need an ef-pair $(e,f) : (Y,S) \to (U,T)$. Moreover, ideally the system $(U,T)$ also is `homogeneous' in the intuitive sense that `every part of it looks the same', so it does not have any `idiosyncrasies'. To formalize this, recall that a topological space $X$ is homogeneous if, for any two points $x,y \in X$, there is a homeomorphism $h : X \to X$ such that $h(x) = y$.\footnote{Equivalently, the group $G$ of hoemeomorphisms of $X$ acts transitively on $X$.} Since points of $X$ are essentially just continuous functions from the one-point space into $X$, we can rephrase this as follows: If $Y$ is the one-point space and $f,g : Y \to X$ are continuous functions, then there is a homeomorphism $h : X \to X$ such that $h \circ g = f$. By demanding this for more spaces $Y$ (e.g., finite spaces), we get a stronger homogeneity.
Formally, then, we define these terms as follows. 

\begin{definition}
\label{def: universal and homogeneous}
Let $\mathsf{C}$ be a category and let $\mathsf{D}$ be a full subcategory. 
An object $U$ in $\mathsf{C}$ is \emph{$\mathsf{D}$-universal}, if, for any object $A$ in $\mathsf{D}$, there is a morphism $f : A \to U$ (see figure~\ref{fig: fraisse definitions universal}).\footnote{We follow~\cite[138]{Droste1993} in this terminology, but \textcite[1757]{Kubis2014} instead speaks of a `cofinal' object, to avoid the category-theoretic connotation that a universal object is canonically determined by a universal property.} 
An object $U$ in $\mathsf{C}$ is \emph{$\mathsf{D}$-homogeneous}, if, for any object $A$ in $\mathsf{D}$ and morphisms $f,g : A \to U$, there is an isomorphism $h : U \to U$ such that $h \circ g = f$ (see figure~\ref{fig: fraisse definitions homogeneous}).  
\end{definition}

Regarding our existence result for nondeterministic systems, write $\Syseffin$ for the full subcategory of $\Sysef$ consisting of those systems with a finite state space (i.e., those for which we want homogeneity). Then the result says:

\begin{theorem}[Universal nondeterministic system]
\label{thm: universal nondeterministic system}
There is an up to isomorphism unique system in $\Sysef$ that is $\Sysef$-universal and $\Syseffin$-homogeneous. 
\end{theorem}

Recall from the end of section~\ref{sec: literature review} that this is a generalization of the result for profinite undirected graphs~\cite{Irwin2006, Camerlo2010, Geschke2022} (more details in remark~\ref{rmk: generalization of universal profinite undirected graph} below). Section~\ref{sec: universal and homogeneous nondeterministic system} discusses further properties of the universal nondeterministic system.

Regarding deterministic systems, there cannot be a universal system, as we discuss in section~\ref{sec: universality for deterministic systems}. The shifts are the basic building blocks of deterministic systems (i.e., every deterministic system is a limit of shifts). But the obstacle to universality is that there are uncountably many shifts. So this suggests a strategy to build universal systems for large subclasses of deterministic systems. 
Namely, choose a countable class $S$ of shifts and consider the full subcategory $\DSef(S)$ of $\detSysef$ consisting of those systems that are limits of shifts in $S$. Recall remark~\ref{rmk: limit vs colimit} that a colimit along ef-pairs is the same thing as a limit along factors. So our formal definition is this:

\begin{definition}
\label{def: DSef}
If $S$ is a collection of shifts, let $\DSef(S)$ be the full subcategory of $\detSysef$ consisting of those objects that are colimits of $\omega$-chains of shifts in $S$. The objects of $\DSef(S)$ we also call the \emph{$\omega$-proshifts} over $S$.  
\end{definition}

As mentioned in example~\ref{exm: symbolic dynamics}, the most studied classes of shifts are the class $\SFT$ of shifts of finite type and the class $\SOF$ of sofic shifts. So they are natural choices for $S$. Surprisingly, they share a universal system, as our existence result for deterministic systems says:

\begin{theorem}[Universal deterministic system]
\label{thm: universal and homogeneous proshifts}
There is a deterministic system $(U,T)$ such that
\begin{enumerate}
\item
$(U,T)$ is the up to isomorphism unique system in $\DSef(\SFT)$ that is $\DSef(\SFT)$-universal and $\SFT$-homogeneous. 

\item
$(U,T)$ is the up to isomorphism unique system in $\DSef(\SOF)$ that is $\DSef(\SOF)$-universal and $\SOF$-homogeneous. 
\end{enumerate}
For brevity, we call $(U,T)$ the \emph{universal proshift}.
\end{theorem}

An important corollary of a result of~\cite{Good2020} is that the universal proshift has the shadowing property (i.e., simulated orbits are close to a true orbit).

\begin{figure}[t]
\begin{subfigure}[b]{.24\linewidth}
\centering
\begin{tikzcd}
U
\\
A
\arrow[u, dashed, "f"]
\end{tikzcd}
\caption{Universal}
\label{fig: fraisse definitions universal}
\end{subfigure}
\begin{subfigure}[b]{.24\linewidth}
\centering
\begin{tikzcd}
U
\arrow[loop above, dashed, "h"]
\\
A
\arrow[u, bend left, "f"]
\arrow[u, bend right, swap, "g"]
\end{tikzcd}
\caption{Homogeneous}
\label{fig: fraisse definitions homogeneous}
\end{subfigure}
\begin{subfigure}[b]{.24\linewidth}
\centering
\begin{tikzcd}[column sep=small]
&
C
&
\\
A
\arrow[ur, dashed, "f"]
&
&
B
\arrow[ul, dashed, swap, "g"]
\end{tikzcd}
\caption{Joint embedding}
\label{fig: fraisse definitions joint embedding}
\end{subfigure}
\begin{subfigure}[b]{.24\linewidth}
\centering
\begin{tikzcd}[column sep=small]
&
C
&
\\
B
\arrow[ur, dashed, "g"]
&
&
B'
\arrow[ul, dashed, swap, "g'"]
\\
&
A
\arrow[ul, "f"]
\arrow[ur, swap, "f'"]
&
\end{tikzcd}
\caption{Amalgamation}
\label{fig: fraisse definitions amalgamation}
\end{subfigure}
\caption{Definitions related to the category-theoretic \Fraisse{} theorem.
Dashed arrows indicate the required existence of a morphism (however, unique existence is \emph{not} required). 
}
\label{fig: fraisse definitions}
\end{figure}

\subsection{Proof technique: \Fraisse{} limits}
\label{ssec: proof technique Fraisse limits}

To prove our universality results, we will use the \Fraisse{} theorem (aka \Fraisse{}--\Jonsson{} theorem) from model theory in its category-theoretic generalization due to \textcite[thm.~1.1]{Droste1993}. (This is further discussed and generalized by~\cite{Kubis2014} and~\cite{Caramello2014}.)

We recall the terminology from \cite{Droste1993} to state their version of the \Fraisse{} theorem. 
Let $\mathsf{C}$ be a category. We write $\omega = \{ 0,1,2, \ldots \}$. An \emph{$\omega$-chain} in $\mathsf{C}$ is a structure $(A_i , f_{i,i+1})_{i \in \omega}$ where each $A_i$ is an object in $\mathsf{C}$ and each $f_{i,i+1} : A_i \to A_{i+1}$ is a morphism in $\mathsf{C}$ (we will drop the commas in the subscripts). For $i \leq j$ in $\omega$, define $f_{ij} := f_{j-1 j} \circ \ldots \circ f_{i i+1}$, with $f_{i j} := \id_{A_i}$ if $i = j$. So we have the diagram $(A_i , f_{ij})$ of shape $\omega$ in $\mathsf{C}$. If it has a colimit in $\mathsf{C}$, we denote it $(A,f_i)$ with $f_i : A_i \to A$, and we also call it the colimit of the $\omega$-chain $(A_i , f_{i i+1})$.
An object $B$ of $\mathsf{C}$ is \emph{(category-theoretically) finite} if, for all $\omega$-chains $(A_i , f_{i i+1})_{i \in \omega}$ with colimit $(A,f_i)$, if $g : B \to A$ is a morphism, then there exists $n$ such that, for all $i \geq n$, there is a unique morphism $h : B \to A_i$ such that $g = f_i \circ h$.

\begin{definition}
\label{def: algebroidal category}
A category $\mathsf{C}$ is \emph{semi-algebroidal} if 
\begin{enumerate}
\item
all morphisms in $\mathsf{C}$ are monic,
\item 
every object of $\mathsf{C}$ is a colimit of an $\omega$-chain of finite objects in $\mathsf{C}$, and 
\item
every $\omega$-chain of finite objects in $\mathsf{C}$ has a colimit in $\mathsf{C}$.
\end{enumerate}
Finally, $\mathsf{C}$ is \emph{algebroidal} if it is semi-algebroidal, it contains up to isomorphism only countably many finite objects, and between any two finite objects there are only countably many morphisms.
\end{definition}

The term `algebroidal' traces back, as \cite[138]{Droste1993} note, to \cite{Smyth1982} (in domain theory) and \cite{Banaschewski1976} (in category theory). 
Indeed, \cite[780]{Smyth1982} see algebroidal categories as a generalization of $\omega$-algebraic cpo's, which are a common kind of partial order studied in domain theory.\footnote{We will not discuss cpo's later, but we give here a definition for completeness and to highlight that it is a special case of an algebroidal category. A cpo is a partial order $(P,\leq)$ with a least element in which every increasing $\omega$-chain has a least upper bound (i.e., `colimit'). An element $b \in P$ is $\omega$-compact (i.e., category-theoretically finite) if any increasing $\omega$-chain $(a_n)$ whose least upper bound is $\geq b$ contains an element $a_n$ which is $\geq b$. Finally, $P$ is $\omega$-algebraic if it contains only countably many compact elements, and for any element $a \in P$, there is an increasing $\omega$-chain of compact elements whose least upper bound is $a$. In section~\ref{sec: coalgebra and domain theory}, we give domain-theoretic definitions in (the more common) terms of directed subsets rather than $\omega$-chains, but these are equivalent if there are only countably many compact elements~\cite[sec.~2.2.4]{Abramsky1994}.} Thus, $\omega$-algebraic cpo's---when viewed as a poset category\footnote{The objects are the elements of the partial order $(P,\leq)$ and there is exactly one morphism from $a$ to $b$ if $a \leq b$, with no further morphisms.}---provide a simple example of algebroidal categories. Other common examples also arise in domain theory: many categories of domains form algebroidal categories~\cite{Droste1993}.

\begin{definition}
Let $\mathsf{C}$ be a category. Then $\mathsf{C}$ has the \emph{joint embedding property} if for any objects $A$ and $B$, there is an object $C$ with morphisms $f : A \to C$ and $g : B \to C$ (see figure~\ref{fig: fraisse definitions joint embedding}).
And $\mathsf{C}$ has the \emph{amalgamation property} if for any objects $A$, $B$, and $B'$ with morphisms $f : A \to B$ and $f' : A \to B'$, there is an object $C$ with morphisms $g : B \to C$ and $g' : B' \to C$ such that $g \circ f = g' \circ f'$ (see figure~\ref{fig: fraisse definitions amalgamation}).
\end{definition}

Now the category-theoretic \Fraisse{} theorem says:

\begin{theorem}[\textcite{Droste1993}]
\label{thm: category theoretic fraisse jonsson}
Let $\mathsf{C}$ be an algebroidal category and $\mathsf{C}_\mathsf{fin}$ the full subcategory of category-theoretically finite objects. Then there exists a $\mathsf{C}$-universal and $\mathsf{C}_\mathsf{fin}$-homogeneous object $U$ in $\mathsf{C}$ iff $\mathsf{C}_\mathsf{fin}$ has the joint embedding property and the amalgamation property.
Moreover, if a $\mathsf{C}$-universal and $\mathsf{C}_\mathsf{fin}$-homogeneous object exists in $\mathsf{C}$, then it is unique up to isomorphism.
\end{theorem}

Thus, our proof strategy, both for nondeterministic and deterministic systems, will be to show that the relevant category of systems (1) is algebroidal and (2) that the full subcategory of category-theoretically finite systems has the joint embedding property and the amalgamation property.
We will achieve (1) by proving `algebroidality theorems', i.e., theorems providing easily check-able sufficient conditions for a category of systems to be (semi-) algebroidal (theorems~\ref{thm: algebroidal category of systems}, \ref{thm: category of det sys semialgebroidal}, and~\ref{thm: generating algebroidal categories of det sys}).

\section{Connections to coalgebra and domain theory} 
\label{sec: coalgebra and domain theory}

To put things into perspective, we establish connections to the two main areas of research on computation from a general point of view: coalgebra and domain theory. In particular, we make precise the strong analogies to domain theory alluded to in section~\ref{sec: categories of dynamical systems}.

\emph{Coalgebra}.
The field of coalgebra studies the `mathematics of computational dynamics'~\parencite[vii]{Jacobs2016}. From this perspective, we view our nondeterministic systems $(X,T)$ as the coalgebra $T : X \to \F(X)$ that maps each $x \in X$ to the (closed) set $T(x)$ of its successor states. To make this precise, we need to specify the endofunctor $\F$. Since $X$ is a second-countable Stone space, it is natural to work in that category. So we need to put a topology on the collection $\F(X)$ of all closed subsets of $X$. Such topologies are well-studied as hyperspaces~\cite{Bezhanishvili2022}. A standard one is the \emph{Vietoris topology}: it is generated by (with $U \subseteq X$ open)
\begin{align*}
\Box U &:= \{ A \in \F(X) : A \subseteq U \}
&
\Diamond U &:= \{ A \in \F(X) : A \cap U \neq \emptyset \}.
\end{align*}
If $X$ is a non-empty second-countable Stone space, so is $\F(X)$ with this topology; and $\F$ indeed becomes an endofunctor on the category of second-countable Stone spaces by sending a morphism $f : X \to Y$ to the morphism $\F(f) : \F(X) \to \F(Y)$ which maps a closed set $A$ to its image $f[A]$~\cite{Kupke2003}.
However, then $T$ is an $\F$-coalgebra only if it also is lower-hemicontinuous (and not only upper-hemicontinuous), and, further, coalgebra morphisms do not include all our system morphisms.
We get a closer match, if we work only with the \emph{upper Vietoris topology} (aka \emph{miss topology}) which is generated only by the $\Box U$. The downside is that then $\F(X)$ is only a spectral space, not necessarily a Stone space. This brings us to domain theory.

\emph{Domain theory}.
\textcite{Edalat1995} studies dynamical systems via domain theory by considering, among others, the upper Vietoris topology on $\F(X) \setminus \{ \emptyset \}$ which (in our case where $X$ is a second-countable Stone space) is an algebraic domain and $\F(T)$ a Scott-continuous function (definitions below).\footnote{Moreover, the algebraic domain has only countably many compact elements (definition below) and is bounded-complete (any two elements with an upper bound have a least upper bound).}  
\cite{Hornischer2021} provides another way of assigning a dynamical system $(X,T)$ (viewed as a computational process) to a domain $D$ with a continuous $\llbracket T \rrbracket : D \to D$ (viewed as the denotational semantics of that process: a program $\llbracket T \rrbracket$ of type $D$). This approach also incorporates probability measures on $X$. 
Here, we will make precise the analogy between, on the one hand, the dynamical systems that we study in this paper and, on the other hand, domain theory. To do so, we prove an equivalence between the category $\Sysef$ and a `domain-theoretic' category that we call $\dynAlg$.

To state this equivalence, we recall some terminology~\parencite[e.g.][sec.~I-4]{Gierz2003}.
Let $(P,\leq)$ be a partial order (aka poset). For a subset $S \subseteq P$, the least upper bound of $S$ (resp., greatest lower bound), if it exists, is denoted $\bigvee S$ (resp., $\bigwedge S$). A subset $D \subseteq P$ is \emph{directed} if it is nonempty and, for all $a,b \in D$, there is $c \in D$ with $a,b \leq c$. An element $x \in P$ is (order-theoretically) \emph{compact} if for all directed subsets $D \subseteq P$, if $\bigvee D$ exists and $\bigvee D \geq x$, there is $d \in D$ with $d \geq x$. The set of compact elements of $P$ is written $K(P)$. We call $P$ \emph{algebraic} if, for all $x \in P$, the set $\{ c \in K(P) : c \leq x \}$ is directed and its least upper bound is $x$. We call $P$ \emph{directed complete} if every directed subset of $P$ has a least upper bound. An \emph{algebraic domain} is a directed complete partial order that also is algebraic. 
An \emph{algebraic lattice} is an algebraic domain that also is a lattice (i.e., each finite subset has a least upper bound and a greatest lower bound).
A function $f : A \to B$ between directed complete posets is \emph{Scott continuous} if it preserves directed least upper bounds, i.e., for all directed $D \subseteq A$, we have $f ( \bigvee D ) = \bigvee f [D]$.
The category $\Alg$ of algebraic lattices with Scott continuous functions is a category of domains that was considered early on, when \textcite{Scott1970} used it for an ``Outline of a Mathematical Theory of Computation''.\footnote{Later, other categories of domains were considered in which the partial orders need not have a top element---e.g., algebraic domains. But for us, the top element has an intuitive interpretation: If $X$ is a second-countable Stone space, then $\F(X)$, the set of closed subsets of $X$ ordered by reverse inclusion, is an algebraic lattice. For a dynamics $T$ on $X$, the bottom element $X$ and the top element $\emptyset$ play two dual roles, when considered as possible sets of successor states: The bottom element $X$ is one extreme of being nondeterministic (every state can be a successor), and the top element $\emptyset$ is the other extreme (no state can be a successor).}

Here, we are interested in the following subcategory of $\Alg$. Call an element $c$ of an algebraic lattice $A$ a \emph{co-atom} if $c$ is not the greatest element of $A$, but for all $y \in A$, if $c \leq y$, then either $y = c$ or $y$ is the greatest element of $A$.\footnote{The order-theoretically dual notion is that of an \emph{atom}: an element $a \in A$ that is not the least element but for all $y \in A$, if $y \leq a$, then either $y = a$ or $y$ is the least element.} For an element $x \in A$, write $\CoAt(x)$ for the set of all co-atoms $c$ in $A$ with $c \geq x$. A function $f : A \to B$ between algebraic lattices is \emph{co-atomic} if, for all $x \in A$, we have $f(x) = \bigwedge f[\CoAt(x)]$.
Now, we define the subcategory $\bAlgc$ of $\Alg$ as follows: The objects are the algebraic lattices whose compact elements form a countable sublattice that is a Boolean algebra (i.e., distributive and complemented), and the morphisms are the Scott-continuous and co-atomic functions.\footnote{Algebraic lattices whose compact elements form a sublattice are known as arithmetic lattices~\cite[def.~I-4.7]{Gierz2003}. So we might call the objects of $\bAlgc$ `Boolean-arithmetic lattices'.}

Then our domain-theoretic characterization of the category $\Sysef$ of dynamical systems goes as follows.
For a second-countable Stone space $X$, we (still) write $\F(X)$ for the set of closed subsets of $X$ ordered by reverse inclusion, and if $f : X \rightrightarrows Y$ is a closed-valued and upper-hemicontinuous multifunction between Stone spaces, then $\F(f) : \F(X) \to \F(Y)$ maps $A$ to the image $f[A] = \bigcup \{ f(x) : x \in A \}$.

\begin{theorem}
\label{thm: Sef equivalent to domain theoretic category}
The category $\Sysef$ is equivalent to the following category $\dynAlg$ of `dynamic algebraic lattices':
\begin{itemize}
\item
Objects: Pairs $(A,\alpha)$ where $A$ is an object in $\bAlgc$ and $\alpha: A \to A$ is a morphism in $\bAlgc$.
\item
Morphisms: Pairs $(\epsilon,\pi) : (A,\alpha) \to (B,\beta)$ of morphisms $\epsilon : A \to B$ and $\pi : B \to A$ in $\bAlgc$ such that 
\begin{enumerate}
\item
\label{itm: algebraic morphisms 1}
$\pi$ preserves co-atoms (if $b$ is a co-atom in $B$, then $\pi(b)$ is a co-atom in $A$)
\item
\label{itm: algebraic morphisms 2}
$\pi \circ \epsilon = \id_A$
\item
\label{itm: algebraic morphisms 3}
$\epsilon \circ \pi \leq \id_B$
\item
\label{itm: algebraic morphisms 4}
$\pi \circ \beta \circ \epsilon \leq \alpha$
\item
\label{itm: algebraic morphisms 5}
$\alpha \circ \pi \leq \pi \circ \beta$
\end{enumerate}
\end{itemize}
via the functor that sends the system $(X,T)$ to $(\F(X) , \F(T))$ and that sends the ef-pair $(e,f) : (X,T) \to (Y,S)$ to $(\F(e), \F(f))$.
\end{theorem}

To not digress from establishing the universality results for $\Sysef$, we move the proof to appendix~\ref{app: coalgebra and domain theory}. Of course, the universality result then also holds for $\dynAlg$. Hence the result can be seen within the tradition of existence results concerning universal domains~\parencite{Droste1993}, albeit now for the new category of dynamic algebraic lattices.
To connect back to coalgebras, note that the objects of $\dynAlg$ are coalgebras for the identity functor on $\bAlgc$. However, the morphisms are \emph{not} coalgebra morphisms. In this case, a coalgebra morphism $\varphi : (A,\alpha) \to (B,\beta)$ would be a morphism $\varphi : A \to B$ in $\bAlgc$ such that $\varphi \circ \alpha = \beta \circ \varphi$, yet here we do not have equality, but only the inequality expressed in~\ref{itm: algebraic morphisms 5}, which, together with~\ref{itm: algebraic morphisms 4}, forms a weaker equivariance condition. On the other hand, the morphisms in $\dynAlg$ are also stronger in the sense that they are not just a `one-directional' morphism in $\bAlgc$ but they rather are `bidirectional': They are embedding-projection pairs, as demanded by~\ref{itm: algebraic morphisms 2} and~\ref{itm: algebraic morphisms 3}, which play an important role in $\Alg$. Thus, this corroborates the analogy to ef-pairs, which we mentioned in section~\ref{ssec: morphisms embedding factor pairs}.

\section{Algebroidal categories of nondeterministic systems}
\label{sec: algebroidal categories of systems}

We now work toward the universality result for nondeterministic systems. As described in the proof strategy (section~\ref{ssec: proof technique Fraisse limits}), we start with an algebroidality theorem: sufficient conditions for when a category of (nondeterministic) systems is algebroidal. 
We state the theorem in section~\ref{ssec: algebroidality theorem for nondeterministic systems} and prove it in section~\ref{ssec: proof of the algebroidality theorem}. In the next section, we use this to prove the existence of a universal nondeterministic system.

\subsection{Algebroidality theorem for nondeterministic systems}
\label{ssec: algebroidality theorem for nondeterministic systems}

To state the theorem, let us first recall some terminology from topology. By a \emph{countable} \emph{inverse} (or \emph{projective}) \emph{system} of topological spaces we mean the data of a countable directed partial order $(I,\leq)$, a family $(X_i)_{i \in I}$ of topological spaces, and, for each $i \leq j$ in $I$, a surjective map $f_{ij} : X_j \to X_i$ such that $f_{ii}$ is the identity function on $X_i$ and, for $i \leq j \leq k$ in $I$, we have $f_{ij} \circ f_{jk} = f_{ik}$. 
The category-theoretic limit $(X,f_i)$, also called \emph{inverse} or \emph{projective} \emph{limit}, is given as
\begin{align}
\label{eqn: construction of limit}
X := \Big\{ 
\langle x_i : i \in I \rangle \in \prod_{i \in I} X_i 
: 
\forall i,j \in I \text{ with } i \leq j .
f_{ij} (x_j) = x_i
\Big\},
\end{align}  
endowed with the subspace topology of the product topology, and $f_i \big( \langle x_i : i \in I \rangle \big) := x_i$ are the canonical projection maps. One can show that the $f_i$ are surjective, as is any mediating morphism from another cone.
A standard result is (e.g., \cite[prop.~3.1]{Danos2015}): 
\begin{lemma}
\label{lem: second-countable profinite space iff compact zero dim polish}
A topological space is a countable projective limit of finite discrete spaces iff it is a compact, zero-dimensional Polish space. Because a countable directed partial order has a cofinal chain, the index set of the projective system can be chosen to be $I = \omega$.
\end{lemma}

Now we can state the sufficient condition for algebroidality of a category of systems, i.e., a full subcategory of the category of all dynamical systems $\allSysef$.

\begin{theorem}
\label{thm: algebroidal category of systems}
Let $\subSysef$ be a full subcategory of $\allSysef$ such that
\begin{enumerate}
\item
\label{itm: algebroidal limit closure}
If $\big( (X_i,T_i), (e_{i i+1}, f_{i i+1} ) \big)_{i \in \omega}$ is an $\omega$-chain in $\subSysef$, then $(X,T)$ is again in $\subSysef$, where $X$ is defined as in equation~\ref{eqn: construction of limit} with canonical projections $f_i$ and  
\begin{align*}
T(x) := \big\{ x' \in X : \forall i \in \omega .  f_i (x) \xrightarrow{T_i} f_i (x') \big\}.
\end{align*}

\item
\label{itm: algebroidal factor closure}
If $(X,T)$ is in $\subSysef$ and $f : X \to Y$ is a continuous surjection into a finite discrete space $Y$, then $(Y,S)$ is in $\subSysef$ with $S := f \circ T \circ f^{-1}$, i.e., 
\begin{align*}
S(y) = \big\{ y' \in Y : 
\exists x, x' \in X . 
y = f(x) \text{ and } 
x \xrightarrow{T} x' \text{ and } 
f (x') = y' 
\big\}.
\end{align*}

\end{enumerate}  
Then $\subSysef$ is algebroidal and its category-theoretically finite objects are precisely those systems in $\subSysef$ that have a finite state space.
\end{theorem}

We will prove this result in section~\ref{ssec: proof of the algebroidality theorem}. We end this subsection by checking to which categories of systems the result applies (theorem~\ref{thm: algeboidal subcategories} below) and to which it does not (example~\ref{exm: deterministic systems not algeboidal wrt finite system} below).

\begin{theorem}
\label{thm: algeboidal subcategories}
The following full subcategories of $\allSysef$ satisfy the conditions~\ref{itm: algebroidal limit closure} and~\ref{itm: algebroidal factor closure} of theorem~\ref{thm: algebroidal category of systems}.
\begin{enumerate}
\item
\label{itm: algeboidal subcategories all systems}
$\allSysef$ itself (i.e., the full subcategory of all systems in $\allSysef$).
\item
\label{itm: algeboidal subcategories non-trivial systems}
$\Sysef$ consisting of all nontrivial systems in $\allSysef$.
\item
\label{itm: algeboidal subcategories total systems}
$\tSysef$ consisting of all total and nontrivial systems in $\allSysef$.
\end{enumerate}
\end{theorem}

\begin{proof}
Ad~(\ref{itm: algeboidal subcategories all systems}).
Concerning condition~\ref{itm: algebroidal limit closure}, let $\big( (X_i,T_i), (e_{i i+1}, f_{i i+1} ) \big)_{i \in \omega}$ be an $\omega$-chain in $\allSysef$. We need to show that $(X,T)$ is a dynamical system, i.e., that (a) $X$ is a compact, zero-dimensional Polish space and (b) $T$ is a closed-valued, upper-hemicontinuous multifunction.
Regarding (a), $X$ is, via the $f_i$, a projective limit of a countable projective diagram of compact, zero-dimensional Polish spaces, and these spaces are closed under such limits.
Regarding (b), we show that the graph of $T$ is closed. First, note that, since $f_i$ is continuous, its graph is closed, and so is the graph of $f_i^{-1}$. By assumption, the graph of $T_i$ is closed, so the graph of $f^{-1} \circ T_i \circ f_i$ is closed qua composition of multifunctions with closed graphs. Hence $\Gr (T) = \bigcap_{i \in \omega} f^{-1} \circ T_i \circ f_i$ is closed qua intersection of closed sets.

Concerning condition~\ref{itm: algebroidal factor closure}, let $(X,T)$ be a dynamical system and $f : X \to Y$ a continuous surjection into a finite discrete space $Y$. We need to show that $(Y,S)$ is a dynamical system. 
Indeed, qua finite discrete space, $Y$ is compact zero-dimensional Polish and the multifunction $S$ trivially has a closed graph.

Ad~(\ref{itm: algeboidal subcategories non-trivial systems}).
Concerning condition~\ref{itm: algebroidal limit closure}, let $\big( (X_i,T_i), (e_{i i+1}, f_{i i+1} ) \big)_{i \in \omega}$ be an $\omega$-chain in $\Sysef$. From~(\ref{itm: algeboidal subcategories all systems}), we know that $(X,T)$ is a dynamical system, so it remains to show that it is nontrivial.
Indeed, since $X_0$ is nontrivial, let $x_0 \xrightarrow{T_0} x'_0$. Since $f_{01} : X_1 \to X_0$ is a factor, there is $x_1 \xrightarrow{T_1} x'_1$ with $f_{01} (x_1) = x_0$ and $f_{01} (x_1') = x_0'$. Since $f_{12} : X_2 \to X_1$ is a factor, there is $x_2 \xrightarrow{T_2} x'_2$ with $f_{12} (x_2) = x_1$ and $f_{12} (x_2') = x_1'$. We continue like this to build sequences $x = \langle x_i : i \in \omega \rangle$ and $x' = \langle x_i' : i \in \omega \rangle$ with 
$x_i \xrightarrow{T_i} x_i'$ and
$f_{i i+1} (x_{i+1}) = x_i$ and
$f_{i i+1} (x_{i+1}') = x_i'$.
So, by definition, $x,x' \in X$ and $x \xrightarrow{T} x'$.

Concerning condition~\ref{itm: algebroidal factor closure}, let $(X,T)$ be a nontrivial dynamical system and $f : X \to Y$ a continuous surjection into a finite discrete space $Y$. From~(\ref{itm: algeboidal subcategories all systems}), we know that $(Y,S)$ is a dynamical system, and it again is nontrivial: if $x \xrightarrow{T} x'$, then $f(x) \xrightarrow{S} f(x')$ by definition of $S$.

Ad~(\ref{itm: algeboidal subcategories total systems}).
Concerning condition~\ref{itm: algebroidal limit closure}, let $\big( (X_i,T_i), (e_{i i+1}, f_{i i+1} ) \big)_{i \in \omega}$ be an $\omega$-chain in $\tSysef$. From~(\ref{itm: algeboidal subcategories non-trivial systems}), we know that $(X,T)$ is a nontrivial dynamical system, so it remains to show that $T$ is total.
Let $x \in X$ and find $x' \in T(x)$. Consider 
$F_i := f_i^{-1} ( T_i ( f_i (x) ) )$, 
which is a nonempty closed set.
If $i \leq j$, then $F_i \supseteq F_j$: if $x' \in F_j$, then 
$f_j (x') \in  T_j ( f_j (x) )$, 
so, since $f_{ij}$ is a factor and hence 
$f_{ij} \circ T_j \subseteq T_i \circ f_{ij}$,
we have
\begin{align*}
f_i (x')
=
f_{ij}  \big( f_j (x') \big) \in  f_{ij} \big[ T_j ( f_j (x) ) \big]
\subseteq
T_i \big[ f_{ij} ( f_j (x) ) \big]
=
T_i \big[ f_i  (x) \big],
\end{align*}
hence $x' \in F_i$. By compactness, there is $x' \in \bigcap_i F_i = T(x)$.

Concerning condition~\ref{itm: algebroidal factor closure}, let $(X,T)$ be a total and nontrivial dynamical system and $f : X \to Y$ a continuous surjection into a finite discrete space $Y$. From~(\ref{itm: algeboidal subcategories non-trivial systems}), we know that $(Y,S)$ is a nontrivial dynamical system, so it remains to show that it is total.
Let $y \in Y$ and find $y'$ with $y' \in S(y)$. Since $f$ is surjective, there is $x \in X$ with $f(x) = y$. Since $T$ is total, there is $x' \in T(x)$. Define $y' := f(x')$. Then, by definition of $S$, $y' \in S(y)$, as needed. 
\end{proof}

The list in theorem~\ref{thm: algeboidal subcategories} begs the question: can deterministic systems be included? No, as the next example shows: they do not satisfy the conclusion of theorem~\ref{thm: algebroidal category of systems} (they violate part~(\ref{itm: algebroidal factor closure}) of the sufficient condition).

\begin{example}
\label{exm: deterministic systems not algeboidal wrt finite system}
The subcategory $\detSysef$ of deterministic systems cannot be algebroidal with category-theoretically finite objects being the systems with a finite state space. This is because not every deterministic system is a colimit along ef-pairs of finite deterministic systems. In fact, the full shift $(2^\omega , \sigma)$ does not have any nontrivial finite factor. (An elementary proof is in appendix~\ref{app: proofs from section algebroidal categories of systems}.) Hence there are not enough factors to form a diagram with the full shift being the colimit along ef-pairs.
\end{example}

However, we will see, in section~\ref{sec: universality for deterministic systems}, that the category $\detSysef$ \emph{is} semi-algebroidal, but its category-theoretically finite objects are very different: they are the shifts. Since there are uncountably many such shifts, $\detSysef$ is only semi-algebroidal and not algebroidal. This indicates that the notion of category-theoretical finiteness depends quite subtly on the surrounding category.

\subsection{Proof of the algebroidality theorem}
\label{ssec: proof of the algebroidality theorem}

In this section, we prove theorem~\ref{thm: algebroidal category of systems} in a sequence of propositions. (We will use some of these again when considering deterministic systems.)

Regarding part~(1) of being algebroidal, we have the following.

\begin{proposition}
\label{prop: all morphisms are monic}
Let $\subSysef$ be a full subcategory of $\allSys$. Then all morphisms in $\subSysef$ are monic.
\end{proposition}

\begin{proof}[Proof]
Let $(e,f) : (X,T) \to (Y,S)$ and $(e_1,f_1), (e_2,f_2) : (Z,R) \to (X,T)$ be morphisms in $\subSysef$ such that
$(e,f) \circ (e_1,f_1) = (e,f) \circ (e_2,f_2)$.
We show that 
$(e_1,f_1) = (e_2,f_2)$.
By definition,
$f_1 \circ f = f_2 \circ f$.
Since $f$ is a surjective function, 
$f_1 = f_2$. 
Since $e_1$ (resp.\ $e_2$) is the unique embedding corresponding to $f_1$ (resp.\ $f_2$), also $e_1 = e_2$.
\end{proof}

We show, more than part (3) requires, that any $\omega$-chain has a colimit.

\begin{proposition}
\label{prop: every chain has colimit}
Let $\subSysef$ be a full subcategory of $\allSys$ with property~(\ref{itm: algebroidal limit closure}).
Then every $\omega$-chain $\big( (X_i,T_i), (e_{i i+1}, f_{i i+1} ) \big)_{i \in \omega}$ in $\subSysef$ has a colimit in $\subSysef$, namely $\big( (X,T) , (e_i,f_i) \big)$ with $(X,T)$ and $f_i$ as in condition~\ref{itm: algebroidal limit closure}.
In particular, in any (up to unique isomorphism unique) colimit $\big( (X,T) , (e_i,f_i) \big)$, we have, for all $x,x' \in X$:
\begin{enumerate}
\item
\label{prop: every chain has colimit 1}
$x = x'$
iff 
$\forall i \in \omega : f_i (x) = f_i(x')$.
 
\item
\label{prop: every chain has colimit 2}
$x' \in T(x)$
iff
$\forall i \in \omega : f_i(x') \in T_i ( f_i(x) )$.
\end{enumerate}
\end{proposition}

\begin{proof}
By property~\ref{itm: algebroidal limit closure}, $(X,T)$ is an object in $\subSysef$. We first show that $f_i : (X,T) \to (X_i , T_i)$ is a factor. Then, writing $e_i := \underline{f_i}$, we have the ef-pairs $(e_i , f_i) : (X_i , T_i) \to (X,T)$, which commute with the $(e_{i j}, f_{i j} )$. In particular, we have that $\big( (X,T), (e_i , f_i) \big)$ is a cocone.

To show that $f_i$ is a factor, note that it already is a surjective continuous function. Concerning equivariance~\ref{itm: nondet equivariance factor}, if $x \xrightarrow{f_i} f_i (x)$ and $x \xrightarrow{T} x'$, then, by definition, $f_i (x') \in T_i ( f_i (x) )$, so $x' \xrightarrow{f_i} f_i (x')$ and $f_i (x) \xrightarrow{T_i} f_i (x')$.
For equivariance~\ref{itm: nondet equivariance hom back},
Assume 
$x_i \xrightarrow{T_i} x_i'$.
For $k < i$, define $x_k := f_{ki}(x_i) \in X_k$ and $x'_k := f_{ki}(x'_i) \in X_k$, so, since, $f_{ki}$ is a factor, $x_k \xrightarrow{T_k} x_k'$. 
For $i+1$, there is, since $f_{i i+1}$ is a factor, some $x_{i+1} \xrightarrow{T_{i+1}} x_{i+1}'$ such that $f_{i i+1} (x_{i+1}) = x_i$ and $f_{i i+1} (x'_{i+1}) = x'_i$. 
Similarly, for $i+2$, and so on. 
So we get sequences $x = (x_j)$ and $x' = (x'_j)$ with 
$x_j \xrightarrow{T_j} x_j'$ and
$f_{j j+1} (x_{j+1}) = x_j$ and
$f_{j j+1} (x_{j+1}') = x_j'$.
So, by definition, $x,x' \in X$ and $x \xrightarrow{T} x'$ with $x \xrightarrow{f_i} x_i$ and $x' \xrightarrow{f_i} x_i'$.

Finally, we show that $\big( (X,T), (e_i , f_i) \big)$ is limiting. 
Let 
$\big( (Y,S), (g_i , h_i) \big)$ 
be a cocone in $\subSysef$ to 
$( (X_i,T_i), (e_{ij}, f_{ij} ) )$.
In particular, $(Y,h_i)$ is a cone to $(X_i , f_{ij})$, so there is a unique mediating  
$u : Y \to X$, 
which is a continuous surjective function that is, as usual, given by $u(y) := \langle h_i (y) : i \in \omega \rangle$. 
So it suffices to show that it is a factor.

Concerning equivariance~\ref{itm: nondet equivariance factor}, assume
$y \xrightarrow{u} u(y) =: x$
and
$y \xrightarrow{S} y'$.
We need to show 
$u(y) \xrightarrow{T} u(y') =: x'$.
Given $i \in \omega$, we have, since $h_i$ is a factor, 
$h_i (y') \in T_i ( h_i (y) )$.
Since $h_i = f_i \circ u$, we have
$h_i (y) = f_i \circ u (y) = f_i (x)$
and similarly
$h_i (y') = f_i (x')$. Hence
$f_i (x') \in T_i ( f_i (x) )$, as needed.

Concerning equivariance~\ref{itm: nondet equivariance hom back}, assume $x \xrightarrow{T} x'$, and find $y,y' \in Y$ with $u(y) = x$, $u(y') = x'$, and $y \xrightarrow{S} y'$. We make a compactness argument:
Define
\begin{align*}
F_i := \big\{ (y,y') \in Y \times Y : 
h_i(y) = f_i (x) \text{ and }
h_i(y') = f_i (x') \text{ and }
y \xrightarrow{S} y' 
\big\}.
\end{align*}
It is straightforward to check that $F_i$ is nonempty and closed, and if $i \leq j$, then $F_i \supseteq F_j$.
So, by compactness, there is $y,y' \in Y$ with $(y,y') \in \bigcap_i F_i$, hence $y \xrightarrow{S} y'$ and, for all $i \in \omega$, 
$f_i ( u (y) ) = h_i (y) = f_i(x)$ and
$f_i ( u (y') ) = h_i (y') = f_i(x')$, 
so 
$u(y) = x$ and
$u(y') = x'$, as needed.
\end{proof}

The remaining part~(2) of being semi-algebroidal we show in two steps. First, we show it when understanding `finite' in the set-theoretic sense, and then we show that this is equivalent to the category-theoretic sense.

We start, though, with a lemma, which is a version of \parencite[lem.~2.1]{Irwin2006}. For completeness, we add a proof in appendix~\ref{app: proofs from section algebroidal categories of systems}.

\begin{lemma}
\label{lem: commuting trianlge lemma}
Let $f : (X,T) \to (Y,S)$ and $g : (X,T) \to (Z,R)$ be factors of systems in $\allSys$ and let $h : Y \to Z$ be a function with $h \circ f = g$:
\begin{equation*}
\begin{tikzcd}
&
(X,T)
\arrow[dl, swap, "f"]
\arrow[dr, "g"]
&
\\
(Y,S)
\arrow[rr, "h"]
&
&
(Z,R)
\end{tikzcd}
\end{equation*} 
Then $h : (Y,S) \to (Z,R)$ already is a factor.
\end{lemma}

\begin{proposition}
\label{prop: every object is colimit}
Let $\subSysef$ be a full subcategory of $\allSys$ with property~(\ref{itm: algebroidal factor closure}).
Then every object of $\subSysef$ is a colimit of an $\omega$-chain of dynamical systems in $\subSysef$ that have finite state spaces.
\end{proposition}

\begin{proof}
Let $(X,T)$ be an object in $\subSysef$. 
Since $X$ is a compact zero-dimensional Polish space, there is, by lemma~\ref{lem: second-countable profinite space iff compact zero dim polish}, an inverse sequence $(X_i , f_{ij})_{i \in \omega}$ of finite discrete spaces whose limit is $(X, f_i)$.
For $i \in \omega$, define
$T_i : X_i \rightrightarrows X_i$  
as
$f_i \circ T \circ f_i^{-1}$, so, by condition~\ref{itm: algebroidal factor closure}, $(X_i ,T_i)$ is in $\subSysef$.

We first show that each $f_i : X \to X_i$ is a factor. By construction, it is a continuous surjection. Concerning equivariance~\ref{itm: nondet equivariance factor}, if
$x \xrightarrow{f_i} x_i$ and $x \xrightarrow{T} x'$, then take 
$x_i' := f_i (x')$, 
so, since $x \in f_i^{-1} (x_i)$ and $x' \in T(x)$,
we have $x_i \xrightarrow{T_i} x_i'$ and $x' \xrightarrow{f_i} x_i'$.
Concerning equivariance~\ref{itm: nondet equivariance hom back}, if
$x_i \xrightarrow{T_i} x_i'$, then, by definition, there is $x, x' \in X$ with $x \in f_i^{-1} (x_i)$ and $x' \in T(x)$ and $x_i' = f_i (x')$, as needed.

By construction, $f_{ij} : X_j \to X_i$ is a function with $f_{ij} \circ f_j = f_i$. So, by lemma~\ref{lem: commuting trianlge lemma}, also $f_{ij}$ is a factor.  
Hence
$\big( (X_i, T_i) , (e_{i i+1} , f_{i i+1}) \big)_{i \in \omega}$, with $e_{i i+1}  := \underline{f_{i i+1}}$, is an $\omega$-chain in $\subSysef$ where each $X_i$ is finite.
Moreover, $\big( (X,T) , (e_i,f_i) \big)$, with $e_i := \underline{f_i}$, is a cocone to the $\omega$-chain. 

So it remains to show that the cocone is limiting: If $\big( (Y,S) , (g_i,h_i) \big)$ is another cocone, then 
$(Y , h_i)$ is a cone of $(X_i , f_{ij})$, 
so there is a unique continuous surjection $u : Y \to X$ commuting with the $f_{i}$. Hence it suffices to show that $u$ is a factor from $(Y,S)$ to $(X,T)$.

Concerning equivariance~\ref{itm: nondet equivariance hom back}, assume $x \xrightarrow{T} x'$, and find $y,y' \in Y$ with $u(y) = x$, $u(y') = x'$, and $y \xrightarrow{S} y'$.
Consider 
\begin{align*}
F_i := \big\{ (y,y') \in Y \times Y : 
h_i (y) = f_i (x) \text{ and } h_i (y') = f_i (x')  \text{ and } y \xrightarrow{S} y' 
\big\}.
\end{align*}
It is readily checked that $F_i$ is closed and nonempty, and if $i \leq j$, then $F_i \supseteq F_j$. 
Hence, by compactness, there is $(y,y') \in \bigcap_i F_i$, so $y \xrightarrow{S} y'$ and, for all $i \in \omega$, 
$f_i ( u(y)) = h_i (y) = f_i (x)$
and
$f_i ( u(y')) = h_i (y') = f_i (x')$, 
so, since $(X,f_i)$ is the projective limit,
$u(y) = x$
and
$u(y') = x'$,
as needed.

Concerning equivariance~\ref{itm: nondet equivariance factor}, assume 
$y \xrightarrow{u} x$ and $y \xrightarrow{S} y'$, and find 
$x' \in X$ with 
$x \xrightarrow{T} x'$ and $y' \xrightarrow{u} x'$.
Consider
\begin{align*}
F_i := \big\{ (x,x') \in X : 
x \xrightarrow{T} x' \text{ and } f_i (x) = h_i (y) \text{ and } f_i (x') = h_i (y')
\big\}.
\end{align*}
Again, it is readily checked that $F_i$ is closed and nonempty, and if $i \leq j$, then $F_i \supseteq F_j$. 
Hence, by compactness, there is $(x,x') \in \bigcap_i F_i$.
So $x \xrightarrow{T} x'$. 
Moreover, $x = u(y)$ because, for any $i$, we have
$f_i(x) = h_i (y) = f_i ( u(y) )$.
Similarly, $x' = u(y')$. So $u(y) \xrightarrow{T} x'$ and $y' \xrightarrow{u} x'$, as needed.
\end{proof}

To show the equivalence of the two senses of `finite', we need two more lemmas. 
The first is a useful fact about algebroidal categories.
\begin{lemma}
\label{lem: finite colimit is isomorphic to object in diagram}
Let $\mathsf{C}$ be a category in which all morphisms are monic. Let $B$ be a category-theoretically finite object of $\mathsf{C}$. If $(B,f_i)$ is the colimit of an $\omega$-chain  $(A_i, f_{i i+1})_{i \in \omega}$ in $\mathsf{C}$, then some $f_i : A_i \to B$ is an isomorphism. 
\end{lemma}

\begin{proof}
By using $g = \id_B$ in the definition of being category-theoretically finite, there is, in particular, some $i$ and a unique morphism $h : B \to A_i$ such that $g = f_i \circ h$. We claim that $h$ is the inverse of $f_i : A_i \to B$. 
Since $\id_B = g = f_i \circ h$, it remains to show $h \circ f_i = \id_{A_i}$.
Indeed, note that
$f_i \circ \id_{A_i} = \id_B \circ f_i = f_i \circ h \circ f_i$. 
Since $f_i$ is monic, this implies $\id_{A_i} =  h \circ f_i$, as needed.
\end{proof}

The second lemma is a standard fact about the interplay between partitions and projections in inverse limits. For completeness, we add a proof in appendix~\ref{app: proofs from section algebroidal categories of systems}.

\begin{lemma}
\label{lem: refine partitions}
Let $\subSysef$ be a full subcategory of $\allSys$.
Let $\big( (X_i,T_i), (e_{i i+1}, f_{i i+1} ) \big)_{i \in \omega}$ be an $\omega$-chain in $\subSysef$. Let $\big( (X,T) , (e_i,f_i) \big)$ be a colimit in $\subSysef$. 
Assume $\mathcal{C} = \{ C_1 , \ldots , C_n \}$ is a finite partition of $X$ consisting of closed sets. 
Then there is $i \in \omega$ such that $X / f_i$ (i.e., the quotient under the equivalence relation $x \equiv_i x'$ iff $f_i(x) = f_i (x')$) refines $\mathcal{C}$.\footnote{If $\mathcal{C}$ and $\mathcal{D}$ are two partitions of an underlying set $X$, we say that $\mathcal{D}$ \emph{refines} $\mathcal{C}$ if, for every $D \in \mathcal{D}$, there is $C \in \mathcal{C}$ such that $D \subseteq C$.}   
\end{lemma}

\begin{proposition}
\label{prop: category finite iff system finite}
Let $\subSysef$ be a full subcategory of $\allSys$ with property~(\ref{itm: algebroidal factor closure}).
Let $(X,T)$ be in $\subSysef$. Then 
$(X,T)$ is category-theoretically finite in $\subSysef$ 
iff
$X$ is a finite set.
\end{proposition}

\begin{proof}
($\Rightarrow$)
Assume $(X,T)$ is category-theoretically finite.
By proposition~\ref{prop: every object is colimit}, $(X,T)$ is the colimit in $\subSysef$ of an $\omega$-chain $\big( (X_i , T_i) , (e_{i i+1} , f_{i i+1}) \big)$ of systems with finite state spaces.
By lemma~\ref{lem: finite colimit is isomorphic to object in diagram}, $(X,T)$ is isomorphic to some $(X_i, T_i)$, so, like $X_i$, also $X$ is a finite set.

($\Leftarrow$)
Assume $X$ is finite.
Let $( (Y_i,S_i), (e_{i i+1}, f_{i i+1} ) )_i$ be an $\omega$-chain with colimit $((Y,S), (e_i, f_i ))$ in $\subSysef$ and $(g,h) : (X,T) \to (Y,S)$ a morphism. We need to find $n \in \omega$ such that, for all $i \geq n$, there is a unique morphism $(g',h') :  (X,T) \to (Y_i,S_i)$ such that 
$(g,h) = (e_i, f_i ) \circ (g',h')$.

By lemma~\ref{lem: refine partitions}, there is $n \in \omega$ such that $\mathcal{D} := Y/f_n$ refines the closed and finite partition $ \mathcal{C} := \{ g ( x ) : x \in X \}$ of $Y$.
We claim:
\begin{itemize}
\item[($*$)]
For each $y_n \in Y_n$, there is a unique $x \in X$ such that $f_n^{-1} (y_n) \subseteq g(x)$. We write $u(y_n) := x$. 
\end{itemize}
Indeed, regarding existence, given $y_n \in Y_n$, since $\{ f_n^{-1} (y_n) : y_n \in Y_n \}$ refines $\mathcal{C}$, there is $x \in X$ with $f_n^{-1} (y_n) \subseteq g(x)$.
Regarding uniqueness, if $x, x' \in X$ are such that $f_n^{-1}(y_n) \subseteq g(x)$ and $f_n^{-1}(y_n) \subseteq g(x')$, then, by surjectivity of $f_n$, we have $\emptyset \neq f_n^{-1}(y_n) \subseteq g(x) \cap g(x')$, so, since $g$ is partition-injective, $x = x'$.

By~($*$), we have the function $u : Y_n \to X$. For the functions $f_n :Y \to Y_n$ and $h : Y \to X$, we have $h = u \circ f_n$: Given $y \in Y$, we need to show that $x := h(y)$ is such that $f_n^{-1} ( f_n (y) ) \subseteq g(x)$. Since $\mathcal{D}$ refines $\mathcal{C}$, it suffices to show $f_n^{-1} ( f_n (y) ) \cap g(x) \neq \emptyset$. Indeed, we trivially have $y \in f_n^{-1} ( f_n (y) )$ and, qua ef-pair, we have $y \in g \circ h (y) =  g(x)$.

By lemma~\ref{lem: commuting trianlge lemma}, since $f_n$ and $h$ are factors, also $u$ is. 
This finishes the proof: Given $i \geq n$, consider $(g',h') := (e_{n i} , f_{n i}) \circ  (\underline{u} , u) : (X,T) \to (Y_i,S_i)$. 
Then 
$(e_i, f_i ) \circ (g',h')
=
(e_n , f_n ) \circ  (\underline{u} , u)
=
(g,h)
$.
And $(g',h')$ is unique with this property since $(e_i,f_i)$ is monic (proposition~\ref{prop: all morphisms are monic}). 
\end{proof}

Now we know that a full subcategory $\subSysef$ of $\allSys$ with property~\ref{itm: algebroidal limit closure} and~\ref{itm: algebroidal factor closure} is semi-algebroidal. And it immediately is algebroidal:

\begin{proposition}
\label{prop: only countably many finite objects}
Let $\subSysef$ be a full subcategory of $\allSys$ with property~(\ref{itm: algebroidal factor closure}).
Then $\subSysef$ contains up to isomorphism only countable many category-theoretically finite objects, and between any two such objects there are only countably many morphisms.
\end{proposition}

\begin{proof}
Immediate, since the category-theoretically finite objects in $\subSysef$ are the systems with a finite state space.
\end{proof}

\section{The universal and homogeneous nondeterministic system}
\label{sec: universal and homogeneous nondeterministic system}

Now we can show the universality result for nondeterministic systems (subsection~\ref{ssec: the universality result}). Then we comment on the nature of the universal system (subsection~\ref{ssec: the nature of the universal system}).

\subsection{Proof of theorem~\ref{thm: universal nondeterministic system}}
\label{ssec: the universality result}

Having established the algebroidality of $\Sysef$, it only remains, according to our proof strategy (section~\ref{ssec: proof technique Fraisse limits}), to establish the joint embedding property and the amalgamation property of $\Syseffin$. (To recall, $\Syseffin$ is the full subcategory of $\Sysef$ consisting of the category-theoretically finite objects in $\Sysef$.) Then theorem~\ref{thm: universal nondeterministic system} follows from the \Fraisse{} theorem. 
Fortunately, here this is straightforward.

We start with the amalgamation property. We can use the usual pullback (or fiber product) construction for sets.

\begin{proposition}
\label{prop: syseffin has amalgamation property}
$\Syseffin$ has the amalgamation property.
\end{proposition}

\begin{proof}
Let $(X,T)$, $(Y_0,S_0)$, and $(Y_1,S_1)$ be in $\Syseffin$ and let
$(e_0,f_0) : (X,T) \to (Y_0,S_0)$
and
$(e_1,f_1) : (X,T) \to (Y_1,S_1)$
be morphisms in $\Syseffin$.
Consider
\begin{align*}
Z := \big\{ 
(y_0,y_1) \in Y_0 \times Y_1 
: 
f_0 (y_0) = f_1 (y_1) 
\big\},
\end{align*}
Define $R : Z \rightrightarrows Z$ by
\begin{align*}
R (y_0 , y_1) 
:=
\big\{
(y_0' , y_1') \in Z 
: 
y_0' \in S_0 (y_0) \text{ and } y_1' \in S_1 (y_1) 
\big\}.
\end{align*}
So $(Z,R)$ is a finite dynamical system, and it is nontrivial since $(X,T)$ is nontrivial and $f_0$ and $f_1$ are factors.
It is readily checked that the projections $\pi_0 : Z \to Y_0$ and $\pi_1 : Z \to Y_1$ to the first and second component, respectively, are factors. Then 
$(\underline{\pi_0} , \pi_0) : (Y_0,S_0) \to (Z , R)$
and
$(\underline{\pi_1} , \pi_1) : (Y_1,S_1) \to (Z , R)$
are morphisms in $\Syseffin$
and, by construction,
$f_0 \circ \pi_0 = f_1 \circ \pi_1$, 
so 
$(\underline{\pi_0} , \pi_0) \circ (e_0 , f_0)
=
(\underline{\pi_1} , \pi_1) \circ (e_1 , f_1)
$, as needed.\footnote{This proof does not carry over to total systems: The dynamics $R$ is nontrivial, but it need not be total, even if the systems to be amalgamated are. In section~\ref{ssec: no universal deterministic system}, we show that in fact no other construction can work, because the total systems do not have a universal system.}
\end{proof}

Regarding the joint embedding property, it is well-known that this is implied by the amalgamation property if there is an initial object. 

\begin{proposition}
\label{prop: syseffin has joint embedding property}
$\Sysef$ has an initial object. Since $\Syseffin$ has the amalgamation property, this implies that it has the joint embedding property.
\end{proposition}

\begin{proof}
The initial object is the system $(X,T)$ consisting of one state $*$ with a self-loop, i.e., $T(*) = \{ * \}$. Given any nontrivial system $(Y,S)$, the unique ef-pair $(e,f) : (X,T) \to (Y,S)$ is given by $f$ mapping every $y \in Y$ to $*$, which is readily checked to be a factor.
\end{proof}

This finishes the proof of theorem~\ref{thm: universal nondeterministic system}. As noted, the theorem generalizes the result for profinite undirected graphs. More precisely:

\begin{remark}
\label{rmk: generalization of universal profinite undirected graph}
As mentioned in section~\ref{sec: literature review}, from a dynamical systems perspective, we consider \emph{directed} edges $x \xrightarrow{T} x'$, rather than \emph{undirected} (i.e., reflexive and symmetric) edges. But, since theorem~\ref{thm: algebroidal category of systems} applies to any subcategory of $\Sysef$, we can, as mentioned in section~\ref{sec: literature review}, also recover the existence result of the projectively universal homogeneous (undirected) graph~\cite{Irwin2006, Camerlo2010, Geschke2022}. Specifically, call a system $(X,T)$ in $\Sysef$ a \emph{graph} if, for all $x, x' \in X$, we have $x \xrightarrow{T} x$ (reflexivity) and, if $x \xrightarrow{T} x'$, then $x' \xrightarrow{T} x$ (symmetry).
Let $\mathsf{G}^\mathsf{ef}$ be the full subcategory of $\Sysef$ consisting of those systems that are graphs. Then conditions~\ref{itm: algebroidal limit closure} and~\ref{itm: algebroidal factor closure} of theorem~\ref{thm: algebroidal category of systems} are readily checked, so $\mathsf{G}^\mathsf{ef}$ is algebroidal and the full subcategory $\mathsf{G}_\mathsf{fin}^\mathsf{ef}$ of category-theoretically finite objects consists of precisely the finite graphs. The proofs of propositions~\ref{prop: syseffin has amalgamation property} and~\ref{prop: syseffin has joint embedding property} still work and yield the amalgamation property and the joint embedding property for $\mathsf{G}_\mathsf{fin}^\mathsf{ef}$. Hence the \Fraisse{} limit theorem yields that $\mathsf{G}^\mathsf{ef}$ has a $\mathsf{G}^\mathsf{ef}$-universal and $\mathsf{G}^\mathsf{ef}_\mathsf{fin}$-homogeneous system. 
\end{remark}

\subsection{The nature of the universal system}
\label{ssec: the nature of the universal system}

We end this section with some comments about the nature of the universal nondeterministic system $(U,T)$: its symmetry, universality, complexity, and its status as an analog computer.

\emph{Symmetry}.
The homogeneity property, that makes $(U,T)$ unique, is quite remarkable. In section~\ref{ssec: statement of the main results}, we motivated it as saying that `every part of $(U,T)$ looks the same'. Another way of saying this is that $(U,T)$ is `maximally symmetric':
whenever we have two ways of embedding a finite dynamical system into the universal system, there is a symmetry (i.e. isomorphism) of the universal system that transforms one embedding into the other. Thus, the embedding of each finite system into the universal one is, in a sense, unique up to isomorphism.

\emph{Universality}. 
In which of the senses that we have identified in section~\ref{sec: literature review} is $(U,T)$ universal?
\begin{itemize}
\item[(\ref{itm: notion of universality factor})]
By construction, $(U,T)$ is factor universal: every dynamical system is a factor of $(U,T)$.

\item[(\ref{itm: notion of universality embedding})]
By construction, $(U,T)$ also is embedding universal, provided that `embedding' is understood as a multi-valued function. This sense of `embedding' comes natural in a multi-valued setting, but it is not the usual `injective function' idea of an embedding.

\item[(\ref{itm: notion of universality Turing})]
Since every Turing machine can be seen as a dynamical system, it is a factor of and embeds into the universal system. Such an ef-pair can be seen as a simulation (see remark~\ref{rmk: simulation and abstraction}). Thus, in this sense, $(U,T)$ is Turing universal.

\item[(\ref{itm: notion of universality approximation})]
Whether $(U,T)$ is approximation universal depends on the understanding of `approximation'. Simulation understood as an ef-pair is, in a sense, best possible approximation, so in this sense $(U,T)$ is approximation universal. But a sense of a universal function approximator (as with neural networks or differential equations) is not available in our general setting here, as there are no associated real-valued functions that can approximate a given function.

\end{itemize}

\emph{Complexity}.
One would expect the universal system to be a very `complicated' or `complex' object since it contains every other system. However, it actually is surprisingly simple, as we will show now: Its state space is a Cantor space (i.e., isomorphic to $2^\omega$), and its dynamics has no complicated long-term behavior because the orbits consist of at most two states. The latter, surprising fact is known for the universal profinite graph, proven by \textcite[prop.~12]{Camerlo2010} (also discussed by~\cite{Geschke2022}). However, in our systems setting, a much simpler proof is available.

First, the claim about the state space \cite[cf.][13]{Geschke2022}.

\begin{proposition}
The state space of the universal system $(U,T)$ from theorem~\ref{thm: universal nondeterministic system} is a Cantor space (i.e., isomorphic to $2^\omega$).
\end{proposition}

\begin{proof}
By Brouwer's theorem~\cite[e.g.][thm.~7.4]{Kechris1995}, a space is a Cantor space if it is a nonempty, compact, zero-dimensional, metrizable space without isolated points (i.e., no points whose singleton is open). 
Since $(U,T)$ is in $\Sysef$, it hence remains to show that $U$ has no isolated points. 

Toward a contradiction, assume $x_0 \in U$ is isolated. Consider the two-element discrete space $Y = \{ y_0, y_1\}$ and the function $f : U \to Y$ which maps $x_0$ to $y_0$ and any $x \in U \setminus \{ x_0 \}$ to $y_1$---so $f$ is a continuous function. Define the dynamics $S$ on $Y$ by: $y \xrightarrow{S} y'$ iff there is $x,x' \in U$ with $f(x) = y$, $f(x') = y'$, and $x \xrightarrow{T} x'$. So $f : (U,T) \to (Y,S)$ is a factor. Now define the system $(Z,R)$ by `doubling' $y_0$: let $Z := \{ y_0, z_0, y_1 \}$ be a three-element discrete space and let $R(z) := S(z)$ if $z \in \{y_0,y_1\}$ and $R(z_0) := \emptyset$. Define $g : Z \to Y$ by mapping $y_0 \mapsto y_0$, $z_0 \mapsto y_0$, and $y_1 \mapsto y_1$. So $g : (Z,R) \to (Y,S)$ is a factor. By universality, there is a factor $h : (U,T) \to (Z,R)$, and, by homogeneity, we can choose $h$ such that $g \circ h = f$.\footnote{This is the idea of $(U,T)$ being a `saturated object', which we will define and use again in the proof of proposition~\ref{prop: isomorphism for universal proshifts}.} 
Hence, since $g ( h(x_0) ) = f(x_0) = y_0$ and $g^{-1}(y_0) = \{ y_0, z_0\}$, we must have either $h(x_0) = y_0$ or $h(x_0) = z_0$. Let's consider the first case (the other is analogous). By surjectivity, there is $x \in U$ with $h(x) = z_0$. But then $x \in U \setminus \{ x_0 \}$, so $y_1 = f(x) = g ( h(x) ) = g (z_0) = y_0$, contradiction. 
\end{proof}

Now, the claim about the dynamics. Let us first fix some terminology. Let $(X,T)$ be a system. An \emph{orbit} in $(X,T)$ is a sequence $(x_n)_{n = 0}^\kappa$ in $X$  with $\kappa \leq \omega$ and $x_n \xrightarrow{T} x_{n+1}$ for all $n < \kappa$. So the orbit is finite if $\kappa < \omega$ or infinite if $\kappa = \omega$. 
If $x \xrightarrow{T} y$, we call $y$ a \emph{successor state} of state $x$. We call a state $x \in X$ a \emph{dead end} (or \emph{sink}) if $x$ has no successor states.
A \emph{loop} in $(X,T)$ is a finite orbit $(x_0, \ldots , x_\kappa)$ with $x_0 = x_\kappa$ and $\kappa \geq 1$. It is a \emph{self-loop} if $\kappa = 1$.

\begin{proposition} 
Orbits in the universal system $(U,T)$ have at most $2$ states. In other words, for every state $x \in U$, either $x$ is a dead end or all its successor states are dead ends. In particular, $(U,T)$ does not have any loops (also no self-loops).
\end{proposition}

\begin{proof}
Toward a contradiction, assume there was an orbit with three states: $x_0 \xrightarrow{T} x_1 \xrightarrow{T} x_2$.
Consider the system $(Y,S)$ with $Y = \{y_0,y_1\}$ and $S(y_0) = \{y_1\}$ and $S(y_1) = \emptyset$. 
It is a non-trivial system and hence is in $\Sysef$. So there is a factor $f : (U,T) \to (Y,S)$.
We cannot have $f(x_0) = y_1$: Otherwise, since $f$ is a factor and $x_0 \xrightarrow{T} x_1$, there must be $y' \in Y$ with $f(x_1) = y'$ and $y_1 \xrightarrow{S} y'$, but $y_1$ is a dead end.
Hence $f(x_0) = y_0$. 
Similarly, we cannot have $f(x_1) = y_1$.
Hence $f(x_1) = y_0$. 
But then, since $x_0 \xrightarrow{T} x_1$ and $f$ is a factor, we have $f(x_0) \xrightarrow{S} f(x_1)$, and hence $y_0 \xrightarrow{S} y_0$, contradiction. 
\end{proof}

\emph{Analog computer?} 
In section~\ref{sec: literature review}, we said that every analog computer can be seen as a dynamical system. But we also noted that the converse is a difficult philosophical question. Accordingly, it is out of scope of this paper to discuss whether the universal system $(U,T)$ is an analog computer, and hence is a universal analog computer. We only note that the simple nature of $(U,T)$ seems to intuitively speak in favor of $(U,T)$ being an analog computer---if it were an extremely intricate and complicated system, it would seem less plausible to view it as a `computing machine'. We pick this up again as an interesting open question in section~\ref{sec: conclusion}.

\section{Universality for deterministic systems?}
\label{sec: universality for deterministic systems}

We have established the existence of a universal nondeterministic system, which even is unique in also being homogenous. Now, for the remainder of this paper, we consider the special case of deterministic systems. After all, they are more prominently researched than nondeterministic systems. One might expect similar universality results, since there are known ways to transform a nondeterministic system into a deterministic one, which we review in section~\ref{ssec: deterministic vs nondeterministic systems}. However, the connection is not close enough: in section~\ref{ssec: no universal deterministic system}, we see that there cannot be a universal deterministic system. But we are very close to universality: in section~\ref{ssec: category of det systems still semi algebroidal}, we show that the category $\detSysef$ still is semi-algebroidal, with category-theoretically finite objects being shifts, which have the joint embedding property and the amalgamation property. So all that is missing for universality is the countability of the category-theoretically finite objects. In the next section, we hence consider how to get subcategories of $\detSysef$ with universal systems.

\subsection{Deterministic vs nondeterministic systems}
\label{ssec: deterministic vs nondeterministic systems}

There are two standard ways to transform a nondeterministic system into a deterministic one. Conceptually, these two ways are reminiscent of a right and a left adjoint to the inclusion $\Inc : \detSysef \hookrightarrow \Sysef$, but they do \emph{not} formally satisfy all the requirements for an adjunction.

The first way (cf.\ `right adjoint' to $\Inc$) is reminiscent of how the powerset functor is right adjoint to the inclusion of the category of sets and functions in the category of sets and relations.
Given a nondeterministic system $(X,T)$, we already saw in section~\ref{sec: coalgebra and domain theory} that we can consider the `deterministic system' $(\F (X) , \F (T) )$, where $\F (X)$ is the set of closed subsets of $X$ and $\F (T) : \F (X) \to \F (X)$ maps $A$ to $T[A]$. But the problem was how to topologize $\F(X)$: if we use the Vietoris topology, $\F(X)$ is again a second-countable Stone space but $\F(T)$ need not be continuous; if we use the upper Vietoris topology, $\F(T)$ is continuous but then $X$ in general only is a spectral space. Moreover, the natural choice for a universal morphism 
$\epsilon_{(X,T)} : (\F (X) , \F (T) ) \to (X,T)$
is a multifunction $e$ mapping $A \in \F(X)$ to $A \subseteq X$ or, in the other direction, a function $f$ mapping $x \in X$ to $\{x\} \in \F(X)$. However, these do not form an ef-pair (e.g., $f$ is not surjective).

The second way (cf.\ `left adjoint' to $\Inc$) is reminiscent of the unraveling functor for labeled transition systems (or Kripke frames).\footnote{See, e.g., \parencite[30]{Winskel1995} or \parencite[220]{Blackburn2001}.} 
Given a nondeterministic system $(X,T)$, define $\overline{X}$ to be the set of all infinite orbits. To recall, an infinite orbit---which, for brevity, we call \emph{path}---is a sequence $\overline{x} = (x_n)_{n \in \omega}$ in $X$ such that, for all $n \in \omega$, we have $x_n \xrightarrow{T} x_{n+1}$. Hence $\overline{X}$ is a subset of $\prod_\omega X$ and we equip it with the subset topology. Define $\sigma(\overline{x})$ as the shift of $\overline{x}$, i.e., the sequence $(x_{n+1})_{n \in \omega}$, which again is a path. 
Write $\Path (X,T) := ( \overline{X} , \sigma )$.
Moreover, the natural choice for a universal morphism 
$\eta_{(X,T)} : (X,T) \to \Path (X,T)$
is the ef-pair whose factor $f : \overline{X} \to X$ projects to the first component, i.e., $f \big( (x_n) \big) := x_0$.
The next proposition shows that, this is all well-defined for total nondeterministic systems. (The straightforward proof is in appendix~\ref{app: proofs from section universality for deterministic systems}.)

\begin{proposition}
\label{prop: path system}
Let $(X,T)$ be a system in $\Sysef$. Then 
\begin{enumerate}
\item
\label{lem: path system 1}
$\Path (X,T)$ is in $\detSysef$ iff $(X,T)$ has an infinite path (which, in particular, is the case if $(X,T)$ is total).

\item
\label{lem: path system 2}
The projection to the first component $\pi_0 : \Path (X,T) \to (X,T)$ is a factor iff $(X,T)$ is total.
\end{enumerate}
\end{proposition}
However, as we will see in corollary~\ref{cor: inclusion of det sys in nondet sys cannot have left adjoint} below, neither $\Path$ nor any other functor can be a left adjoint to the inclusion $\Inc : \detSysef \hookrightarrow \Sysef$.

\subsection{No universal deterministic system}
\label{ssec: no universal deterministic system}

Despite the close connections between deterministic systems and nondeterministic systems, the former cannot have a universal system, in contrast to the latter.

\begin{proposition}
\label{prop: no universal deterministic system}
The category $\detSysef$ does not have a universal system. 
\end{proposition}

\begin{proof}
The claim follows from two results by~\textcite{Darji2017}.
For the purposes of this proof, let us call $(X,T)$ a compact (metric) system if $X$ is a compact (metric) space and $T : X \to X$ a continuous function. A Cantor system is a pair $(2^\omega , S)$ where $2^\omega$ is the Cantor space and $S: 2^\omega \to 2^\omega$ a continuous function. 
\begin{enumerate}
\item
\label{itm: Darji Matheron 2017 1}
\cite[Corollary~4.2]{Darji2017}: 
If $(X,T)$ is a compact metric system, then there is a Cantor system $(2^\omega , S)$ and a factor $f : (2^\omega , S) \to (X,T)$. (This extends the universality of Cantor space from spaces to systems.) 

\item
\label{itm: Darji Matheron 2017 2}
\cite[Remark~4.6]{Darji2017}: 
Assume $Z$ is a compact system that is universal for compact metric systems, i.e., if $(X,T)$ is a compact metric system, there is a factor $f : (Z,R) \to (X,T)$. Then $Z$ cannot be metrizable. (In fact, $(Z,R)$ cannot even factor onto all rotations of the circle.)
\end{enumerate}
Toward a contradiction, assume $(Z,R)$ is universal in $\detSysef$. 
Then $(Z,R)$ is a compact system, and it is universal for compact metric systems (not just zero-dimensional ones): If $(X,T)$ is a compact metric system, then, by~\ref{itm: Darji Matheron 2017 1}, there is a Cantor system $(2^\omega , S)$ and a factor $f : (2^\omega , S) \to (X,T)$.   
Hence $(2^\omega , S)$ is in $\detSysef$, so, by universality, there is a factor $g : (Z,R) \to (2^\omega , S)$. In sum, there is a factor $f \circ g : (Z,R) \to (X ,T)$, as needed.
However, then~\ref{itm: Darji Matheron 2017 2} implies that $Z$ is not metrizable, in contradiction to $Z$ being a Polish space.
\end{proof}

The nonexistence of a universal system in $\detSysef$ can also be extended to some categories of nondeterministic systems, namely $\tSysef$. In theorem~\ref{thm: algebroidal category of systems}, we have seen that this category is algebroidal. But, as we see now, $\tSysef$ cannot have a universal system. In particular, the finite total nondeterministic systems also cannot have the amalgamation property.

\begin{corollary}
\label{cor: cat of total system cannot have universal system}
The category $\tSysef$ does not have a universal system.
\end{corollary}

\begin{proof}
Assume $(Z,R)$ were universal in $\tSysef$. By proposition~\ref{prop: path system}, we can define the deterministic system $\Path (Z,R)$ and have a factor $\pi_0 : \Path(Z,R) \to (Z,R)$. So $\Path(Z,R)$ is in $\detSysef$ and it is universal in there: if $(X,T)$ is a deterministic system, it is, in particular, in $\tSysef$, so, by assumption, there is a factor $f : (Z,R) \to (X,T)$, so there indeed is a factor $f \circ \pi_0 : \Path(Z,R) \to (X,T)$. This contradicts proposition~\ref{prop: no universal deterministic system}.
\end{proof}

Another corollary is that, despite the close connections observed in section~\ref{ssec: deterministic vs nondeterministic systems}, the deterministic systems do not form a reflective subcategory of the nondeterministic systems. In other words:

\begin{corollary}
\label{cor: inclusion of det sys in nondet sys cannot have left adjoint}
The inclusion $\Inc : \detSysef \hookrightarrow \Sysef$ does not have a left adjoint. 
\end{corollary}

\begin{proof}
Toward a contradiction, assume there is a left adjoint (aka reflector) $\Refl : \Sysef \to \detSysef$ with unit $\eta$.
Let $(U,T)$ be the universal system in $\Sysef$ (from theorem~\ref{thm: universal nondeterministic system}). We claim that $\Refl (U,T)$ is universal in $\detSysef$, hence contradicting proposition~\ref{prop: no universal deterministic system}.
Indeed, let $(Y,S)$ be in $\detSysef$. We need to find an ef-pair $(Y,S) \to \Refl(U,T)$. By universality of $(U,T)$, there is an ef-pair $(e,f) : (Y,S) \to (U,T)$. Hence, qua left adjoint, there is a unique ef-pair $\Refl (e,f) = (g,h)$ making the following diagram commute:
\begin{equation*}
\begin{tikzcd}
(Y,S)
\arrow[r, "\eta_{(Y,S)}"]
\arrow[d, swap, "{(e,f)}"]
&
\Refl(Y,S)
\arrow[d, dashed, "{(g,h)}"]
\\
(U,T)
\arrow[r, "\eta_{(U,T)}"]
&
\Refl(U,T)
\end{tikzcd}
\end{equation*}
Hence 
$(g,h) \circ \eta_{(Y,S)} : (Y,S) \to \Refl(U,T)$
is the desired ef-pair.
\end{proof}

\subsection{The category of deterministic systems still is semi-algebroidal}
\label{ssec: category of det systems still semi algebroidal}

In order to understand the nonexistence result from the preceding subsection, we now analyze where exactly the \Fraisse{}-limit method---that worked for nondeterministic systems---breaks down for deterministic systems.
As announced, we will find that $\detSysef$ still is semi-algebroidal, and the category-theoretically finite objects are precisely the shifts. (We already saw in example~\ref{exm: deterministic systems not algeboidal wrt finite system} that the category-theoretically finite objects cannot be those systems with finite state spaces: with those, $\detSysef$ cannot be semi-algebroidal.)
Moreover, a later result (lemma~\ref{lem: amalgamation for shifts of finite type}) will show that the category-theoretically finite objects of $\detSysef$ also have the joint embedding property and the amalgamation property. 
So the only reason why the \Fraisse{}-limit method does not yield a universal system for $\detSysef$ is that the countability requirement is violated: there are uncountably many shifts.

We start by recalling some terminology from dynamical systems theory~\parencite{Lind1995}.
A \emph{finite topological generator} of a system $(X,T)$ in $\detSysef$ is a finite clopen partition $\mathcal{C}$ of $X$ (i.e., $\mathcal{C}$ has finitely many cells and they are clopen) such that $\{ T^{-n} (C) : C \in \mathcal{C} , n \in \omega \}$ generates the topology of $X$.\footnote{Recall that a collection $\mathcal{F}$ of subsets of a topological space $X$ \emph{generates} the topology $\tau$ of $X$ if $\tau$ is the smallest topology containing all subsets of $\mathcal{F}$.}
If $(X,T)$ is in $\detSysef$ and $\mathcal{C}$ is a finite clopen partition of $X$, the $\mathcal{C}$-\emph{name} of $x \in X$ is the sequence $(C_n)$ of $\mathcal{C}$-partition cells where, for each $n$, $C_n$ is the cell with $T^n (x)  \in C_n$.
Then the set $Y$ of all $\mathcal{C}$-names is a shift over alphabet $\mathcal{C}$ and the function $f : X \to Y$ that maps $x$ to its $\mathcal{C}$-name is a factor. We call $Y$ the \emph{$\mathcal{C}$-shift} and $f$ the \emph{$\mathcal{C}$-factor}.

A well-known characterization of finite topological generators is the following~\cite{Hedlund1969}.
To state it, recall a classic result of \textcite{DeGroot1956}: An ultrametric on a set $X$ is a metric $d$ on $X$ which, instead of the usual triangle inequality, satisfies the strong triangle inequality $d(x,z) \leq \max( d(x,y) , d(y,z) )$. If $X$ is a Hausdorff space with a countable base $(U_n)$ of clopens, one can define a topology-generating ultrametric on $X$ by $d(x,y) = \frac{1}{n}$ where $n$ is the first $n$ such that $U_n$ separates $x$ and $y$ (i.e., $x \in U_n$ and $y \not\in U_n$ or vice versa).
In particular, if $(X,T)$ is in $\detSysef$, there is an ultrametric on $X$ that generates the topology.

\begin{proposition}[Hedlund's theorem~\cite{Hedlund1969}]
\label{prop: finite top gen characterization}
Let $(X,T)$ be in $\detSysef$. Let $d$ be an ultrametric that generates the topology on $X$. The following are equivalent:
\begin{enumerate}
\item
\label{prop: finite top gen characterization 1}
$(X,T)$ has a finite topological generator.
\item
\label{prop: finite top gen characterization 2}
$(X,T)$ is isomorphic to a shift.
\item
\label{prop: finite top gen characterization 3}
$(X,T)$ is \emph{expansive}, i.e., there is $\epsilon > 0$ such that, for all $x \neq y$ in $X$, there is $n$ such that $d (T^n (x) , T^n (y) ) \geq \epsilon$.
\end{enumerate}
\end{proposition}

With this terminology, we can formulate an analogue to theorem~\ref{thm: algebroidal category of systems} for deterministic systems.

\begin{theorem}
\label{thm: category of det sys semialgebroidal}
Let $\DSef$ be a full subcategory of $\detSysef$ such that
\begin{enumerate}
\item
\label{itm: det sys algebroidal limit closure}
If $\big( (X_i,T_i) , (e_{i i+1} , f_{i i+1}) \big)_{i \in \omega}$ is an $\omega$-chain in $\DSef$, then $(X,T)$ is again in $\DSef$, where $X$ is defined as in equation~\ref{eqn: construction of limit} with canonical projections $f_i$ and 
$T(\langle x_i : i \in \omega \rangle) := \langle T_i(x_i) : i \in \omega \rangle$.

\item
\label{itm: det sys algebroidal limit of shifts}
If $(X,T)$ is in $\DSef$, then it is the colimit of an $\omega$-chain of shifts in $\DSef$.
\end{enumerate}
Then $\DSef$ is semi-algebroidal and its category-theoretically finite objects are precisely those systems in $\DSef$ that are isomorphic to a shift.
\end{theorem}

Before proving this, we note a sufficient condition for requirement~\ref{itm: det sys algebroidal limit of shifts} (proof below).

\begin{lemma}
\label{lem: sufficient condition for det sys limit of shifts}
Let $\DSef$ be a full subcategory of $\detSysef$ such that: If $(X,T)$ is in $\DSef$ and $\mathcal{C}$ is a finite clopen partition of $X$, then the $\mathcal{C}$-shift is again in $\DSef$. Then $\DSef$ satisfies requirement~\ref{itm: det sys algebroidal limit of shifts} of theorem~\ref{thm: category of det sys semialgebroidal}.
\end{lemma}

This sufficient condition is clearly satisfied for $\DSef = \detSysef$, so we in particular get the announced result of this subsection:

\begin{corollary}
\label{cor: detSysef is semi algebroidal}
The category of deterministic systems $\detSysef$ is semi-algebroidal and its category-theoretically finite objects are precisely the deterministic systems that are isomorphic to a shift.
\end{corollary}

In the remainder of this subsection, we prove these results.

\begin{proof}[Proof of lemma~\ref{lem: sufficient condition for det sys limit of shifts}]
Let $(X,T)$ be in $\DSef$ and show that it is the colimit of an $\omega$-chain of shifts in $\DSef$.
Let $(\mathcal{C}_i)$ be a refining sequence of finite clopen partitions of $X$ generating the topology.
Define $f_i : (X,T) \to (X_i,T_i)$ as the $\mathcal{C}_i$-factor.
By assumption, $(X_i, T_i)$ is in $\DSef$.
Define $f_{ij} : X_j \to X_i$ by mapping a sequence $(D_n)$ of $\mathcal{C}_j$-partition cells to the sequence $(C_n)$ of $\mathcal{C}_i$-partition cells with $D_n \subseteq C_n$. Hence $f_{ii} = \id$ and $f_{ij} = f_{i i+1} \circ \ldots \circ f_{j-1 j}$, and it is also easily checked that $f_i = f_{ij} \circ f_j$.
We show that $(X,T)$ with the $f_i$ is the limit of the $\omega$-chain $\big( (X_i ,T_i) , f_{i  i+1} \big)$.

Note that $\big( (X_i ,T_i) , f_{i  i+1} \big)$ is indeed an $\omega$-chain and $\big( (X ,T) , f_i \big)$ a cone, because the $f_i : (X,T) \to (X_i , T_i)$ are factors and, since $f_i = f_{ij} \circ f_j$, lemma~\ref{lem: commuting trianlge lemma} implies that also $f_{ij}$ is a factor. 
So we need to check that $\big( (X ,T) , f_i \big)$ is limiting. If $\big( (Y ,S) , g_i \big)$ is another cone, define the mediating morphism $u : Y \to X$ by mapping $y$ to the $x$ in $\bigcap_i F^y_i$ with $F^y_i := f_i^{-1} \big( \{ g_i (y) \} \big)$. We now show that this is well-defined and has the required properties.

Well-defined:
It is straightforward to check that the $F^y_i$'s for a descending sequence of nonempty closed sets, so, by compactness, $\bigcap_i F^y_i$ is indeed nonempty.
Concerning the uniqueness of $x$, if $x \neq x'$ were in $\bigcap_i F^y_i$, there is, since $(\mathcal{C}_i)$ generates the topology, some $i$ and $C \in \mathcal{C}_i$ such that $x \in C$ and $x' \not\in C$.
But then 
$f_i (x) = g_i (y) = f_i (x')$, so
$C = f_i (x)_0 = f_i (x')_0 \ni x'$, contradiction.

Surjective:
Let $x \in X$ and find $y \in Y$ with $u(y) = x$. Again it is straightforward to check that $G_i := g_i^{-1} ( \{ f_i(x) \} )$ forms a descending sequence of nonempty closed sets, so, by compactness, there is $y \in \bigcap_i G_i$ and, by construction, $x \in \bigcap_i F_i^y$, so $u(y) = x$.

Commuting:
Let $y \in Y$ and $i \in \omega$ and show $f_i ( u (y) ) = g_i(y)$.
By definition, $u (y) \in \bigcap_i  F^y_i$, so $f_i ( u(y) ) = g_i (y)$.

Equivariant:
Let $y \in Y$ and show $u ( S (y) ) = T ( u(y) )$.
So let $i \in \omega$ and show $T( u(y) ) \in F_i^{S(y)}$. Indeed, we have 
\begin{align*}
f_i ( T (u (y) ) 
=
T_i ( f_i ( u (y) ) )
=
T_i ( g_i (y) )
=  
g_i ( S (y) ).
\end{align*}

Unique:
If $u' : (Y,S) \to (X,T)$ is another factor with $f_i \circ u' = g_i$ and $y \in Y$, we claim $u'(y) = u(y)$.
Indeed, for $i \in \omega$, we have $f_i ( u'(y) ) = g_i (y)$, so $u'(y) \in F^y_i$. Since $u(y)$ is the unique element in $\bigcap_i F^y_i$, the claim follows.
\end{proof}

\begin{lemma}
\label{lem: det sys cat finite iff fin gen}
Let $\DSef$ be a full subcategory of $\detSysef$.
If $(X,T)$ has a finite topological generator, then $(X,T)$ is category-theoretically finite in $\DSef$.
If $\DSef$ satisfies requirement~\ref{itm: det sys algebroidal limit of shifts} of theorem~\ref{thm: category of det sys semialgebroidal}, then the converse holds, too.
\end{lemma}

\begin{proof}
($\Rightarrow$) 
Assume $\mathcal{C}$ is a finite topological generator of $(X,T)$. Moreover, let $\big( (Y_i,S_i) , (e_{i i+1} , f_{i i+1}) \big)_{i \in \omega}$ be an $\omega$-chain in $\DSef$ with colimit $(Y,S)$ and let $(g,h) : (X,T) \to (Y,S)$ be an ef-pair. We need to find $n \in \omega$ such that, for all $i \geq n$, there is a unique morphism $(g',h') : (X,T) \to (Y_i,S_i)$ such that $(g,h) = (e_i,f_i) \circ (g',h')$.

By lemma~\ref{lem: refine partitions}, there is $n \in \omega$ such that $\mathcal{D} := Y/f_n$ refines the closed and finite partition $\{ h^{-1} (C) : C \in \mathcal{C} \}$. 
So for each $y_n \in Y_n$, there is a unique $C = C(y_n) \in \mathcal{C}$ such that $f_n^{-1}(y_n) \subseteq h^{-1}(C)$. 
Now we claim: 
\begin{itemize}
\item[(a)]
For any $y \in Y$, $y_n \in Y_n$ and $x \in X$, if $f_n(y) = y_n$ and $x = h(y)$, then, for all $k$, the cell $C(S_n^k(y_n))$ is the one containing $T^k(x)$.

\item[(b)]
For each $y_n \in Y_n$, there is a unique $x \in X$ such that, for all $k$, we have $T^k(x) \in C(S_n^k(y_n))$.
\end{itemize}
Indeed, to show~(a), let $k$ be given and let $C \in \mathcal{C}$ be such that $T^k(x) \in C$. We need to show that $C = C(S_n^k(y_n))$, i.e., that $f_n^{-1}(S_n^k(y_n)) \subseteq h^{-1}(C)$. By the refinement of the partitions, it suffices to show that $f_n^{-1}(S_n^k(y_n)) \cap h^{-1}(C) \neq \emptyset$. For this, note that $f_n( S^k(y) ) = S_n^k ( f_n (y) ) = S_n^k( y_n )$ and $h(S^k(y)) = T^k ( h(y) ) = T^k (x) \in C$, so $S^k(y) \in f_n^{-1}(S_n^k(y_n)) \cap h^{-1}(C)$, as needed.
To show~(b), uniqueness follows since $\mathcal{C}$ is generator: if $x$ and $x'$ have the described property, they have the same $\mathcal{C}$-name, so they must be identical. For existence, there is, since $f_n$ is surjective, some $y \in Y$ with $f_n(y) = y_n$. Set $x := h(y) \in X$. Then~(a) implies $T^k(x) \in C(S_n^k(y_n))$ for all $k$.

Now, by~(b), there is a function $u : Y_n \to X$ that maps $y_n$ to the unique $x \in X$ mentioned in~(b). For the functions $f_n : Y \to Y_n$ and $h : Y \to X$, note that $u \circ f_n = h$: Indeed, for $y \in Y$, we have, by~(a) and writing $y_n := f_n(y)$ and $x := h(y)$, that $T^k(x) \in C(S_n^k(y_n))$ for all $k$, i.e., $u (y_n) = x$, so $h(y) = u ( f_n(y) )$. 
Hence, by lemma~\ref{lem: commuting trianlge lemma}, $u : Y_n \to X$ is a factor. 
This finishes the proof:
Given $i \geq n$, consider $(g' , h') := (e_{ni}, f_{ni}) \circ (\underline{u}, u) : (X, T ) \to (Y_i, S_i)$. Then $(e_i , f_i ) \circ (g' , h') = (e_n, f_n ) \circ (\underline{u}, u) = (g, h)$. And $(g', h')$ is unique with this property since $(e_i, f_i)$ is monic.

($\Leftarrow$)
Now assume that $\DSef$ satisfies requirement~\ref{itm: det sys algebroidal limit of shifts} and let $(X,T)$ be category-theoretically finite. So $(X,T)$ is a colimit of an $\omega$-chain $((X_i,T_i) , (e_{i i+1} , f_{i i+1}))$ of shifts in $\DSef$. 
By lemma~\ref{lem: finite colimit is isomorphic to object in diagram}, $(X,T)$ is isomorphic to some $(X_i,T_i)$, which, qua shift, has a finite topological generator (proposition~\ref{prop: finite top gen characterization}).
\end{proof}

\begin{proof}[Proof of theorem~\ref{thm: category of det sys semialgebroidal}]
Let $\DSef$ be a full subcategory of $\detSysef$ with properties~\ref{itm: det sys algebroidal limit closure} and~\ref{itm: det sys algebroidal limit of shifts}.
By lemma~\ref{lem: det sys cat finite iff fin gen}, the category-theoretically finite objects of $\DSef$ are precisely those which are isomorphic to a shift.
It remains to show semi-algebroidality.
First, all morphisms in $\DSef$ are monic by proposition~\ref{prop: all morphisms are monic}.
Second, if $(X,T)$ is in $\DSef$, then, by~\ref{itm: det sys algebroidal limit of shifts}, it is a colimit of an $\omega$-chain of shifts in $\DSef$, i.e., of category-theoretically finite objects of $\DSef$.
Third, if $\big( (X_i,T_i) , (e_{i i+1} , f_{i i+1}) \big)_{i \in \omega}$ is an $\omega$-chain of (finite) objects in $\DSef$, then $(X,T)$, as described as in requirement~\ref{itm: det sys algebroidal limit closure}, is in $\DSef$, and by proposition~\ref{prop: every chain has colimit}, it is the colimit of the chain in $\DSef$.
\end{proof}

\section{Algebroidal categories of deterministic systems}
\label{sec: algebroidal categories of deterministic systems}

We have seen that the category $\detSysef$ fails to be algebroidal because its collection of category-theoretically finite objects---namely, the set of all shifts---is uncountable. As a result, the  \Fraisse{}-limit method cannot be applied, and indeed there cannot be a universal deterministic system. So a natural question is whether we still get algebroidal \emph{sub}categories of deterministic systems that then, further, also have universal and homogeneous systems.

As described in section~\ref{ssec: statement of the main results}, our strategy is to choose a countable class $S$ of shifts and consider the full subcategory $\DSef(S)$ of $\detSysef$ consisting of those systems that are limits of shifts in $S$. We defined $\DSef(S)$ in definition~\ref{def: DSef} and called its objects $\omega$-proshifts over $S$. Thus, intuitively, the shifts in $S$ are our basic building blocks with which we can build more complex systems, namely the $\omega$-proshifts of over $S$.

\begin{remark}
\label{rmk: ind and pro completion}
Such constructions of categories are common in mathematics, because of their desirable properties. For example, in topology, Stone spaces (aka profinite spaces) are the limits of finite discrete spaces; similarly for profinite groups and profinite graphs. In domain theory, bifinite domains are the colimits along embedding-projection pairs of finite partial orders with a least element. The general idea is that of an Ind-completion (resp., Pro-completion), i.e., the `closure' of a category under filtered colimits (resp., limits)~\cite[ch.~6]{Kashiwara2006}. Given remark~\ref{rmk: limit vs colimit}, we can use Ind- and Pro-completions interchangeably here. However, one difference in our definition is that we only consider $\omega$-chains and not any filtered diagram. We stress this by adding the prefix `$\omega$' to `proshift'.
\end{remark}

In this section, we show that this strategy indeed yields algebroidal subcategories of deterministic systems:

\begin{theorem}
\label{thm: generating algebroidal categories of det sys}
Let $S$ be a countable collection of shifts. Then $\DSef(S)$ is an algebroidal category and its category-theoretically finite objects are, up to isomorphism, exactly those that are in $S$.
Moreover, $\DSef(S)$ contains a $\DSef(S)$-universal and $S$-homogeneous deterministic system $(U,T)$ iff 
\begin{enumerate}
\item
\label{itm: joint embedding shifts}
if $X_0, X_1 \in S$, there is $Y \in S$ and there are factors $f_0 : Y \to X_0$ and $f_1 : Y \to X_1 $. 

\item
\label{itm: amalgamation shifts}
if $X,Y_0, Y_1 \in S$ and $f_0 : Y_0 \to X$ and $f_1 : Y_1 \to X$ are factors, then there is $Z \in S$ and there are factors $g_0 : Z \to Y_0$ and $g_1 : Z \to Y_1$ such that $f_0 \circ g_0 = f_1 \circ g_1$. 
\end{enumerate}
\end{theorem}

We prove the theorem in the remainder of this section. In the next section, we then apply it to the choice of $S$ as the class of shifts of finite type and the class of sofic shifts, respectively.

For the proof, we need the Curtis-Lyndon-Hedlund Theorem (see, e.g., \cite[thm.~6.2.9]{Lind1995}) in the following version.
To fix some terminology: Let $X$ be a shift over alphabet $A$. Write $B_n(X)$ for the set of all words $w \in A^n$ that occur in some $x \in X$ (i.e., there is $k$ such that $w = (x(k), \ldots , x(k + n -1))$).
For a word $w \in A^n$, we call $C_w (X) := \{ x \in X : w = (x(0), \ldots , x(n-1)) \}$ the \emph{cylinder set} of $w$ in $X$. 

\begin{lemma}[Curtis-Lyndon-Hedlund Theorem]
\label{lem: Curtis-Lyndon-Hedlund Theorem}
Let $X$ and $Y$ be shifts over alphabets $A$ and $B$, respectively.
Let $f : X \to Y$ be a continuous function that commutes with the shifts (i.e., $f \circ \sigma = \sigma \circ f$).
Then there is $N \in \omega$ such that $f$ is an \emph{$N$-sliding block code}, i.e., there is a function $\varphi : B_N(X) \to B$ (aka \emph{block map}) such that, for all $x \in X$,
\begin{align*}
f(x) (k) 
&= 
\varphi \big( ( x(k), \ldots , x(k + N-1) ) \big)
&
(k \in \omega).
\end{align*}
Moreover, $f$ also is an $n$-sliding block code for all $n \geq N$.
\end{lemma}

For completeness, we add a proof of this in appendix~\ref{app: proofs from section algeboidal categories of deterministic systems}. 
Next, we need the simple but important observation that generating new systems with shifts does \emph{not} add any new shifts.

\begin{lemma}
\label{lem: shift in S iff limit of shifts in S}
Let $S$ be a collection of shifts and let $X$ be a shift. Then 
$X$ is in $\DSef(S)$ iff $X$ is isomorphic to a system in $S$.
\end{lemma}

\begin{proof}
The right-to-left implication is trivial. For the other direction, assume $X$ is a colimit of an $\omega$-chain of shifts $X_i$ in $S$. By lemma~\ref{lem: det sys cat finite iff fin gen}, since $X$ is a shift---and hence has a finite topological generator---$X$ is category-theoretically finite in $\DSef(S)$. 
Since all morphisms in $\DSef(S)$ are monic (proposition~\ref{prop: all morphisms are monic}), lemma~\ref{lem: finite colimit is isomorphic to object in diagram} implies that $X$ hence is isomorphic to some $X_i$, which is in $S$.
\end{proof}

Finally, we need a basic category-theoretic fact: that $\DSef(S)$ is closed under taking colimits of $\omega$-chains. In other words, a colimit of systems which are colimits of shifts is itself a colimit of shifts. That is well-known for Ind-completions~\cite[ch.~6]{Kashiwara2006}. A sketch of an elementary proof is in appendix~\ref{app: proofs from section algeboidal categories of deterministic systems}.

\begin{lemma}
\label{lem: pro category closed under limits}
Let $\mathsf{C}$ be a category.
Let $A$ be a colimit of an $\omega$-chain with objects $A_i$.
Assume each $A_i$ is, in turn, a colimit of an $\omega$-chain with objects $A_i^j$, such that each $A_i^j$ is category-theoretically finite. 
Then $A$ also is a colimit of an $\omega$-chain with objects that are among the $A_i^j$'s.
\end{lemma}

Now we are ready to prove theorem~\ref{thm: generating algebroidal categories of det sys}.

\begin{proof}[Proof of theorem~\ref{thm: generating algebroidal categories of det sys}]
We want to apply theorem~\ref{thm: category of det sys semialgebroidal}, so we need to show that $\DSef(S)$ satisfies~\ref{itm: det sys algebroidal limit closure} and~\ref{itm: det sys algebroidal limit of shifts}.
The former follows from lemma~\ref{lem: pro category closed under limits}
and the latter is satisfied by construction.

Now theorem~\ref{thm: category of det sys semialgebroidal} entails that $\DSef(S)$ is a semi-algebroidal category and its category-theoretic finite objects are precisely those systems in $\DSef(S)$ that are shifts. 
By lemma~\ref{lem: shift in S iff limit of shifts in S}, these, in turn, are, up to isomorphism, exactly the elements of $S$, as needed.

To show that $\DSef(S)$ is algebroidal, we need to show that (a)~it contains, up to isomorphism, only countably many finite objects, and (b)~between any two finite objects there are only countably many morphisms. Now (a)~follows from the assumption that $S$ is countable and the just proven fact that, up to isomorphism, $S$ contains all the finite objects. Moreover, (b)~follows from lemma~\ref{lem: Curtis-Lyndon-Hedlund Theorem}: If $X$ and $Y$ are two finite objects, they are isomorphic to shifts, and between shifts there are only countably many factors, because each such factor is given by a block map, which is a finite object.

Finally, the necessary and sufficient conditions for the existence of a universal and homogeneous system simply is theorem~\ref{thm: category theoretic fraisse jonsson}.
\end{proof}

\section{The universal proshift}
\label{sec: proshifts}

Theorem~\ref{thm: generating algebroidal categories of det sys} provides a general strategy to build algebroidal categories of deterministic systems: first, choose a countable class of shifts, and then close it under (co-) limits along $\omega$-chains. As mentioned in example~\ref{exm: symbolic dynamics}, the two most studied classes of shifts are the class $\SFT$ of shifts of finite type and the class $\SOF$ of sofic shifts~\cite[ch.~2 and~3, respectively]{Lind1995}. Indeed, both $\SFT$ and $\SOF$ are countable (as they are characterized by finite graphs).
In this section, we apply our strategy to those classes of shifts and obtain our second universality result, namely theorem~\ref{thm: universal nondeterministic system} stating the existence of the universal proshift, i.e., the deterministic system that is the universal and homogeneous system both in $\DSef(\SFT)$ and in $\DSef(\SOF)$.
The proof is in subsection~\ref{ssec: the universality result for proshifts}. Then subsection~\ref{ssec: the nature of the universal proshifts} investigates the universal proshift and avenues for future work.

\subsection{Proof of theorem~\ref{thm: universal and homogeneous proshifts}}
\label{ssec: the universality result for proshifts}

First, let us give a name to $\DSef(\SFT)$ and $\DSef(\SOF)$.

\begin{definition}
\label{def: proshifts of finite type and sofic proshifts}
We call $\DSef(\SFT)$ the category of \emph{$\omega$-proshifts of finite type}, and we call $\DSef(\SOF)$ the category of \emph{sofic $\omega$-proshifts}.\footnote{One might also call them `shifts of profinite type' and `prosofic shifts', respectively, but this would falsely suggests that they are shifts (example~\ref{exm: proshift which is not shift} describes a proshift which is not a shift).}
\end{definition}

From theorem~\ref{thm: generating algebroidal categories of det sys}, we know that $\DSef(\SFT)$ and $\DSef(\SOF)$ are algebroidal. The next lemma is the key to show the joint embedding property and the amalgamation property.

\begin{lemma}
\label{lem: amalgamation for shifts of finite type}
Let $X,Y_0, Y_1$ be shifts over alphabets $A,B_0,B_1$, respectively, and let $f_0 : Y_0 \to X$ and $f_1 : Y_1 \to X$ be factors. 
For $y_0 \in Y_0$ and $y_1 \in Y_1$, let $\langle y_0 ,y_1 \rangle$ be the sequence whose $n$-th element is $(y_0 (n) , y_1 (n))$, so $\langle y_0 ,y_1 \rangle \in (B_0 \times B_1)^\omega$. 
Write 
\begin{align*}
Y_0 \times_{f_0 ,f_1} Y_1 
:= 
\big\{  
\langle y_0,y_1 \rangle 
: 
y_0 \in Y_0, y_1 \in Y_1 , f_0(y_0) = f_1(y_1)
\big\}.
\end{align*}
Then we have the following:
\begin{enumerate}
\item
\label{lem: amalgamation for shifts of finite type 1}
$Y_0 \times_{f_0 ,f_1} Y_1$ is a shift over the alphabet $B_0 \times B_1$.
\item
\label{lem: amalgamation for shifts of finite type 2}
If $Y_0$ and $Y_1$ are shifts of finite type, so is $Y_0 \times_{f_0 ,f_1} Y_1$.
\item
\label{lem: amalgamation for shifts of finite type 3}
For $i \in \{ 0,1\}$, the function
$g_i : Y_0 \times_{f_0 ,f_1} Y_1 \to Y_i$,
which maps $\langle y_0,y_1 \rangle$ to $y_i$,
is a factor.
\end{enumerate}

\end{lemma}

\begin{proof}
Write $Z := Y_0 \times_{f_0 ,f_1} Y_1$.
For $w_0 \in B_0^n$ and $w_1 \in B_1^n$, we also write $\langle w_0 , w_1 \rangle$ for the word in $(B_0 \times B_1)^n$ whose $k$-th element is $(w_0 (k) , w_1 (k))$.
For $i \in \{0,1\}$, let $F_i$ be a set of $B_i$-words such that $Y_i$ is the set of $B_i$-sequences that avoid $F_i$.\footnote{It is a basic fact that shifts can equivalently be defined via a set $F$ of avoided words~\cite[p.~5 and~27]{Lind1995}. (The special case where $F$ is finite yields the shifts of finite type.)}
Define: 
\begin{align*}
\overline{F}_0
&:=
\bigcup_{w_0 \in F_0}
\big\{
\langle w_0 , w_1 \rangle : w_1 \in B_1^{|w_0|}
\big\}
&
\overline{F}_1
&:=
\bigcup_{w_1 \in F_1}
\big\{
\langle w_0 , w_1 \rangle : w_0 \in B_0^{|w_1|}
\big\}.\footnotemark
\end{align*}\footnotetext{Here $|w|$ denotes the length of word $w$.}
By lemma~\ref{lem: Curtis-Lyndon-Hedlund Theorem}, $f_0$ and $f_1$ are $N$-sliding block codes for some sufficiently big $N$. Let $\varphi_0 : B_N(Y_0) \to A$ and $\varphi_1 : B_N(Y_1) \to A$ be their respective block maps.
For $(a, a') \in A \times A$ with $a \neq a'$, define
\begin{align*}
\overline{F}_{(a,a')} 
:=
\big\{
\langle w_0 , w_1 \rangle 
:
w_0 \in \varphi_0^{-1}(a), w_1 \in \varphi_1^{-1}(a')
\big\}
\subseteq 
(B_0 \times B_1)^N.
\end{align*}
Let $\overline{F}$ be the union of all the $\overline{F}_{(a, a')}$ for $(a,a') \in A \times A$ with $a \neq a'$.
Now it is straightforward to show
\begin{align*}
Z 
=
\big\{
z \in (B_0 \times B_1)^\omega
:
z \text{ avoids } \overline{F}_0 \cup \overline{F}_1 \cup \overline{F}  
\big\}.
\end{align*}
So~\ref{lem: amalgamation for shifts of finite type 1} follows: $Z$ is a shift qua set of sequence defined by avoiding a set of words.

Regarding~\ref{lem: amalgamation for shifts of finite type 2}, if $Y_0$ and $Y_1$ are shifts of finite type, then $F_0$ and $F_1$ can be chosen to be finite, so also $\overline{F}_0 \cup \overline{F}_1 \cup \overline{F}$ is finite, rendering $Z$ a shift of finite type, as needed.

Regarding~\ref{lem: amalgamation for shifts of finite type 3}, we show the claim for $i = 0$ (the case $i = 1$ is symmetric). It is straightforward to check that $g_0$ is continuous and surjective, and, for equivariance, we have
\begin{align*}
g_0 \big( \sigma ( \langle y_0 , y_1 \rangle ) \big)
=
g_0 \big( \langle \sigma(y_0) , \sigma(y_1) \rangle \big)
=
\sigma ( y_0 ) 
=
\sigma \big( g_0 ( \langle y_0 , y_1 \rangle ) \big),
\end{align*}
so $g_0 \circ \sigma = \sigma \circ g_0$, as needed.
\end{proof}

\begin{proposition}
\label{prop: amalgamation of shifts}
The following collections of shifts satisfy both the joint embedding property~\ref{itm: joint embedding shifts} and the amalgamation property~\ref{itm: amalgamation shifts} of theorem~\ref{thm: generating algebroidal categories of det sys}:
\begin{enumerate}
\item
the collection $\SHI$ of all shifts,

\item
the collection $\SFT$ of all shifts of finite type,

\item
the collection $\SOF$ of all sofic shifts.
\end{enumerate}
Moreover, $\SFT$ and $\SOF$ are countable.
\end{proposition}

\begin{proof}
Both $\SFT$ and $\SOF$ are countable because they can be represented by finite graphs and finite labeled graphs, respectively. All three collections contain the trivial shift (i.e., the full shift over the one-element alphabet), which is a factor of any shift. Hence the joint-embedding property follows from the amalgamation property, and it remains to show the latter.
For $\SHI$ and $\SFT$, this immediately follows from lemma~\ref{lem: amalgamation for shifts of finite type}. So let us show it for $\SOF$.

Let $X,Y_0, Y_1 \in \SOF$ with factors $f_0 : Y_0 \to X$ and $f_1 : Y_1 \to X$.
Since sofic shifts are factors of shifts of finite type, there are $Z_0, Z_1 \in \SFT$ with factors $h_0 : Z_0 \to Y_0$ and $h_1 : Z_1 \to Y_1$. Now we can apply lemma~\ref{lem: amalgamation for shifts of finite type} to $f_0 \circ h_0 : Z_0 \to X$ and $f_1 \circ h_1 : Z_1 \to X$. So we have $Z := Z_0 \times_{f_0 \circ h_0 , f_1 \circ h_1} Z_1 \in \SFT$ with factors $\pi_0 : Z \to Z_0$ and $\pi_1 : Z \to Z_1$ such that $(f_0 \circ h_0) \circ \pi_0 = (f_1 \circ h_1) \circ \pi_1$.
In particular, $Z$ is in $\SOF$ (since $\SFT \subseteq \SOF$) and we have factors
$g_0 :=  h_0 \circ \pi_0 : Z \to Y_0$
and
$g_1 :=  h_1 \circ \pi_1 : Z \to Y_1$
with 
$f_0 \circ g_0 = f_1 \circ g_1$, as needed.
\end{proof}

Hence both $\DSef(\SFT)$ and $\DSef(\SOF)$ have a universal and homogeneous object. To establish theorem~\ref{thm: universal nondeterministic system} it remains to show that these two objects are isomorphic:

\begin{proposition}
\label{prop: isomorphism for universal proshifts}
Let $(X,T)$ be the universal and homogeneous $\omega$-proshift of finite type and let $(Y,S)$ be the universal and homogeneous sofic $\omega$-proshift. Then $(X,T)$ and $(Y,S)$ are isomorphic.  
\end{proposition}

\begin{proof}
For the proof we use the concept of a saturated object (from~\cite{Droste1993}).
Let $\mathsf{C}$ be a semi-algebroidal category and let $\mathsf{D}$ be a full subcategory.
An object $U$ in $\mathsf{C}$ is \emph{$\mathsf{D}$-saturated}, if, for any objects $A$ and $B$ in $\mathsf{D}$ and morphisms $f : A \to U$ and $g : A \to B$ in $\mathsf{C}$, there is a morphism $h : B \to U$ in $\mathsf{C}$ such that $h \circ g = f$.  
It is easy to show that an object $U$ that is $\mathsf{D}$-universal and $\mathsf{D}$-homogeneous is $\mathsf{D}$-saturated.
In particular, $(X,T)$ is $\SFT$-saturated and $(Y,S)$ is $\SOF$-saturated.

Let $(X_i,T_i), f_{i i+1})_{i \in \omega}$ be an $\omega$-chain of shifts in $\SFT$ with colimit $((X,T) , f_i)_{i \in \omega}$.
Let $(Y_i,S_i), g_{i i+1})_{i \in \omega}$ be an $\omega$-chain of shifts in $\SOF$ with colimit $((Y,S) , g_i)_{i \in \omega}$.
We construct the diagram in figure~\ref{fig: diagram iso for universal proshifts} inductively as follows. 

\emph{Base step}.
Set $n_0 := 0$ and $m_0 := 0$. By the joint embedding property of $\SOF$, there is a sofic shift $(Z_0, R_0)$ with factors 
$p_0 : (Z_0,R_0) \to (X_{n_0}, T_{n_0})$
and
$q_0 : (Z_0,R_0) \to (Y_{m_0}, S_{m_0})$.
Since sofic shifts are factors of shifts of finite type, we can assume that $(Z_0,R_0)$ in fact is a shift of finite type.

\emph{Inductive step}.
Since $(X,T)$ is $\SFT$-saturated, there is a factor $u : (X,T) \to (Z_0,R_0)$ with $p_0 \circ u = f_{n_0}$. Since the shifts in $\DSef(\SFT)$ are category-theoretically finite, there in particular is $n_1 > n_0$ and a factor $k_0 : (X_{n_1} , T_{n_1}) \to (Z_0, R_0)$ such that $k_0 \circ f_{n_1} = u$. Note that
$
f_{n_0 n_1} \circ f_{n_1} 
=  
f_{n_0}
=
p_0 \circ u
=
p_0 \circ k_0 \circ f_{n_1}
$,
so, since factors are epic (and ef-pairs monic), $f_{n_0 n_1} = p_0 \circ k_0$.
Similarly, we get $m_1 > m_0$ and a factor $l_0 : (Y_{m_1} , S_{m_1}) \to (Z_0, R_0)$ such that $g_{m_0 m_1} = q_0 \circ l_0$.
Now, by the amalgamation property of $\SOF$, there is a sofic shift $(Z_1, R_1)$ with factors 
$p_1 : (Z_1,R_1) \to (X_{n_1}, T_{n_1})$
and
$q_1 : (Z_1,R_1) \to (Y_{m_1}, S_{m_1})$
such that
$k_0 \circ p_1 = l_0 \circ q_1$.
Again, we can assume that $(Z_1,R_1)$ in fact is a shift of finite type.
Now we proceed with the next layers of the diagram.
  
\begin{figure}
\begin{equation*}
\begin{tikzcd}
\vdots
&
\vdots
&
\vdots
\\
(X_{n_2}, T_{n_2})
\arrow[d, swap, "f_{n_1 n_2}"]
\arrow[dr, "k_1"]
&
(Z_2, R_2)
\arrow[l, swap, "p_2"]
\arrow[r, "q_2"]
&
(Y_{m_2}, S_{m_2})
\arrow[d, "g_{m_1 m_2}"]
\arrow[dl, swap, "l_1"]
\\
(X_{n_1}, T_{n_1})
\arrow[d, swap, "f_{n_0 n_1}"]
\arrow[dr, "k_0"]
&
(Z_1, R_1)
\arrow[l, swap, "p_1"]
\arrow[r, "q_1"]
&
(Y_{m_1}, S_{m_1})
\arrow[d, "g_{m_0 m_1}"]
\arrow[dl, swap, "l_0"]
\\
(X_{n_0}, T_{n_0})
&
(Z_0, R_0)
\arrow[l, swap, "p_0"]
\arrow[r, "q_0"]
&
(Y_{m_0}, S_{m_0})
\end{tikzcd}
\end{equation*}
\caption{Diagram to prove that $(X,T)$ and $(Y,S)$ are isomorphic.}
\label{fig: diagram iso for universal proshifts}
\end{figure}

Now, $\big( (X,T), q_i \circ k_i \circ f_{n_{i + 1}} \big)_i$
is a cone to the $\omega$-chain
$\big( (Y_{m_i},S_{m_i}), g_{m_i m_{i + 1}} \big)_i$, 
whose limit is 
$\big( (Y,S) , g_{m_i} \big)_i$.
So there is a factor 
$u : (X,T) \to (Y,S)$
such that
$g_{m_i} \circ u = q_i \circ k_i \circ f_{n_{i + 1}}$.
Similarly, 
$\big( (Y,S), p_i \circ l_i \circ g_{m_{i + 1}} \big)_i$
is a cone to the $\omega$-chain
$\big( (X_{n_i},T_{n_i}), f_{n_i n_{i + 1}} \big)_i$, 
whose limit is 
$\big( (X,T) , f_{n_i} \big)_i$.
So there is a factor 
$v : (Y,S) \to (X,T)$
such that
$f_{n_i} \circ v = p_i \circ l_i \circ g_{m_{i + 1}}$.

We claim that $v \circ u = \id_X$ and $u \circ v = \id_Y$, yielding the desired isomorphism. Indeed, some diagram-chasing shows, for any $i$, that $f_{n_i} \circ v \circ u = f_{n_i}$ and $g_{m_i} \circ u \circ v = g_{m_i}$.
\end{proof}

\subsection{The nature of the universal proshift}
\label{ssec: the nature of the universal proshifts}

We explore the categories $\DSef(\SFT)$ and $\DSef(\SOF)$ and their shared universal and homogeneous object. Then we sketch avenues for further investigation, especially in connection to automata theory.

\emph{Examples}.
To better understand the two categories of proshifts, let's consider some examples. Recall that we can think of the collection $S$ of shifts as the basic building blocks with which we build the $\omega$-proshifts in $\DSef(S)$. Lemma~\ref{lem: shift in S iff limit of shifts in S} told us that when building more complex systems, we do not generate any further shifts. In particular, a sofic shift that is not a shift of finite type---like the even shift~\cite[67]{Lind1995}---is in $\DSef(\SOF)$ but not in $\DSef(\SFT)$. Similarly, a shift that is not sofic---like the context-free shift~\cite[74]{Lind1995}---is neither in $\DSef(\SOF)$ nor in $\DSef(\SFT)$. 
However, what are examples of systems which are not basic but complex, i.e., which are not shifts but are in $\DSef(\SFT)$ and hence also in $\DSef(\SOF)$? 
Arguably the simplest such example is this:

\begin{example}
\label{exm: proshift which is not shift}
Let $X := 2^\omega$ be Cantor space and $T := \id_X$ the identity function, i.e., the trivial dynamics. 
We claim that $(X,T)$ is in $\DSef(\SFT)$ but it is not in $\SFT$, in fact, $(X,T)$ is not isomorphic to a shift at all.
The latter follows from Hedlund's theorem, since the identity function is not expansive. To show the former, we need to show that $(X,T)$ is a limit of shifts of finite type. Recall that the Cantor space $2^\omega$ is the limit of the inverse sequence $(2^i , \rho_{ij})$, where $2^i$ is the discrete space consisting of all binary strings of length $i$ with $\rho_{ij} : 2^j \to 2^i$ and $\rho_i : 2^\omega \to 2^i$ being the restriction functions (mapping a sequence $x$ to its restriction $x \restriction i$). 
We regard $2^i$ as a dynamical system with the identity function $\id$ as dynamics, so the restriction maps $\rho_{ij}$ and $\rho_i$ trivially are factors. 
Moreover, each $(2^i , \id)$ is isomorphic to a shift of finite type $X_i$: The alphabet is $2_i$ and $X_i$ consists of all the constant sequences (so the forbidden words are those $w = ab$ where $a$ and $b$ are distinct letters in $2^i$). The isomorphism $\overline{\cdot} : (2^i , \id) \to (X_i , \sigma)$ maps a string $a \in 2^i$ to the constant sequence $\overline{a} := (a_i)$ with $a_i := a$ for all $i \in \omega$. 
Hence $(X,T)$ is a limit of the shifts of finite type $X_i$.\footnote{Note, though, that the set of homeomorphisms on Cantor space that are isomorphic to subshifts of finite type is dense in the set of all homeomorphisms on Cantor space with the topology of uniform convergence~\cite{Shimomura1989}.}
\end{example}

Another example comes from odometers. They play an important role in dynamical systems theory (for an overview, see~\cite{Downarowicz2005}) and have connections to cellular automata~\cite{Coven2006}.

\begin{example}
\label{exm: universal odometer is proshift}
We show that every odometer is in $\DSef(\SFT)$ and hence also in $\DSef(\SOF)$.
An \emph{odometer} (aka adding machine or group of $s$-adic integers) is defined, in the most concrete way, as an inverse limit of finite cyclic groups, as follows. Let $s = (s_i)_{i \in \omega}$ be a sequence of positive integers such that $s_i$ divides $s_{i+1}$. Write $C_n$ for the finite cyclic group $\mathbb{Z}/n\mathbb{Z}$. Let $\pi_{i i+1} : C_{s_{i+1}} \to C_{s_i}$ be the function mapping $k$ to $k \mod s_i$ (a surjective group homomorphism).
The $s$-odometer $G_s$ then is the inverse limit of $(C_{s_i} , \pi_{i i+1})_{i \in \omega}$. So an element $x$ of $G_s$ is a sequence $(x_i)_{i \in \omega}$ with $x_i \in C_{s_i}$ and $x_i  \equiv x_{i+1} \mod s_i$. Addition is defined coordinatewise, as addition modulo $s_i$. In particular, we can consider the `$+1$ dynamics' 
$+_1 : G_s \to G_s$
which maps 
$(x_i)_{i \in \omega}$ to $(x_i + 1 \mod s_i)_{i \in \omega}$.
Then $(G_s, +_1)$ is a deterministic dynamical system in our sense. 

In particular, we can define the sequence $(s_i)_{i \in \omega}$ by $s_i := (i+1)!$. The resulting odometer $(G_s , +_1)$ is universal among the odometers (i.e., every other odometer is a factor of it). Moreover, \cite{Hochman2008} shows that the isomorphism class of $(G_s , +_1)$ is generic in the class of transitive and invertible dynamics on compact metric spaces.

To see that each odometer $(G_s, +_1)$ is in $\DSef(\SFT)$, first note that each cyclic group $C_n$, with the $+1$ dynamics $a \mapsto a + 1 \mod n$, can be regarded as a shift of finite type. Indeed, let $n := \{ 0 , \ldots , n - 1 \}$ be the alphabet and let $X_n$ be the shift over $n$ defined by avoiding the words $w = ab$ where $b \not\equiv a+1 \mod n$. Then $X_n$ contains the sequences
\begin{center}
\begin{tabular}{c c c c c c c c c c c c}
$\overline{0}$ & $:=$ & $0$ & $1$ & $2$ & $\ldots$ & $n-2$ & $n-1$ & $0$ & $1$ & $2$ & $\ldots$ \\
$\overline{1}$ & $:=$ & $1$ & $2$ & $3$ & $\ldots$ & $n-1$ & $0$ & $1$ & $2$ & $3$ & $\ldots$ \\
$\vdots$ \\
$\overline{n-1}$ & $:=$ & $n-1$ & $0$ & $1$ & $\ldots$ & $n-3$ & $n-2$ & $n-1$ & $0$ & $1$ & $\ldots$ \\
\end{tabular}
\end{center}
with the discrete topology, and $\sigma ( \overline{k} ) = \overline{k + 1 \mod n}$. So $\overline{\cdot}$ is an isomorphism between $(C_n , +_1)$ and $(X_n, \sigma)$.
Now, since $G_s$ is the limit of the $C_{s_i}$'s, $(G_s, +_1)$ is the limit of the $\omega$-chain of shifts of finite type $X_{s_i}$.
\end{example}

\emph{Properties}.
Regarding the `complexity' of the universal proshift, we can make similar arguments as for the universal nondeterministic system in section~\ref{ssec: the nature of the universal system}. Knowing that the universal proshift must have as factor every sofic shift tells us quite a lot about it.
For example, it does not have any dense orbits and, indeed, it is not topologically transitive.\footnote{Recall that a system $(X,T)$, with $X$ a topological space and $T: X \to X$ continuous, is \emph{topologically transitive}, if for every two nonempty open sets $U,V \subseteq X$, there is an integer $n \geq 0$ such that $f^n(U) \cap V \neq \emptyset$.} 
 
\begin{proposition}
\label{prop: universal proshift not transitive}
The universal proshift is not topologically transitive.
\end{proposition}

\begin{proof}
Recall, from example~\ref{exm: proshift which is not shift}, that $(2,\id)$, the two-element system with trivial dynamics, is a shift of finite type. Clearly, it is not transitive. By universality, it is a factor of the universal proshift. Because topological transitivity is preserved along factors, the universal proshift hence also cannot be transitive.
\end{proof}

This is noted, e.g., by~\cite[1128]{Hochman2012} and~\cite[2]{Doucha2025} for the generic system on Cantor space. Similarly to the nondeterministic case, this renders the universal system less interesting from a dynamical perspective. However, from a computational perspective, not having dense orbits actually is welcome: a system that transverses all of its state space does not look like a system that implements a goal-oriented algorithm.

Another important concept in dynamical systems theory is the shadowing property~\cite{Palmer2000}. Intuitively, it says that when we numerically simulate an orbit starting in $x_0$ and computing, at each step $x_n$, the next $x_{n+1}$ with some error from the true $T(x_n)$, then there is a true orbit close to the simulated orbit. Clearly, such a stability is again welcome from a computational perspective. Formally, the property is defined on a compact metric space $X$ with metric $d$ and continuous function $T: X \to X$. (But it can also be defined purely topologically.) A sequence $(x_i)$ in $X$ is a \emph{$\delta$-pseudo-orbit} if $d( x_{i+1} , T(x_i) ) < \delta$ for all $i \in \omega$. A point $x \in X$ \emph{$\epsilon$-shadows} a sequence $(x_i)$ in $X$ if $d( x_{i} , T^i(x) ) < \epsilon$ for all $i \in \omega$. Then $(X,T)$ has the \emph{shadowing property} if, for all $\epsilon > 0$, there is $\delta > 0$ such that every $\delta$-pseudo-orbit is $\epsilon$-shadowed by some point.

A result of \cite{Good2020} is that a system $(X,T)$, where $X$ is a compact totally disconnected space and $T: X \to X$ continuous, has the shadowing property iff $(X,T)$ is the inverse limit of a directed system satisfying the Mittag-Leffler condition and consisting of shifts of finite type. If the bonding maps of the inverse system are surjective, i.e., factors in our sense, then the Mittag-Leffler condition is satisfied. Hence:

\begin{corollary}[of \cite{Good2020}]
\label{cor: profinite shifts have shadowing}
All $\omega$-proshifts of finite type, i.e., the systems in $\DSef(\SFT)$, have the shadowing property. In particular, the universal proshift has the shadowing property.
\end{corollary}

It is also known that a shift has the shadowing property iff it is a shift of finite type~\cite{Walters1978}. So, for example, a sofic shift which is not a shift of finite type does not have the shadowing property. In particular, some systems in $\DSef(\SOF)$ fail to have the shadowing property, despite having a universal system with the shadowing property.

\emph{Future work}. There is much more to be investigated in the categories $\DSef(\SFT)$ and $\DSef(\SOF)$. So we now sketch four avenues for future work.

First, one way to better understand an object is to understand its automorphism group. Thus, picking up again the KPT~correspondence discussed in section~\ref{sec: literature review}, one can relate dynamical properties of the automorphism group of our universal systems to combinatorial properties of their category-theoretically finite objects. This way, one could investigate, e.g., if the automorphism group is extremely amenable (and hence uniquely ergodic) or if it admits ample generics~\cite{Kechris2014}.

Second, the latter question about genericity is particularly interesting (see again~\cite{Kechris2006, Doucha2024, Doucha2025, Hochman2008, Hochman2012, Kwiatkowska2012}). It also connects to questions about computability, as \cite{Hochman2012} shows: There are effective actions of $\mathbb{Z}^1$ on Cantor space with dense conjugacy class. However, for dimensions $d \geq 2$, there are no dense effective actions of $\mathbb{Z}^d$ on Cantor space, while there are non-effective ones.
Other computability-theoretic aspects of symbolic dynamics are studied in~\cite{Galatolo2010}, and in~\cite{Barbieri2025} also inverse limits of shifts of finite type occur.
In this vein, one might also ask for a concrete description of the universal proshift, not just an existence results (cf.~\cite{Akin2008}).

Third, one can connect to the construction of compacta \cite{Mahavier2004, Ingram2012, Ingram2012a} using inverse limits along multifunctions, especially through their study via shift maps~\cite{Kawamura2018}.

Fourth, given the close connection between sofic shifts and automata, it is natural to investigate if there is a similarly fruitful connection between our sofic $\omega$-proshifts and profinite automata~\cite{Rowland2017}. Here duality-theoretic tools might also be useful (see \cite{Gehrke2016} and \cite[ch.~8]{Gehrke2024}): they consider the Boolean algebra of regular languages (i.e., those languages that are recognized by some automaton) over a given alphabet and with certain residuation operations, and they show that dual to this Boolean algebra is the free profinite monoid. For further connections between profinite semigroups and shift spaces, see~\cite{Almeida2025}.

\section{Conclusion}
\label{sec: conclusion}

We end with a brief summary and some direction for future work.

\emph{Summary}.
Motivated by the search for a universal analog computer, we considered possibly nondeterministic systems and identified four senses of universality. We provided an equivalent domain-theoretic description of the systems, and we used the \Fraisse{}-limit method to show the existence of a universal and homogeneous nondeterministic system. The subclass of deterministic systems cannot have such a universal system, but we introduced $\omega$-proshifts and showed that both the $\omega$-proshifts of finite type and the sofic $\omega$-proshifts have a shared universal and homogeneous system, which has the shadowing property.

\emph{Future work}. We already mentioned open questions on $\omega$-proshifts: e.g., exploring their automorphism groups, their genericity, their computability, or their connections to compacta. Particularly exciting is to extend the connection between sofic shifts and automata to the profinite---including links to duality theory.

Further questions include, as mentioned, generalizing our results from discrete time to other monoid actions (in particular, see~\cite{Kechris2014, Doucha2025}). Quite generally, it is worth better understanding the structure of the different categories of systems---and their corresponding categories of domains (also see~\cite{Weigert2023}).

Moreover, it is pertinent to clarifying the connection between our existence result on universal analog computation (which includes `learning systems' like neural networks) to results like the No-Free-Lunch theorems stating the non-existence of a universal learner (see~\cite{Sterkenburg2018} for a study of universal prediction methods).

Finally, at a philosophical level, one should clarify the implication of universality and genericity results on the question---which we had to leave open---whether a dynamical system counts as an analog computer. We have mentioned the density and genericity results about actions on the Cantor space, which are within a longer tradition in ergodic theory on understanding which properties a `typical' dynamical system has (see, e.g., \cite{Glasner2001} or \cite{Hochman2008} for a short history). Beyond deterministic systems and closer to nondeterministic systems, \textcite{Erdoes1963} probabilistically constructed a countable graph that, with probability~1, yields \Fraisse{}'s universal homogenous graph; \cite{Droste2007} provides a related probabilistic construction for universal domains in domain theory; and \cite[sec.~3.1]{Geschke2022} provide a concrete model of the projectively universal profinite graph.
So one may try to make precise the question whether being an analog computer (whatever this means exactly) is a typical property. Does the fact that the generic automorphism on Cantor space is effective have any bearing on this question? What about the fact that this fails for higher dimensions? Note that algorithmic randomness also studies such questions of `typicality'~\cite{Downey2010}.

\appendix

\section{Proofs}
\label{app: proofs}

\subsection{Proofs from section~\ref{sec: categories of dynamical systems}}
\label{app: proofs from section categories of dynamical systems}

\begin{remark}[Zero-dimensionalizing and compactifying]
\label{rmk: zerodimensionalize and compactify}
Starting with a `Polish system' $(X,T)$, i.e., where $X$ is a Polish space and $T:X \to X$ is a continuous function, we want to construct a---in a sense best possible---system $(Y,S)$, with $Y$ a compact and zero-dimensional Polish space and $S : Y \to Y$ continuous, together with an injective and equivariant function $\eta : X \to Y$ that is continuous with respect to a finer Polish topology on $X$.

We use a result from~\textcite{Dahlqvist2016} on how to zero-dimensionalize and compactify Polish spaces. To state this precisely, let us introduce their terminology: 
A \emph{based Polish space} is a pair $(X, \mathcal{B})$ where $X$ is a Polish space and $\mathcal{B}$ is a countable base for the topology of $X$.
A \emph{base-preserving map} $f$ from $(X, \mathcal{B})$ to $(Y, \mathcal{C})$ is a function $f : X \to Y$ such that $f^{-1}(C) \in \mathcal{B}$ for all $C \in \mathcal{C}$ (so $f$ is continuous).
A compact zero-dimensional Polish space $X$ can canonically be regarded as a based Polish space by taking the collection $\Clp(X)$ of all clopen subsets of $X$ as the base.
Next,~\cite{Dahlqvist2016} consider the following categories: 
\begin{itemize}
\item
$\mathsf{Pol}_\mathsf{cz}$: compact zero-dimensional Polish spaces with continuous maps.
\item
$\mathsf{Pol}^\flat_z$: based Polish spaces $(X,\mathcal{B})$, where $X$ is zero-dimensional and $\mathcal{B}$ is a Boolean algebra, with base-preserving maps.
\item
$\mathsf{Pol}^\flat$: based Polish spaces with base-preserving maps.
\end{itemize}
Finally,~\cite{Dahlqvist2016} show that (1)~$\mathsf{Pol}_\mathsf{cz}$ is a reflective subcategory of $\mathsf{Pol}^\flat_z$ with reflector $c$ and (2)~$\mathsf{Pol}^\flat_z$, in turn, is a coreflective subcategory of $\mathsf{Pol}^\flat$ with coreflector $z$. 
For~(1), given a based space $(X, \mathcal{B})$ in $\mathsf{Pol}^\flat_z$, the compactification $c (X, \mathcal{B})$ is the compact zero-dimensional space $Y$ obtained as the limit of the diagram of finite spaces $P$, where $P$ is a finite partition of $X$ with partition cells in $\mathcal{B}$. The unit $\eta_{(X, \mathcal{B})} : (X, \mathcal{B}) \to Y$ is the canonical embedding of $X$ into its compactification $Y$.
For~(2), given a based space $(X, \mathcal{B})$ in $\mathsf{Pol}^\flat$, the zero-dimensionalization $z (X, \mathcal{B})$ is the based space $(X, \overline{\mathcal{B}})$ on the same underlying set $X$ but whose topology is generated by the base $\overline{\mathcal{B}}$, which is the Boolean algebra generated by $\mathcal{B}$. The counit $\epsilon_{(X, \mathcal{B})} : z (X, \mathcal{B}) \to (X, \mathcal{B})$ simply is the identity function on $X$.

Now, we use this result to turn a Polish system $(X,T)$ into a deterministic system in our sense:
\begin{itemize}
\item
\emph{Step~0}. Pick a countable base $\mathcal{B}$ of $X$. For example, if $X = \mathbb{R}^n$, pick the open balls centered on rational vectors with rational radii.
Next close $\mathcal{B}$ under the dynamics, i.e., consider $\mathcal{B}_T := \{ T^{-k} (B) : B \in \mathcal{B} , k \in \omega \}$.
Then $(X, \mathcal{B}_T)$ is a based Polish space. (This step is unnecessary if one thinks that Polish systems should make their choice of dynamically closed base explicit.)

\item
\emph{Step~1}. Take the zero-dimensionalization $z(X, \mathcal{B}_T) = (X, \overline{\mathcal{B}_T})$. Note that, by construction, $T : z(X, \mathcal{B}_T) \to (X, \mathcal{B}_T)$ is base-preserving. So, by the adjunction, it lifts to a dynamics $T : z(X, \mathcal{B}_T) \to z(X, \mathcal{B}_T)$.

\item
\emph{Step~2}. Compactify $z(X,\mathcal{B}_T)$ into $Y := c(z(X,\mathcal{B}_T))$, with the embedding $\eta_X :=  \eta_{z(X, \mathcal{B}_T)} : X \to Y$. Consider the map 
$\eta_X \circ T : X \to Y$, i.e., viewing the result of the dynamics $T$ in the larger, compactified space $Y$. By the adjunction, there is a unique commuting map $c(T)$:
\begin{equation*}
\begin{tikzcd}
z(X,\mathcal{B}_T)
\arrow[d, "T"]
\arrow[r, "\eta_X"]
&
c ( z(X, \mathcal{B}_T) )
\arrow[d, "c(T)"]
\\
z(X,\mathcal{B}_T)
\arrow[r, "\eta_X"]
&
c ( z(X, \mathcal{B}_T) )
\end{tikzcd}
\end{equation*}
So $S:= c(T) : Y \to Y$ extends the dynamics $T$ on $X$.
\end{itemize}
In sum, we have the larger system $(Y,S)$ and an injective function $\eta_X : X \to Y$, which is equivariant by the above diagram and continuous with respect to the finer Polish topology on $X$ generated by $\overline{\mathcal{B}_T}$. The fact that $c$ and $z$ are adjoint functors provides a sense in which this $(Y,S)$ is optimal.\footnote{Similar ideas are used in \cite[sec.~5.3]{Hornischer2021}.}
\end{remark}

\begin{lemma}
\label{lem: comp of tcu multifunctions again tcu}
Let $F : X \rightrightarrows Y$ and $G : Y \rightrightarrows Z$ be closed-valued and upper-hemicontinuous multifunctions between compact Hausdorff spaces. Then $G \circ F : X \rightrightarrows Z$ is again such a function. If both $F$ and $G$ are total, so is $G \circ F$.
\end{lemma}

\begin{proof}
Closed-valued: For $x \in X$, $\varphi(x)$ is closed, so the image $\psi[\varphi(x)]$ is closed by the corollary to the closed graph theorem.
Upper-hemicontinuous: The composition of upper-hemicontinuous multifunctions is again upper-hemicontinuous~\cite[thm.~17.23]{Aliprantis2006}.
If $F$ and $G$ are total, then, for $x \in X$, $\varphi (x)$ is nonempty, so $\psi \circ \varphi (x) = \psi [ \varphi (x) ] \neq \emptyset$. 
\end{proof}

\begin{proof}[Proof of proposition~\ref{prop: composition of system morphisms}]
Assume 
$\varphi : (X,T) \to (Y,S)$ 
and
$\psi : (Y,S) \to (Z,R)$
are morphisms. 
By lemma~\ref{lem: comp of tcu multifunctions again tcu}, $\psi \circ \varphi$ is again closed-valued and upper-hemicontinuous. So it remains to check the equivariance conditions.

Concerning~\ref{itm: nondet equivariance hom forth}, if $x \xrightarrow{T} x'$, there is $y,y' \in Y$ with $x \xrightarrow{\varphi} y$, $x' \xrightarrow{\varphi} y'$, and $y \xrightarrow{S} y'$, so there is $z,z' \in  Z$ with $y \xrightarrow{\psi} z$, $y' \xrightarrow{\psi} z'$, and $z \xrightarrow{R} z'$, hence also $x \xrightarrow{\psi \circ \varphi} z$, $x' \xrightarrow{\psi \circ \varphi} z'$, and $z \xrightarrow{R} z'$, as needed.

Concerning~\ref{itm: nondet equivariance hom back}, if $z \xrightarrow{R} z'$, there is $y,y' \in Y$ with $y \xrightarrow{\psi} z$, $y' \xrightarrow{\psi} z'$, and $y \xrightarrow{S} y'$, so there is $x,x' \in  X$ with $x \xrightarrow{\varphi} y$, $x' \xrightarrow{\varphi} y'$, and $x \xrightarrow{T} x'$, hence also $x \xrightarrow{\psi \circ \varphi} z$, $x' \xrightarrow{\psi \circ \varphi} z'$, and $x \xrightarrow{T} x'$, as needed.

Now assume that $\varphi$ and $\psi$ are factors. By the above, $\psi \circ \varphi$ is a morphism, and, qua composition of continuous surjections, it again is a continuous surjection, so it remains to show~\ref{itm: nondet equivariance factor}. 
So assume 
$x \xrightarrow{\psi \circ \varphi} z$ and $x \xrightarrow{T} x'$. So there is $y$ with $x \xrightarrow{\varphi} y$ and $y \xrightarrow{\psi} z$. Since $\varphi$ satisfies~\ref{itm: nondet equivariance factor}, there is $y'$ with 
$y \xrightarrow{S} y'$ and $x' \xrightarrow{\varphi} y'$. Since $\psi$ satisfies~\ref{itm: nondet equivariance factor}, there is $z'$ with 
$z \xrightarrow{R} z'$ and $y' \xrightarrow{\psi} z'$. Hence also $x' \xrightarrow{\psi \circ \varphi} z'$, as needed.

Now assume that $\varphi$ and $\psi$ are embeddings. By the above, $\psi \circ \varphi$ is a morphism and total, hence it remains to show that it is partition-injective and satisfies~\ref{itm: nondet equivariance embedding}.
For the former, we need to show that $\{ \psi \circ \varphi (x) : x \in X \}$ partitions $Z$. Indeed,
if $x \neq x'$, then it is readily checked that $\psi \circ \varphi (x) \cap \psi \circ \varphi (x') = \emptyset$; and if $z \in Z$, then there is $y \in Y$ such that $z \in \psi(y)$, and hence there is $x \in X$ such that $y \in \varphi(x)$, so $z \in \varphi \circ \psi(x)$. 
For~\ref{itm: nondet equivariance embedding}, assume 
$x \xrightarrow{\psi \circ \varphi} z$ and $z \xrightarrow{R} z'$. So there is $y$ with $x \xrightarrow{\varphi} y$ and $y \xrightarrow{\psi} z$. Since $\psi$ satisfies~\ref{itm: nondet equivariance embedding}, there is $y'$ with 
$y \xrightarrow{S} y'$ and $y' \xrightarrow{\psi} z'$. Since $\varphi$ satisfies~\ref{itm: nondet equivariance embedding}, there is $x'$ with 
$x \xrightarrow{T} x'$ and $x' \xrightarrow{\varphi} y'$. Hence also $x' \xrightarrow{\psi \circ \varphi} z'$, as needed.

Finally, it is well-known that relation composition, which is given by $\circ$, is associative with unit the identity function. 
\end{proof}

\begin{proof}[Proof of proposition~\ref{prop: ef pairs determine each other}]
Ad~\ref{prop: ef pairs determine each other 1}.
For uniqueness, if $\underline{f}$ and $\underline{f}'$ are two embeddings $(X,T) \to (Y,S)$ such that $(\underline{f},f)$ and $(\underline{f}',f)$ are ef-pairs, then, for $x \in X$, we have
$\underline{f} (x)  
=  
\underline{f} \big[ f \circ \underline{f}'  (x) \big] 
= 
\underline{f} \circ f [ \underline{f}'  (x) ] 
= 
\underline{f}' (x)$.

For existence, we show that the multifunction $e : X \rightrightarrows Y$ given by $e(x) := f^{-1} (x)$ is such that $(e,f)$ is an ef-pair. 

First, we show that $e$ is an embedding.
Since $f$ is surjective, each $e(x)$ is nonempty, so $e$ is total, and since the graph of $f$ is closed, so is the graph of $e = f^{-1}$.
By construction, $\{ e(x) : x \in X \}$ partitions $Y$.
Regarding equivarience~\ref{itm: nondet equivariance hom forth}, assume $x \xrightarrow{T} x'$. Since $f$ satisfies~\ref{itm: nondet equivariance hom back}, there is $y,y' \in Y$ such that $y \xrightarrow{f} x$,  $y' \xrightarrow{f} x'$, and $y \xrightarrow{S} y'$, hence also $x \xrightarrow{e} y$,  $x' \xrightarrow{e} y'$, as needed.
Regarding~\ref{itm: nondet equivariance embedding}, if $x \xrightarrow{e} y$ and $y \xrightarrow{S} y'$, then $y \xrightarrow{f} x$, so, since $f$ satisfies~\ref{itm: nondet equivariance factor}, there is $x'$ with $x \xrightarrow{T} x'$ and $y' \xrightarrow{f} x'$, so also $x' \xrightarrow{e} y'$, as needed.  

Second, we show properties~\ref{def: embedding-factor pair 1 embed then factor is identity} and~\ref{def: embedding-factor pair 2 factor then embed is subset}:
For $x \in X$, since $f$ is surjective, $f \circ e (x) = f ( f^{-1} (x) ) =  \{ x \}$. 
For $y \in Y$, trivially $y \in f^{-1} ( f (y) ) = e \circ f (y)$.

Ad~\ref{prop: ef pairs determine each other 2}.
For uniqueness, if $\overline{e}$ and $\overline{e}'$ are two factors $(Y,S) \to (X,T)$ such that $(e, \overline{e})$ and $(e, \overline{e}')$ are ef-pairs, then, for $y \in Y$, we have
$y \in e \circ \overline{e} (y)$
and
$y \in e \circ \overline{e}' (y)$,
so there is $x \in \overline{e} (y)$ and $x' \in \overline{e}' (y)$ with 
$y \in e (x)$ and $y \in e(x')$. Since $e$ is partition-injective, $x = x'$. So, since $\overline{e}$ and $\overline{e}'$ are functions, 
$\overline{e} (y) = x = x' = \overline{e}' (y)$.

For existence, we show that the (multi)function $f : Y \rightrightarrows X$, given by $f(y) := $~the $x$ with $y \in e(x)$, is such that $(e,f)$ is an ef-pair.
Note that, since $e$ is partition-injective, $f$ is indeed a function. 
Also note that, for any $A \subseteq X$, we have
$f^{-1} (A) = e [ A ]$:
for $y \in Y$, we have $y \in f^{-1} (A)$ iff $f(y) \in A$ iff the $x$ with $y \in e(x)$ is in $A$ iff there is $x \in A$ with $y \in e(x)$ iff $y \in e[A]$.

First, we show that $f$ is a factor.
To show that $f$ is surjective, let $x \in X$. Since $e$ is total, pick $y \in e(x)$. Then $f(y) = x$.
Also, $f$ is continuous: if $A \subseteq X$ is closed, then $f^{-1} (A) = e [ A ]$ is closed, since, by the corollary to the closed graph theorem, the image of a closed set is closed.
Regarding equivarience~\ref{itm: nondet equivariance hom back}, if $x \xrightarrow{T} x'$, then, since $e$ satisfies~\ref{itm: nondet equivariance hom forth}, there is $y,y' \in Y$ with $x \xrightarrow{e} y$, $x' \xrightarrow{e} y'$, and $y \xrightarrow{S} y'$, so also $y \xrightarrow{f} x$ and $y' \xrightarrow{f} x'$, as needed.
For~\ref{itm: nondet equivariance factor}, if $y \xrightarrow{f} x$ and $y \xrightarrow{S} y'$, then $x \xrightarrow{e} y$, so, since $e$ is an embedding, there is $x'$ with $x \xrightarrow{T} x'$ and $x' \xrightarrow{e} y'$, so also $y' \xrightarrow{f} x'$, as needed.

Second, we show properties~\ref{def: embedding-factor pair 1 embed then factor is identity} and~\ref{def: embedding-factor pair 2 factor then embed is subset}:
For $x \in X$, since $f$ is surjective, 
$f \circ e (x) = f [ f^{-1} (x) ] =  \{ x \}$.
For $y \in Y$, we have, by definition, that $f(y)$ is the $x \in X$ with $y \in e(x)$, so $y \in e(x) = e \circ f (y)$.
\end{proof}

\begin{proposition}
\label{prop: isomorphism in Sysef}
Let $(e,f) : (X,T) \to (Y,S)$ be an ef-pair. The following are equivalent.
\begin{enumerate}
\item
\label{prop: isomorphism in Sysef 1}
$(e,f) : (X,T) \to (Y,S)$ is an isomorphism in $\allSysef$
\item
\label{prop: isomorphism in Sysef 2}
$e : X \to Y$ is a homeomorphism with inverse $f : Y \to X$ such that $e \circ T = S \circ e$ (and hence also $f \circ S = T \circ f$). 
\end{enumerate}
\end{proposition}

\begin{proof}
(\ref{prop: isomorphism in Sysef 1})$\Rightarrow$(\ref{prop: isomorphism in Sysef 2}). 
Assume $(e,f) : (X,T) \to (Y,S)$ is an isomorphism, with inverse $(e',f') : (Y,s) \to (X,T)$.
We show that $e$ is a continuous bijective function $X \to Y$ such that $e \circ T = S \circ e$ and $e^{-1} = f$. (Recall that bijective continuous functions on compact Hausdorff spaces are homeomorphisms; and if $e \circ T = S \circ e$, then by `multiplying' $e^{-1}$ to the left and to the right, $e^{-1} e T e^{-1} = e^{-1} S  e e^{-1}$, so $T e^{-1} = e^{-1} S$.)

First, $e$ is a function, because if $y \neq y'$ were in $e(x)$, then both $e'(y)$ and $e'(y')$ are, by totality, nonempty subsets of $e' \circ e(x) = \{ x \}$, so $e'(y) = \{ x \}  = e'(y')$, contradicting that $e'$ is partition-injective.
And $e$ is bijective: if $x \neq x'$, then $e(x) \cap e(x') = \emptyset$, and if $y \in Y$, then it is in some $e(x)$, so $e(x) = y$.
Further, as a upper-hemicontinuous multifunction, $e$ is continuous.  
And $e^{-1} = f$ because, given $y \in Y$, we need to show that $x := f(y)$ is such that $e(x) = y$; indeed, since $f$ is the factor corresponding to $e$, we have $y \in e (f(y))$, and since $e$ is a function, we actually have $y = e(f(y)) = e(x)$.

Finally, we check, for $x \in X$, that $e \circ T (x) = S \circ e (x)$.
($\subseteq$)
If $y' \in e \circ T (x)$, there is $x' \in X$ with $x \xrightarrow{T} x'$ and $x' \xrightarrow{e} y'$, so, since $e$ is a morphism, there is $y_0,y_1 \in Y$ such that $x \xrightarrow{e} y_0$, $x' \xrightarrow{e} y_1$, and $y_0 \xrightarrow{S} y_1$, so, since $e$ is a function, $y' = y_1 \in S \circ e (x)$.
($\supseteq$)
If $y' \in S \circ e (x)$, there is $y \in Y$ with $x \xrightarrow{e} y$ and $y \xrightarrow{S} y'$, so, since $e$ is an embedding, there is $x' \in X$ such that $x \xrightarrow{T} x'$ and $x' \xrightarrow{e} y'$, so $y' \in e \circ T (x)$.

(\ref{prop: isomorphism in Sysef 2})$\Rightarrow$(\ref{prop: isomorphism in Sysef 1}).
Assume $e : X \to Y$ is a homeomorphism with inverse $f$ and $e \circ T = S \circ e$ (so, as noted, also $T \circ f = f \circ S$). We show that $(e,f) : (X,T) \to (Y,S)$ is an isomorphism in $\allSysef$.
It suffices to show that both $(e,f)$ and $(f,e)$ are ef-pairs: Then 
$(f,e) \circ (e,f) = (f \circ e, f \circ e ) = (\id_X , \id_X) = \id_{(X,T)}$
and similarly
$(e,f) \circ (f,e) = \id_{(Y,S)}$. 

Indeed, qua continuous functions, both $e$ and $f$ are closed-valued and upper-hemicontinuous multifunctions. 
We show that $e$ is both an embedding and a factor: Qua bijective continuous function, it is both a surjective function and a total partition-injective multifunction.
It suffices to show equivariance~\ref{itm: nondet equivariance factor} and~\ref{itm: nondet equivariance embedding} (which implies~\ref{itm: nondet equivariance hom forth} and~\ref{itm: nondet equivariance hom back}).
For~\ref{itm: nondet equivariance factor}, if $x \xrightarrow{e} y$ and $x \xrightarrow{T} x'$, then $y' := e(x') \in e \circ T (x) = S \circ e(x)$, so there is $y_0$ such that $x \xrightarrow{e} y_0$ and $y_0 \xrightarrow{S} y'$, hence, since $e$ is a function $y_0 = y$, as needed.
For~\ref{itm: nondet equivariance embedding}, if $x \xrightarrow{e} y$ and $y \xrightarrow{S} y'$, then $y' \in S \circ e(x) = e \circ T (x)$, so there is $x'$ with $x \xrightarrow{T} x'$ and $x' \xrightarrow{e} y'$, as needed.
Similarly, we show that $f$ is both an embedding and a factor (now using $T \circ f = f \circ S$).
Moreover, $f \circ e = \id_X$ and $e \circ f = \id_Y$, so both $(e,f)$ and $(f,e)$ are ef-pairs.
\end{proof}

\subsection{Proofs from section~\ref{sec: coalgebra and domain theory}}
\label{app: coalgebra and domain theory}

Here we prove theorem~\ref{thm: Sef equivalent to domain theoretic category}. 
We call a poset $P$ \emph{co-atomic} if each element is the greatest lower bound of its co-atoms, i.e., for all $x \in A$, we have $x = \bigwedge \CoAt(x)$.\footnote{Note that a function between co-atomic partial orders that preserves arbitrary infima is co-atomic, so co-atomicity can be seen as a weakening of infima preservation (which, for monotone functions on complete lattices, is equivalent to having a lower adjoint).}

The next two propositions are fairly standard (cf.\ e.g.~\cite{Edalat1995}), but, for completeness, we still add proofs.

\begin{proposition}
\label{prop: closed subsets algebraic lattice}
Let $X$ be a compact and zero-dimensional Polish space. Then $\F(X)$, i.e., the set of closed subsets of $X$ ordered by reverse inclusion, is an algebraic lattice. Its compact elements (in the order-theoretic sense) are the clopen subsets of $X$. Moreover, $\F(X)$ is co-atomic, with co-atoms being the singletons $\{ x \}$.
\end{proposition}

\begin{proof}
Least upper bounds in $\F(X)$ are given by intersections, so $\F(X)$ is a complete lattice. (The greatest lower bounds hence are closures of unions.) 
The greatest element is $\emptyset$, hence, since singletons are closed, each $\{x\} \in \F(X)$ is a co-atom, and every co-atom is a issingleton. If $A \in \F(X)$, then $\CoAt(A) := \big\{ \{a\} : a \in A \big\}$, so $A = \bigcup \CoAt(A)$, and since $A$ is already closed, it is the greatest lower bound of $\CoAt(A)$.

It remains to show that $\F(X)$ is algebraic, with the set $K(\F(X))$ of compact elements being the set $\Clp(X)$ of clopen subsets of $X$. We start by showing, for $A \in \F(X)$, that $\mathcal{F} := \{ B \in \Clp(X) : B \leq A \}$ is directed with $\bigcap \mathcal{F} = A$.

Indeed, $\mathcal{F}$ is nonempty: If $A = X$, then it is clopen, so $A \in \mathcal{F}$. If $A \neq X$, let $x \in X$ with $x \not\in A$. Since $X$ is normal and $A$ closed, there are disjoint opens $U$ and $V$ with $x \in U$ and $A \subseteq V$. Since the clopens form a base, there in particular is a clopen $B$ with $x \in B \subseteq U$ and $A \subseteq V \subseteq U^c \subseteq B^c$. So $B^c \in \mathcal{F}$.
Moreover, $\mathcal{F}$ is closed under intersection, i.e., least upper bounds, hence directed.
Finally, by construction, $A \subseteq \bigcap \mathcal{F}$. For the other inclusion, if $x \in X \setminus A$, there is, as above, a clopen $B$ such that $x \in B$ and $A \subseteq B^c$, so $B^c \in \mathcal{F}$, hence $x \not\in \bigcap \mathcal{F}$.

Hence it remains to show $K(\F(X)) = \Clp(X)$.
If $A \in K(\F(X))$, then $A$ is clopen: Indeed, we already know that $\mathcal{F} := \{ B \in \Clp(X) : B \leq A \}$ is directed with $\bigcap \mathcal{F} \subseteq A$, so there is $B \in \mathcal{F}$ with $B \subseteq A$, so $B \in \Clp(X)$ with $B \supseteq A$ and $B \subseteq A$, i.e., $A = B$ is clopen.
For the other direction, assume $A$ is clopen and show that $A$ is order-theoretically compact. So let $\mathcal{D} \subseteq \F(X)$ be directed with $\bigcap \mathcal{D}  \subseteq A$. We need to find $B \in \mathcal{D}$ with $B \subseteq A$.
Note that $\{ B^c : B \in \mathcal{D} \}$ is an open cover of $A^c$: if $x \not\in A$, then $x \not\in \bigcap \mathcal{D}$, so $x \not\in B$ for some $B \in \mathcal{D}$. Now, since $A^c$ is a closed subset of a compact space, there is a finite subcover, say, $\{ B_1^c , \ldots , B_n^c \}$. By directedness, let $B \in \mathcal{D}$ with $B \subseteq B_k$ for $k = 1 , \ldots , n$. Then $B \subseteq A$: If $x \not\in A$, then $x \in B_k^c$ for some $k$, so $x \not\in B$.
\end{proof}

\begin{proposition}
\label{prop: up hem cont and closed val multifunction lift to scott cont}
If $T : X \rightrightarrows Y$ is an upper-hemicontinuous and closed-valued multifunction between compact and zero-dimensional Polish spaces, then $\F(T) : \F(X) \to \F(Y)$ mapping $A$ to $T[A]$ is a Scott-continuous and co-atomic function.
\end{proposition}

\begin{proof}
Note that $\F(T)$ is well-defined: by the corollary to the closed graph theorem, $T[A]$ is a closed subset of $Y$.
It is monotone: if $A \supseteq B$, then $T[A] \supseteq T[B]$. And $\F(T)$ is co-atomic: For $A \in \F(X)$, by definition
$\F(T) (A) = T[A] = \bigcup \{ T(x) : x \in A \}$, and, since $\CoAt(A) = \big\{ \{ x \} : x \in A \big\}$, this is further identical to $\bigcup \F(T) [ \CoAt(A) ] $. So, since $\F(T) (A)$ already is closed, we have $\F(T)(A) = \bigwedge \F(T) [ \CoAt(A) ]$.

To show Scott-continuity, let $\mathcal{F} \subseteq \F(X)$ be directed and show
\begin{align*}
T \big[ \bigcap \mathcal{F} \big]
=
\bigcap_{A \in \mathcal{F}} T[A].
\end{align*}

($\subseteq$) 
For $A \in \mathcal{F}$, we have $\bigcap \mathcal{F} \subseteq A$, so, by monotonicity, 
$T \big[ \bigcap \mathcal{F} \big] \subseteq T[A]$.
Hence 
$T \big[ \bigcap \mathcal{F} \big] \subseteq \bigcap_{A \in \mathcal{F}} T[A]$.

($\supseteq$)  
Assume $y \in \bigcap_{A \in \mathcal{F}} T[A]$.
Then, for each $A \in \mathcal{F}$, there is $x_A \in A$ with $y \in T(x_A)$. 
So the set
$C_A := \{ x \in X : x \in A \text{ and } y \in T(x) \}$
is nonempty.
It also is closed: it is the intersection of $A$ and $\{ x \in X : y \in T(x) \}$, and the latter also is closed, since the graph of $T$ is closed.
Moreover, if $A \supseteq B$, then $C_A \supseteq C_B$. Hence $\{ C_A : A \in \mathcal{F} \}$ is a collection of closed subsets of $X$ with the finite intersection property.
By compactness of $X$, the intersection of the $C_A$ is nonempty, so let $x \in \bigcap_{A \in \mathcal{F}} C_A$. Then $x \in \bigcap \mathcal{F}$ and $y \in T(x)$ (note that there is some $A \in \mathcal{F}$), so $y \in T \big[ \bigcap \mathcal{F} \big]$, as needed.
\end{proof}

Write $\czPolm$ for the category of compact and zero-dimensional Polish spaces together with upper-hemicontinuous and closed-valued multifunctions.
Recall that $\bAlgc$ is the category of algebraic lattices whose compact elements form a countable sublattice that is a Boolean algebra together with Scott-continuous and co-atomic functions.

\begin{proposition}
\label{prop: F categorical equivalence}
$\F : \czPolm \to \bAlgc$ is a functor establishing a categorical equivalence.
\end{proposition}

In particular, this shows that the algebraic lattices in $\bAlgc$ are co-atomic.

\begin{proof}
By proposition~\ref{prop: closed subsets algebraic lattice}, if $X$ is a compact and zero-dimensional Polish space, then $\F(X)$ is indeed an algebraic lattice whose compact element are $\Clp(X)$, which is a countable sublattice and a Boolean algebra.
By proposition~\ref{prop: up hem cont and closed val multifunction lift to scott cont}, $\F$ sends upper-hemicontinuous and closed-valued multifunctions to Scott-continuous and co-atomic functions.
And $\F$ preserves composition:
$\F (S \circ T) (A) =
S \circ T [A]
=
S [ T [A] ]
=
\F(S) (T[A])
= 
\F(S) \circ \F(T) (A)$.

Hence $\F$ is a functor. We show that it is 
(1)~essentially surjective, 
(2)~faithful, and 
(3)~full. 

Ad~(1). Let $L$ be an algebraic lattice whose set $K(L)$ of compact elements is a countable sublattice of $L$ and a Boolean algebra. Let $X$ be the Stone space corresponding to $K(L)$ under Stone duality. Then $X$ is a compact and zero-dimensional Polish space and $\Clp(X)$ is isomorphic to $K(L)$. Hence also the algebraic lattice $\F(X)$ (whose compact elements are $\Clp(X)$) is isomorphic to the algebraic lattice $L$ (whose compact elements are $K(L)$).

Ad~(2). If $S,T : X \rightrightarrows Y$ are two upper-hemicontinuous and closed-valued multifunctions between compact and zero-dimensional Polish spaces with $S \neq T$, then there is $x \in X$ with $S(x) \neq T(x)$, hence $\F(S) (\{ x \} ) \neq \F(T) ( \{ x \} )$, so $\F(S) \neq \F(T)$.

Ad~(3). Let $X$ and $Y$ be compact and zero-dimensional Polish spaces and let $f : \F(X) \to \F(Y)$ be a Scott-continuous and co-atomic function. Define the multifunction $T : X \rightrightarrows Y$ by $T(x) := f ( \{ x \} )$. We show that $T$ is upper-hemicontinuous and closed-valued with $\F(T) = f$.

By construction, $T(x) \in F(Y)$, so $T$ is closed-valued. For upper-hemicontinuity, let $x \in X$ and $V \subseteq Y$ open with $T(x) \subseteq V$. We need to find an open $U \subseteq X$ with $x \in U$ and $T[U] \subseteq V$.

We first note that it is enough to show this for clopen $V$: For a (merely) open $V$, it is, by zero-dimensionality, the union of clopen $V_i$, which hence form an open cover of the closed set $T(x)$, so there is a finite subcover $V_{i_1} , \ldots , V_{i_n}$ of $T(x)$. Hence, for the clopen $V^* :=  V_{i_1} \cup \ldots \cup V_{i_n}$, we have $T(x) \subseteq V^* \subseteq V$. Hence, if the claim holds for clopens, there is an open $U \subseteq X$ with $x \in U$ and $T[U] \subseteq V^* \subseteq V$, as needed.

So assume $V$ is clopen. So $V \in F(Y)$ is (order-theoretically) compact. By Scott-continuity of $f$,
\begin{align*}
\bigcap \big\{  f(C) :  x \in C \in \Clp(X) \big\} 
=
f \big( \bigcap \{ C \in \Clp(X) : \{x\} \subseteq C \} \big)
=
f( \{ x \} )
=
T(x)
\subseteq V.   
\end{align*}
Since $V$ is compact and $\leq$ is reverse inclusion, there hence is $U \in \Clp (X)$ with $x \in U$ and $f(U) \geq V$.
Hence, for $x' \in U$, we have, by monotonicity of $f$, that $T(x') = f(\{x'\}) \geq f(U) \geq V$, so $T(x') \subseteq V$, as needed.

To show $\F(T) = f$, let $A \in \F(X)$ and show $\F(T)(A) = f(A)$. Since $f$ is co-atomic, we have
\begin{align*}
f (A) = \bigwedge f [ \CoAt(A) ] = \Cl \bigcup \{ T(x) : x \in A \} = \Cl T[A].
\end{align*}
Since we already know that $T$ is upper-hemicontinuous and closed-valued, the corollary to the closed graph theorem implies that $T[A]$ is closed, so $f (A) = \Cl T[A] = T[A]$, as needed.
\end{proof}

Finally, we prove theorem~\ref{thm: Sef equivalent to domain theoretic category}: The category $\Sysef$ is equivalent to the category $\dynAlg$ via the functor $\G$ sending $(X,T)$ to $(\F(X) , \F(T))$ and sending $(e,f) : (X,T) \to (Y,S)$ to $(\F(e), \F(f))$.

\begin{proof}[Proof of theorem~\ref{thm: Sef equivalent to domain theoretic category}]
First note that $\dynAlg$ is indeed a category:
Composition of $(\epsilon, \pi) : (A,\alpha) \to (B, \beta)$ and $(\epsilon', \pi') : (B,\beta) \to (C, \gamma)$ is given by $(\epsilon' \circ \epsilon , \pi \circ \pi')$, which again satisfies properties~\ref{itm: algebraic morphisms 1}--\ref{itm: algebraic morphisms 5}.
The identity morphism is $\id_{(A,\alpha)} = (\id_A , \id_A)$ and clearly satisfies properties~\ref{itm: algebraic morphisms 1}--\ref{itm: algebraic morphisms 5}.

Next we show that $\G$ is a well-defined functor.
If $(X,T)$ is in $\Sysef$, then, by proposition~\ref{prop: F categorical equivalence}, $\G(X,T) = (\F(X), \F(T))$ is in $\dynAlg$. So assume $(e,f) : (X,T) \to  (Y,S)$ is an ef-pair and show that $\G(e,f) = (\F(e) , \F(f) )$ has properties~\ref{itm: algebraic morphisms 1}--\ref{itm: algebraic morphisms 5}.
(\ref{itm: algebraic morphisms 1})~Recall that the co-atoms of $\F(Y)$ are singletons, and $\F(f) ( \{ y \} ) = f(y)$ is again a singleton since $f$ is a function (i.e., multifunction with singleton values).
(\ref{itm: algebraic morphisms 2})~Show
$\F(f) \circ \F(e) = \id_{\F(X)}$.
Indeed, for $A \in \F(X)$, we have, since $f \circ e (x) = \{ x \}$, that $\F(f) \circ \F(e) (A) = f \circ e [A] = A$.
(\ref{itm: algebraic morphisms 3})~Show 
$\F(e) \circ \F(f) \leq \id_{\F(Y)}$.
Indeed, for $B \in \F(Y)$, we have, since $y \in e \circ f (y)$, that
$\F(e) \circ \F(f)(B) = e \circ f (B) \supseteq B$.
(\ref{itm: algebraic morphisms 4})~Show $\F (f) \F(S) \F(e)  \leq \F(T)$. 
It suffices to show, for $x \in X$, that $\F (f) \F(S) \F(e) (x)  \supseteq \F(T) (x)$.
So let $x' \in X$ with $x' \in \F(T) (x) = T(x)$ and show $x' \in \F (f) \F(S) \F(e) (x) = f \circ S \circ e (x)$. Since $f : (Y,S) \to (X,T)$ is a factor, $x \xrightarrow{T} x'$ implies that there are $y,y' \in Y$ such that $y \xrightarrow{f} x$, $y' \xrightarrow{f} x'$, and $y \xrightarrow{S} y'$.
Since $y \in e \circ f(y) = e (x)$, we have $x \xrightarrow{e} y$.
So $x \xrightarrow{e} y \xrightarrow{S} y' \xrightarrow{f} x'$, so $x' \in f \circ S \circ e (x)$, as needed.
(\ref{itm: algebraic morphisms 5})~Show $\F(T) \F(f) \leq \F(f) \F(S)$.
We show, for $y \in Y$, that 
$f \circ S (y) \subseteq T \circ f (y)$.
If $x' \in f \circ S (y)$, there is $y' \in Y$ with $y \xrightarrow{S} y' \xrightarrow{f} x'$. Write $x := f(y)$. Since $f$ is a factor, there is $x'' \in X$ with  
$y' \xrightarrow{f} x''$ and $x \xrightarrow{T} x''$. Since $f$ is a function, $x'' = f(y') = x'$. Hence $y \xrightarrow{f} x \xrightarrow{T} x'$, so $x' \in T \circ f (y)$. 
Also, it is straightforward that $\G$ preserves composition.

It remains to show that $\G$ is essentially surjective, full, and faithful. 

Essentially surjective: Since, by proposition~\ref{prop: F categorical equivalence}, $\F$ is an equivalence, given $(A, \alpha)$ in $\dynAlg$, there is $X$ in $\czPolm$ with an isomorphism $\iota : \F(X) \to A$. Moreover, for the morphism 
\begin{align*}
\F(X) \xrightarrow{\iota} A \xrightarrow{\alpha} A \xrightarrow{\iota^{-1}} \F(X)
\end{align*}
there is $T : X \rightrightarrows X$ in $\czPolm$ with $ \F(T) = \iota^{-1} \circ \alpha \circ \iota$. Hence $\G ( X , T) = (\F(X) , \F(T) )$ is isomorphic to $(A,\alpha)$ via the isomorphism $(\iota, \iota^{-1})$ in $\dynAlg$, whose inverse is $(\iota^{-1}, \iota)$.

Faithful: If $(e,f) , (e',f') : (X,T) \to (Y,S)$ are ef-pairs with $(e,f) \neq (e',f')$, then $e \neq e'$ (and also $f \neq f'$, since one determines the other), so, since $\F$ is faithful, $\F(e) \neq \F(e')$, so $(\F(e) , \F(f) ) \neq (\F(e') , \F(f'))$.

Full: If $(\epsilon , \pi) : \G(X,T) \to \G(Y,S)$ is a morphism, then $\epsilon : \F(X) \to \F(Y)$ and $\pi : \F(Y) \to \F(X)$ are morphisms in $\bAlgc$, so, since $\F$ is full, there are upper-hemicontinuous and closed-valued multifunctions $e : X \rightrightarrows Y$ and $f : Y \rightrightarrows X$ such that $\F(e) = \epsilon$ and $\F(f) = \pi$.
It suffices to show that $(e,f): (X,T) \to (Y,S)$ is an ef-pair.

We first show that $f$ is a surjective function. It is a function, since, by~\ref{itm: algebraic morphisms 1}, $f(y) = \F(f)(\{y\}) = \pi ( \{ y \} )$ is a co-atom, i.e., a singleton. To show that $f$ is surjective, note that, given $x \in X$, the set $\epsilon(\{x\})$ is nonempty: otherwise we get the contradiction
\begin{align*}
\{ x \} = \pi \circ \epsilon ( \{ x\} ) = \pi ( \emptyset ) = \F(f)(\emptyset) = f[\emptyset] = \emptyset.
\end{align*}
So let $y \in \epsilon(\{x\})$. Hence $\{ y \} \subseteq \epsilon(\{x\})$ and, by monotonicity, $f(y) = \pi(\{y\}) \subseteq \pi \circ \epsilon ( \{ x \} ) = \{ x \}$. So the singleton $f(y)$ is identical to $\{x\}$. Hence, by viewing $f$ as a function rather than a multifunction with singleton values, we have $f(y) = x$, as needed.

Next we note that, for all $x \in X$, we have $e(x) = f^{-1}(x)$. Indeed, if $y \in e(x) = \epsilon( \{ x \})$, then, as above, $f(y) = x$. And if $f(y) = x$, or, when rather viewing $f$ as a multifunction, $f(y) = \{ x\}$, then $\{ y \} \subseteq \epsilon ( \pi ( \{ y \} ) = \epsilon ( f(y) ) = \epsilon ( \{ x \} ) = e(x)$, so $y \in e (x)$, as needed.  

Now we show that $f : Y \to X$ is a factor. Concerning equivariance~\ref{itm: nondet equivariance hom back}, if $x \xrightarrow{T} x'$, then 
\begin{align*}
x' 
\in 
T(x) 
= 
\F(T) (\{x\}) 
\subseteq 
\pi \circ \F(S) \circ \epsilon (\{x\})
= 
\pi \circ \F(S) (e(x))
=
f [ S [ f^{-1} (x) ] ]
\end{align*}
hence there is $y,y' \in Y$ with $x \xrightarrow{f^{-1}} y \xrightarrow{S} y' \xrightarrow{f} x'$. In other words, $y \xrightarrow{f} x$, $y' \xrightarrow{f} x'$, and $y \xrightarrow{S} y'$, as needed.
Concerning equivariance~\ref{itm: nondet equivariance factor}, if $y \xrightarrow{S} y'$ and $y \xrightarrow{f} x$, write $x' := f(y')$. Then 
\begin{align*}
x' 
\in 
f [ S(y) ]
=
\pi \circ \F(S) ( \{ y \} )
\subseteq
\F(T) \circ \pi ( \{ y \} )
=
T [ f (y) ],  
\end{align*}
so there is $x_0 \in X$ with $y \xrightarrow{f} x_0 \xrightarrow{T} x'$. Since $f$ is a function, $x_0 = x$. In sum, we have $x' \in X$ with $x \xrightarrow{T} x'$ and $y' \xrightarrow{f} x'$, as needed.

Finally, to show that $(e,f)$ is an ef-pair, we know, by proposition~\ref{prop: ef pairs determine each other}, that there is a unique embedding $\underline{f} : (X,T) \to (Y,S)$ such that $(\underline{f},f)$ is an ef-pair and, for all $x$, $\underline{f}(x) = f^{-1}(x) = e(x)$. 
\end{proof}

\subsection{Proofs from section~\ref{sec: algebroidal categories of systems}}
\label{app: proofs from section algebroidal categories of systems}

\begin{proof}[Proof of the claim in example~\ref{exm: deterministic systems not algeboidal wrt finite system}]
Assume there was a factor $f : (2^\omega , \sigma) \to (Y,S)$ with $S: Y \to Y$ a function on a finite discrete space $Y$ with at least two elements.
Then $\{ f^{-1} (y) : y \in Y \}$ is a clopen partition of $2^\omega$, so there is $k$ such that any two sequences $x,x' \in 2^\omega$ that agree on their first $k$ elements are in the same partition cell, i.e., $f(x) = f(x')$.

Our strategy is to find $y \neq y'$ in $Y$ and $j \geq k$ such that $S^j (y) \neq y'$.
Then we get the desired contradiction as follows: 
Since $f$ is surjective, there are $x,x' \in 2^\omega$ with $f(x) = y$ and $f(x') = y'$.
Consider the concatenation $x'' := (x \restriction j) x' \in 2^\omega$. Then, by agreement on the first $k$ elements, $f(x'') = f(x) = y$. By construction, $\sigma^j (x'') = x'$. So, by equivariance, $S^j(y) = S^j(f(x'')) = f ( \sigma^j (x'') ) = f ( x' ) = y'$, contradiction.

If $S$ is the identity function on $Y$, this is easy: since $Y$ has at least two elements, let $y \neq y'$ and pick $j := k$, so $S^j(y) = y \neq y'$.
So we assume that there is $y \in Y$, with $S(y) \neq y$. Consider the orbit $y_n := S^n(y)$. Since $Y$ is finite, there is a first time $n$ such that $y_n = y_i$ for some $i < n$. 
If $i = 0$, the orbit $(y_n)$ is of the form
$y_0 (y_1 \ldots y_{n-1} y_0)^\omega$ with $y_1 \neq y_0$.
If $i > 0$, the orbit $(y_n)$ is of the form
$y_0 \ldots y_{i-1} (y_i \ldots y_{n-1})^\omega$ with $y_i \neq y_0$ (otherwise $y_0 = y_i = y_n$ with $0 < i < n$, so $n$ would not be minimal).
In either case, the orbit $(y_n)$ is of the form $\rho_0 (\rho_1)^\omega$ with $\rho_0$ and $\rho_1$ being finite paths in $(Y,S)$. The first state of $\rho_0$ is $y_0 = y$. Let $y'$ be the first state of $\rho_1$ (either $y_1$ or $y_i$). Then, for $j := k |\rho_1| \geq k$, we have $S^j(y') = y' \neq y$, as needed.
\end{proof}

\begin{proof}[Proof of lemma~\ref{lem: commuting trianlge lemma}]
Let $f : (X,T) \to (Y,S)$ and $g : (X,T) \to (Z,R)$ be factors of systems in $\allSys$ and let $h : Y \to Z$ be a function with $h \circ f = g$. We show that $h$ is a factor.

First note that, for $C \subseteq Z$, we have $h^{-1}(C) = f[g^{-1}(C)]$:
Indeed, for $y \in Y$, if $y \in h^{-1}(C)$, then, by surjectivity, let $x \in X$ with $f(x) = y$, so $g(x) = h ( f(x) ) = h(y) \in C$, so $y \in  f[g^{-1}(C)]$. Conversely, if $y = f(x)$ for some $x \in g^{-1}(C)$, then $h(y) = h ( f(x) ) = g(x) \in C$, so $y \in h^{-1}(C)$.

Now, $h$ is continuous: If $C \subseteq Z$ is closed, then $h^{-1}(C) = f[g^{-1}(C)]$ is an image of a closed set and hence closed, since $h$ is closed qua continuous function between compact Hausdorff spaces.
Also, $h$ is surjective: if $z \in Z$, let, by surjectivity, $x \in X$ with $g(x) = z$, then, for $y := f(x) \in Y$, we have $h(y) = h ( f(x) ) = g(x) = z$. 

Concerning equivariance~\ref{itm: nondet equivariance hom back}, if $z,z' \in Z$ with $z \xrightarrow{R} z'$, there is $x,x' \in X$ with $g(x) = z$, $g(x') = z'$, and $x \xrightarrow{T} x'$, so, for $y := f(x)$ and $y' := f(x')$ in $Y$ we have $y \xrightarrow{S} y'$ and $h (y) = h (f(x)) = g(x) = z$ and, similarly, $h(y') = z'$, as needed.

Concerning equivariance~\ref{itm: nondet equivariance factor}, if $y,y' \in Y$ and $z \in Z$ with $h(y) = z$ and $y \xrightarrow{S} y'$, consider $z' := h(y')$. We need to show $z \xrightarrow{R} z'$. Since $f$ is a factor, there is $x,x' \in X$ with $f(x) = y$, $f(x') = y'$, and $x \xrightarrow{T} x'$. Hence $g(x) \xrightarrow{R} g(x')$, and we have $g(x) = h ( f(x) ) = h (y)$ and, similarly, $g(x') = h (y')$, so $z = h(y) \xrightarrow{R} h(y') = z'$.  
\end{proof}

\begin{proof}[Proof of lemma~\ref{lem: refine partitions}]
Recall that $\big( (X_i,T_i), (e_{i i+1}, f_{i i+1} ) \big)_{i \in \omega}$ is an $\omega$-chain in $\subSysef$ and $\big( (X,T) , (e_i,f_i) \big)$ is a colimit.
Assume $\mathcal{C} = \{ C_1 , \ldots , C_n \}$ is a finite partition of $X$ consisting of closed sets. 
Assume for contradiction, that there is no $i \in \omega$ such that $X / f_i$ refines $\mathcal{C}$.
Then, for every $i \in \omega$, there is $x,x' \in X$ with $f_i(x) = f_i(x')$ but $x \in C_k$ and $x' \in C_l$ for $k \neq l$,
so the following subset of $X \times X$ is nonempty:
\begin{align*}
F_i 
:=
\big\{
(x,x') \in X \times X : f_i (x) = f_i(x')
\big\}
\cup
\bigcup_{k \neq l \text{ in } \{ 1, \ldots , n \}}
C_k \times C_l.
\end{align*}
This set is closed qua finite union of closed sets (the equivalence relation $\equiv$ of a continuous function $f$ into a Hausdorff space is closed).
Moreover, if $i \leq j$ in $\omega$, then $F_i \supseteq F_j$ (if $f_j(x) = f_j (x')$, then $f_i(x) = f_{ij} \circ f_j(x) = f_{ij} \circ f_j (x') = f_i (x')$). Since $X \times X$ is compact, there is 
$(x,x') \in \bigcap_{i \in \omega} F_i$. So there is $k \neq l$ with $x \in C_k$ and $x' \in C_l$, so $x \neq x'$. However, for all $i \in \omega$, we also have $f_i(x) = f_i(x')$, which implies $x = x'$ (proposition~\ref{prop: every chain has colimit}).
\end{proof}

\subsection{Proofs from section~\ref{sec: universality for deterministic systems}}
\label{app: proofs from section universality for deterministic systems}

\begin{proof}[Proof of proposition~\ref{prop: path system}]
Ad~\ref{lem: path system 1}.
It is straightforward to show that $\overline{X}$ is a closed subset of $\prod_\omega X$.
So $\overline{X}$ again is a zero-dimensional compact Polish space.
It also is immediate that $\sigma : \overline{X} \to \overline{X}$ is a continuous function. 
Hence, per definition~\ref{def: category of systems}, $\Path (X,T) = ( \overline{X} , \sigma )$ is in $\detSysef$ iff $\overline{X}$ is nonempty, as needed.

Ad~\ref{lem: path system 2}.
First, let us assume that $(X,T)$ is total.
Write $p := \pi_0 : \overline{X} \to X$ for the projection to the first component, and show that $p$ is a factor.
Indeed, $p$ is a continuous: if $U \subseteq X$ is open, then $p^{-1} (U) = \pi_0^{-1} (U) \cap \overline{X}$, which is open.
And $p$ is surjective: If $x \in X$, then, since $T$ is total, we can continue $x$ to some path $\overline{x} \in \overline{X}$, so $p (\overline{x}) = x$. 
Concerning equivariance~\ref{itm: nondet equivariance hom back}, 
if $x \xrightarrow{T} x'$, continue this, since $T$ is total, into a path $\overline{x}$:
\begin{align*}
x \xrightarrow{T} x' \xrightarrow{T} x_2 \xrightarrow{T} x_3 \xrightarrow{T}  \ldots,  
\end{align*}
then 
$\overline{x} \xrightarrow{p} x$,
$\sigma(\overline{x}) \xrightarrow{p} x'$, and
$\overline{x} \xrightarrow{\sigma} \sigma(\overline{x})$.
Concerning equivariance~\ref{itm: nondet equivariance factor}, if
$\overline{x} \xrightarrow{\sigma} \sigma(\overline{x})$ 
and
$\overline{x} \xrightarrow{p} x$,
then, for $x' := p ( \sigma(\overline{x}) )$, we have, since $\overline{x}$ is a path, that $x \xrightarrow{T} x'$, as needed.

Conversely, assume that $p : \overline{X} \to X$ is a factor, and show that $(X,T)$ is total. Given $x \in X$, there is, since $p$ is surjective, some path $(x_n) \in \overline{X}$ with $p\big( (x_n) \big) = x$. Hence $x = x_0 \xrightarrow{T} x_1$, i.e., $x$ has a successor state, as needed.  
\end{proof}

\subsection{Proofs from section~\ref{sec: algebroidal categories of deterministic systems}}
\label{app: proofs from section algeboidal categories of deterministic systems}

\begin{proof}[Proof of lemma~\ref{lem: Curtis-Lyndon-Hedlund Theorem}]
We can assume that $X$ and $Y$ are nonempty: Otherwise, if $X = \emptyset$, then $f = \emptyset$ and we can choose $N := 0$ and $\varphi = \emptyset : B_0(X) \to B$. If $Y = \emptyset$, then, since $f$ is a function, also $X = \emptyset$, and the claim, as just seen, holds.

Since $f$ is continuous and $\{ C_b(Y) : b \in B \}$ is a clopen partition of $Y$, also 
$\{ f^{-1}( C_b(Y) ) : b \in B \}$ 
is a clopen partition of $X$. (Since $X$ and $Y$ are nonempty, these partitions contain at least one cell.) So there is $N$, such that this partition is refined by the partition $\{ C_w (X) : w \in A^N\}$.
So we can define $\varphi : B_N(X) \to B$ by mapping $w \in B_N(X)$ to the unique $b \in B$ such that $C_w (X) \subseteq f^{-1}( C_b(Y) )$.
Finally, let $x \in X$ and $k \geq 0$ and show 
$f(x) (k) = \varphi ( x(k), \ldots , x(k + N-1) )$.
By equivariance, 
\begin{align*}
b := f(x) (k)
=
\sigma^k ( f(x) ) (0)
=
f ( \sigma^k (x) ) (0)
\in 
B.
\end{align*}
So 
$x' := \sigma^k (x) \in f^{-1}( C_b(Y) )$.
Write $w := (x'(0), \ldots , x'(N-1)) = (x(k), \ldots , x'(k + N-1)) \in B_N(X)$.
We have
$x' \in C_w (X) \cap  f^{-1}( C_b(Y) )$, hence, qua refinement, $C_w (X) \subseteq f^{-1}( C_b(Y) )$, 
so 
$\varphi(w) = b$.
Hence
$f(x) (k) = b = \varphi(w) = \varphi ( x(k), \ldots , x(k + N-1) )$, as needed.

For the `moreover' part, note that, for $n \geq N$, the partition $\{ C_w (X) : w \in A^n\}$ still refines $\{ C_w (X) : w \in A^N\}$, which was all that is needed for the above proof.
\end{proof}

\begin{proof}[Proof of lemma~\ref{lem: pro category closed under limits}]
To fix notation, let $(A_i, f_{i i+1})_{i \in \omega}$ be an $\omega$-chain with colimit $(A,f_i)_{i \in \omega}$. 
For each $i \in \omega$, let $(A_i^j, f_i^{j j+1})_{j \in \omega}$ be an $\omega$-chain with colimit $(A_i, f_i^j)_{j \in \omega}$ such that each $A_i^j$ is category-theoretically finite in $\mathsf{C}$.
For each $i \in \omega$, we construct a strictly increasing sequence $(k_i^j)_{j \in \omega}$ and morphisms 
$g_{i i+1}^j : A_i^{k_i^j} \to A_{i+1}^{k_{i+1}^j}$ as follows (see figure~\ref{fig: finding a diagonal omega chain}).

For $i = 0$, we set $k_i^j := j$. For $i = 1$, we build $k_1^0 < k_1^1 < k_1^2 < \ldots$ inductively. Set $k_1^{-1} := 0$ and assume $k_1^{j-1}$ is defined, Note that we have the $\omega$-chain 
$(A_1^j, f_1^{j j+1})_{j \in \omega}$ 
with colimit $A_1$, and we have the morphism 
$f_{01} \circ f_0^{k_0^j} : A_0^{k_0^j} \to A_1$. 
Since $A_0^{k_0^j}$ is category-theoretically finite, there hence is $k_1^j$ and a morphism 
$g_{0 1}^j : A_0^{k_0^j} \to A_1^{k_1^j}$
such that
$f_{01} \circ f_0^{k_0^j} = f_1^{k_1^j} \circ g_{0 1}^j$ and
\begin{itemize}
\item[($*$)]
For each $l \geq k_1^j$, there is a unique morphism
$u : A_0^{k_0^j} \to A_1^l$
such that
$f_{01} \circ f_0^{k_0^j} = f_1^l \circ u$.
\end{itemize}
Now we continue similarly for $i = 2, 3 , \ldots$, and we get the diagram in figure~\ref{fig: finding a diagonal omega chain}.

\begin{figure}
\begin{equation*}
\begin{tikzcd}[column sep = large]
A_0
\arrow[r, "f_{01}"]
&
A_1
\arrow[r, "f_{12}"]
&
A_2
\arrow[r]
&
\cdots
&
A
\\
\vdots
&
\vdots
&
\vdots
&
~
&
\\
A_0^{k_0^2}
\arrow[u]
\arrow[r, "g_{01}^2"]
&
A_1^{k_1^2}
\arrow[u]
\arrow[r, "g_{12}^2"]
&
A_2^{k_2^2}
\arrow[u]
\arrow[r]
\arrow[ur, dashed,  "h_{23}"]
&
\cdots
&
\\
A_0^{k_0^1}
\arrow[u, "f_0^{k_0^1 k_0^2}"]
\arrow[r, "g_{01}^1"]
&
A_1^{k_1^1}
\arrow[u, "f_1^{k_1^1 k_1^2}"]
\arrow[r, "g_{12}^1"]
\arrow[ur, dashed, "h_{12}"]
&
A_2^{k_2^1}
\arrow[u, "f_2^{k_2^1 k_2^2}"]
\arrow[r]
&
\cdots
&
\\
A_0^{k_0^0}
\arrow[u, "f_0^{k_0^0 k_0^1}"]
\arrow[r, "g_{01}^0"]
\arrow[ur, dashed, "h_{01}"]
&
A_1^{k_1^0}
\arrow[u, "f_1^{k_1^0 k_1^1}"]
\arrow[r, "g_{12}^0"]
&
A_2^{k_2^0}
\arrow[u, "f_2^{k_2^0 k_2^1}"]
\arrow[r]
&
\cdots
&
\\
\end{tikzcd}
\end{equation*}
\caption{Finding a diagonal $\omega$-chain.}
\label{fig: finding a diagonal omega chain}
\end{figure}

It is straightforward to show that this diagram commutes.
So we can define the `diagonal' $\omega$-chain
$(A_m^{k_m^m} , h_{m m+1})_{m \in \omega}$,
where
\begin{align*}
h_{m m+1} 
:= 
f_{m+1}^{k_{m+1}^m k_{m+1}^{m+1} } \circ g_{m m+1}^m
=
g_{m m+1}^{m+1} \circ f_{m}^{k_m^m k_m^{m+1}}.
\end{align*}
Now it is a bit tedious but straightforward to show that $A$ also is a colimit of this $\omega$-chain, via the natural choice for the projections: $h_m := f_m \circ f_m^{k_m^m}$.
\end{proof}

\printbibliography

\end{document}